\documentclass[10pt,letterpaper]{amsart} 

\usepackage{color}
\usepackage{multicol}
\usepackage{empheq}
\usepackage[letterpaper, margin=1in]{geometry}
\usepackage{subfig}

\usepackage{listings}
\usepackage{amsmath, amsthm, amssymb, wasysym, verbatim, bbm, color, graphics}
\usepackage{url}
\usepackage{mathtools}
\usepackage[colorlinks,linkcolor=blue]{hyperref}
\usepackage{romannum}
\usepackage{xfrac}
\usepackage{float}
\usepackage{amsmath,bm}
\usepackage[title]{appendix}
\usepackage{todonotes}

\newcommand{\R}{\mathbb{R}}

\newcommand{\N}{\mathbb{N}}

\DeclareMathOperator{\tr}{\text{tr}}

\newcommand\norm[1]{\left\lVert#1\right\rVert}
\newcommand{\Grad}{\nabla}
\newcommand{\Div}{\operatorname{div}}
\newcommand{\dom}{\Omega}

\newcommand{\CTR}{C_{\mathcal{T}_R}}

\renewcommand{\div}{\operatorname{div}}
\newcommand{\T}{\mathcal T}
\renewcommand{\P}{\mathcal P}
\DeclareMathOperator*{\esssup}{ess\,sup}

\newtheorem{lemma}{Lemma}[section]

\newtheorem{corollary}[lemma]{Corollary}
\newtheorem{theorem}[lemma]{Theorem}
\newtheorem{remark}[lemma]{Remark}

\newtheorem*{maintheorem*}{Main Theorem}
\theoremstyle{definition}{\newtheorem{definition}[lemma]{Definition}}

\theoremstyle{note}{
\newtheorem*{claim*}{Claim}}

\allowdisplaybreaks

\numberwithin{equation}{section}

\title[Maxwell Q-tensor]{A convergent finite element scheme for the Q-tensor model of liquid crystals subjected to an electric field}
\date{\today}

\author[M. Hirsch]{Max Hirsch}
\address[Max Hirsch]{\newline Department of Mathematics \newline University of California, Berkeley \newline Berkeley, CA 94720, USA.}
\email[]{mhirsch@berkeley.edu}
\author[F. Weber]{Franziska Weber}
\address[Franziska Weber]{\newline Department of Mathematics \newline University of California, Berkeley \newline Berkeley, CA 94720, USA.}
\email[]{fweber@berkeley.edu}
\thanks{F.W. and M.H. were partially supported by NSF DMS 2042454. This material is based upon work supported by the National Science Foundation Graduate Research Fellowship Program under Grant No.\ DGE 2146752. Any opinions, findings, and conclusions or recommendations expressed in this material are those of the authors and do not necessarily reflect the views of the National Science Foundation.}

\begin{document}
 \pagenumbering{arabic}
\maketitle
\begin{abstract}
We study the Landau-de Gennes Q-tensor model of liquid crystals subjected to an electric field and develop a fully discrete numerical scheme for its solution. The scheme uses a convex splitting of the bulk potential, and we introduce a truncation operator for the Q-tensors to ensure well-posedness of the problem. We prove the stability and well-posedness of the scheme.  Finally, making a restriction on the admissible parameters of the scheme, we show that up to a subsequence, solutions to the fully discrete scheme converge to weak solutions of the Q-tensor model as the time step and mesh are refined. We then present numerical results computed by the numerical scheme, among which we show that it is possible to simulate the Fr\'{e}edericksz transition with this scheme. 
\end{abstract}
\section{Introduction}
Liquid crystals are a state of matter with properties intermediate to those of solids and liquids. They flow like liquids, but they also possess properties of solids. Namely, their molecules have an orientation as in crystals. These materials generally consist of elongated molecules that can be pictured as rods which align along a common direction due to intermolecular forces~\cite{Virga1995,SonnetVirga2001,Stewart2008}. This microscopic structure affects the mechanical response of the material to stress and strain at the macroscopic level. This makes liquid crystals useful for many engineering applications. For example, there are liquid crystals that react to electric fields and polarize the light going through them. This is taken advantage of in liquid crystal displays (LCDs), smart glasses, and other technologies~\cite{Chang2022,Castellano2005,Loo2006}.
The goal of this work is to introduce and analyze a numerical scheme for a model of liquid crystal subjected to an electric field. The convergence proof will also demonstrate existence of weak solutions of the underlying system of partial differential equations (PDEs). The model that we consider here is based on the Landau and de Gennes Q-tensor theory of liquid crystal dynamics~\cite{Virga1995,deGennes1995}. 

In the Q-tensor model, the orientation of liquid crystals is described by a symmetric, trace-free $d\times d$ matrix $Q$ which is assumed to minimize the Landau-de Gennes free energy
\[
    E_{LG}(Q) = \int_\Omega \mathcal{F}_B(Q) + \mathcal{F}_E(Q) + \mathcal{F}_e(Q),
\]
in equilibrium. Here $\Omega \subset \R^d$ with $d=2,3$ is the spatial domain in which the liquid crystal molecules lie, $\mathcal{F}_e$ is the electrostatic energy density which we discuss below, $\mathcal{F}_B$ is a bulk potential, and $\mathcal{F}_E$ is the elastic energy given by
\[
    \mathcal{F}_B(Q) = \frac{a}{2}\operatorname{tr}(Q^2) - \frac{b}{3}\operatorname{tr}(Q^3) + \frac{c}{4}(\operatorname{tr}(Q^2))^2 + A_0,\quad \mathcal{F}_E(Q) = \frac{L_1}{2}|\Grad Q|^2 + \frac{L_2}{2}|\Div Q|^2 +\frac{L_3}{2}\sum_{i,j,k=1}^d \partial_i Q_{jk}\partial_k Q_{ji},
\]
where $a,b,c,A_0,L_1,L_2,L_3$ are constants with $c,L_1>0$; $L_2+L_3\geq 0$. Here $(\Div Q)_i =\sum_{j=1}^d \partial_j Q_{ij}$ and $|\Grad Q|^2 = \sum_{i,j,k=1}^d (\partial_k Q_{ij})^2$. We use $|\cdot |$ to denote the Euclidean vector norm, the Frobenius norm for matrices, and corresponding generalizations for higher order tensors. Having $c>0$ ensures that the bulk potential is bounded from below~\cite{Zhao2016}, and $A_0$ ensures $\mathcal{F}_B(Q) \ge 0$. We further assume that $\Omega$ is a bounded, simply connected domain with Lipschitz boundary. For the clarity of exposition, we will also assume the one-constant approximation~\cite{Virga1995}, i.e., that $L_1:=L>0$ and $L_2=L_3=0$. The results of this paper can be extended to nonzero $L_2$ and $L_3$ in a straightforward manner, but this assumption simplifies our presentation.

Now non-equilibrium situations are described via the gradient flow
\begin{equation}
\label{eq:gradflow}
    Q_t =  L\Delta Q - \frac{\partial\mathcal{F}_B(Q)}{\partial Q} - \frac{\partial\mathcal{F}_e(Q)}{\partial Q}.
\end{equation}
Here, we have
\[
    \frac{\partial\mathcal{F}_B(Q)}{\partial Q} = aQ - b\left(Q^2 - \frac1d\operatorname{tr}(Q^2)I\right) + c\operatorname{tr}(Q^2)Q.
\]
The gradient flow above will be coupled to another equation related to the electric potential. We turn now to its derivation and the derivation of $\mathcal{F}_e$.

\subsection{Electrostatic Energy Density Derivation}
Liquid crystals interact with externally applied electric fields or self-induce an internal electric field due to dielectric and spontaneous polarization effects. We let $D$ be the electric displacement field and $E$ be the electric field. One may use Maxwell's equations to find the electric field in $\Omega$, which in this case with no free charges are given by
\begin{equation}
\label{eq:Maxwell}
    \begin{split}
        \operatorname{div}D &= 0,\\
        \operatorname{curl}E &= 0.
    \end{split}
\end{equation}
As the time scale of the electromagnetic waves is much faster than the one of the liquid crystal molecule motion, it makes sense to consider the stationary Maxwell equations. The equations for the magnetic field can be included also, however, they decouple if the liquid crystal material is not susceptible to magnetic fields and for this reason, we will not consider them here. 
Since $\Omega$ is simply connected, it follows that we may write $E = \nabla u$, where $u$ is a scalar potential. In the setting of nematic liquid crystals, we have
\begin{equation*}
    D = \varepsilon_0 \varepsilon E + P_s,
\end{equation*}
where $\varepsilon = \Delta\varepsilon^\ast Q + \overline{\varepsilon}I$ and
\begin{equation*}
    P_s = \overline{e}\operatorname{div}Q
\end{equation*}
is the spontaneous polarization to the leading order~\cite{mottram2014}. Defining the constants $\varepsilon_1 = \varepsilon_0\overline{\varepsilon}$, $\varepsilon_2 = \varepsilon_0\Delta\varepsilon^\ast$, and $\varepsilon_3 = \overline{e}$, we can rewrite $D$ as 
\begin{equation*}
    D = \varepsilon_1\nabla u + \varepsilon_2 Q\nabla u + \varepsilon_3\operatorname{div}Q,
\end{equation*}
so with (\ref{eq:Maxwell}) we obtain
\begin{equation}\label{eq:ellipticEfield}
    \operatorname{div}(\varepsilon_1\nabla u + \varepsilon_2 Q\nabla u + \varepsilon_3\operatorname{div}Q) = 0.
\end{equation}
Without loss of generality, we assume that $\varepsilon_1>0$ in the following.
Then the electrostatic energy density is given by
\begin{equation*}
\begin{split}
    \mathcal{F}_e(Q) 
        &= -\int (\varepsilon_0 \varepsilon E + P_s)\cdot \, d{E} = -\frac12 \varepsilon_0 (\varepsilon E)\cdot E - P_s\cdot E\\
        &= -\frac12 \varepsilon_0(\Delta\varepsilon^\ast Q + \overline{\varepsilon}I)\nabla u \cdot \nabla u - \overline{e}\operatorname{div}(Q)\cdot\nabla u\\
        &= -\frac{\varepsilon_1}{2} \lvert \nabla u\rvert^2 -\frac{\varepsilon_2}{2} \nabla u^\top Q\nabla u - \varepsilon_3\operatorname{div}(Q)\cdot\nabla u.
\end{split}
\end{equation*}
Then~\eqref{eq:gradflow} becomes
\begin{equation*}
  Q_t =  L\Delta Q - \frac{\partial \mathcal{F}_B(Q)}{\partial Q} + \frac{\varepsilon_2}{2}\left[\nabla u\nabla u^\top - \frac1d|\Grad u|^2I\right]- \varepsilon_3\left[\nabla^2u - \frac1d\Delta uI\right],
\end{equation*}
and the coupled system for the Q-tensor and the electric field with~\eqref{eq:ellipticEfield} becomes
\begin{subequations}\label{eq:fulluntruncatedpde}
	\begin{align}
	Q_t &=  L\Delta Q - \frac{\partial \mathcal{F}_B(Q)}{\partial Q} + \frac{\varepsilon_2}{2}\left[\nabla u\nabla u^\top - \frac1d|\Grad u|^2 I\right]- \varepsilon_3\left[\nabla^2u - \frac1d\Delta uI\right]  ,\\ 0&=\operatorname{div}(\varepsilon_1\nabla u + \varepsilon_2 Q\nabla u + \varepsilon_3\operatorname{div}Q),\label{seq:elliptic1}
	\end{align}
\end{subequations}
with suitable boundary conditions. The elliptic equation in~\eqref{seq:elliptic1} is only coercive (and thus solvable given $Q$) if 
\begin{equation}\label{eq:epsbound}
|\varepsilon_1|> |\varepsilon_2| |Q|,
\end{equation} 
for almost every $(t,x)\in [0,T]\times\Omega$. From a physical perspective, the eigenvalues of $Q$ should be in the range $[-1/d,2/d]$ (see, for example, the derivation of the $Q$-tensor model in the book by Virga~\cite{Virga1995}). If this is indeed the case, then under the condition $d |\varepsilon_1|>2|\varepsilon_2|$, the elliptic equation~\eqref{seq:elliptic1} is coercive and solvable.  
For the pure gradient flow (without electric field), it has been shown for smooth solutions that if the eigenvalues of the Q-tensor are initially within a certain range, this is preserved for the time evolution~\cite{Wu2019,Contreras2019}. Those results also apply in the more general case of the Beris-Edwards system (liquid crystal Q-tensor coupled to fluid flow) if solutions are smooth. However, it is unknown whether similar estimates can be proved for the system~\eqref{eq:fulluntruncatedpde}, and in fact, it is not even clear that $Q$ would remain bounded. We emphasize that even though in the pure gradient flow case there have been numerical schemes that maintain $L^\infty$ bounds on the solution via stabilization terms \cite{boundpreserving2017,boundpreserving2025}, it is still not clear that the equation with electric field \eqref{eq:fulluntruncatedpde} has solutions which remain bounded. Our numerical experiments in Section~\ref{sec:num} indicate that $Q$ may not stay within the range $[-1/d,2/d]$ in some cases and may grow without bound.
Thus it is unclear whether~\eqref{eq:fulluntruncatedpde} is a practical model for liquid crystal dynamics under the influence of an electric field. Therefore, to ensure well-posedness of our problem, we introduce a bounded, twice continuously differentiable mapping $\T_R:\R^{d\times d}\to\R^{d\times d}$ that satisfies
\begin{equation}
	\label{eq:conditionsonTr}
	\mathcal{T}_R(Q)_{ij}\in \left[-\frac{R}{d},\frac{R}{d}\right],\quad \left|\frac{\partial\mathcal{T}_R(Q)_{ij}}{\partial Q}\right|\leq \CTR ,\quad \left|\frac{\partial^2\mathcal{T}_R(Q)_{ij}}{\partial Q^2}\right|\leq \CTR\quad i,j=1,\dots, d
\end{equation}
for some $R>0$ to be chosen, and where $0<\CTR<\infty$. What we have in mind, is a function that increases approximately linear in each component between $[-R/d+\epsilon,R/d-\epsilon]$ and smoothly truncates values outside the physical range. For example, we could choose $\mathcal{T}_R(Q)$ to be a smooth, componentwise approximation of a `Heaviside function':
\begin{equation*}
	\mathcal{T}_R(Q)_{ij}= \frac{2R}{\pi d}\arctan\left(\frac{d Q_{ij}}{R}\right).
\end{equation*}
Many other choices of functions are possible here, c.f.,~\eqref{eq:TRnumerics} -- \eqref{eq:TRnumerics2}, which is used in our numerical experiments section.
Now we obtain the modified elliptic equation
\begin{equation}
    \operatorname{div}(\varepsilon_1\nabla u + \varepsilon_2 \T_R(Q)\nabla u + \varepsilon_3\operatorname{div}Q) = 0
\end{equation}
and the modified electrostatic energy density
\begin{equation}
    \tilde{\mathcal{F}}_e(Q) = -\frac{\varepsilon_1}{2} \lvert \nabla u\rvert^2 -\frac{\varepsilon_2}{2} \nabla u^\top \T_R(Q)\nabla u - \varepsilon_3\operatorname{div}(Q)\cdot\nabla u.
\end{equation}
 Note that it follows from our assumption~\eqref{eq:conditionsonTr} that $|\T_R(Q)| \le R$. This way, we can ensure that the elliptic equation remains solvable under constraints on $\varepsilon_1$ and $\varepsilon_2$ that depend on $R$. 

Now to determine the evolution PDE for the Q-tensor, we compute the variational derivative of this modified electrostatic energy,
\[
    \int_\Omega \tilde{\mathcal{F}}_e(Q)\, dx.
\]
Letting $\phi:\Omega\times[0,T) \to \R^{d\times d}$ be smooth and compactly supported in $\Omega$, we compute
\begin{align*}
    &\left.\frac{d}{d\delta}\right\lvert_{\delta=0} \int_\Omega \tilde{\mathcal{F}}_e(Q+\delta\phi)\, dx\\
        &= \left.\frac{d}{d\delta}\right\lvert_{\delta=0} \int_\Omega \left(-\frac{\varepsilon_1}{2} \lvert \nabla u\rvert^2 -\frac{\varepsilon_2}{2} \nabla u^\top \T_R(Q+\delta\phi)\nabla u - \varepsilon_3\operatorname{div}(Q+\delta\phi)\cdot\nabla u\right)\, dx\\
        &= \int_\Omega \left(-\frac{\varepsilon_2}{2} \nabla u^\top\left(\frac{\partial\T_R(Q)}{\partial Q}\odot\phi\right)\nabla u - \varepsilon_3\operatorname{div}(\phi)\cdot\nabla u\right)\, dx\\
        &= \int_\Omega \left(-\frac{\varepsilon_2}{2} \left(\frac{\partial\T_R(Q)}{\partial Q}\odot \nabla u\nabla u^\top\right):\phi + \varepsilon_3\nabla^2 u:\phi\right)\, dx.
\end{align*}
Here we denoted by $\odot$ the Hadamard product $(A\odot B)_{ij} = A_{ij} B_{ij}$.
We will also denote $\P(Q) := \frac{\partial\T_R(Q)}{\partial Q}$. Note that $\P$ depends on the parameter $R$ which is omitted in the notation for simplicity. 
Thus, the complete system of equations, along with boundary conditions, is given by
\begin{subequations}
\label{eq:system_strong_formulation}
\begin{align}
    Q_t = L\Delta Q - \frac{\partial \mathcal{F}_B(Q)}{\partial Q} + \frac{\varepsilon_2}{2}\left[\P(Q)\odot\nabla u\nabla u^\top - \frac1d\operatorname{tr}\left(\P(Q)\odot\nabla u\nabla u^\top\right)I\right]- \varepsilon_3\left[\nabla^2u - \frac1d\Delta uI\right],\\
    \div(\varepsilon_1\nabla u + \varepsilon_2 \T_R(Q)\nabla u + \varepsilon_3\div(Q)) = 0,\label{seq:elliptic}\\
    u = g \text{ on } \partial\Omega,\\
    Q = q \text{ on } \partial \Omega.
\end{align}
\end{subequations}
Here $g=g(t,x)$ and $q=q(x)$ are the Dirichlet boundary conditions for $u$ and $Q$ respectively.
The elliptic equation~\eqref{seq:elliptic} can be derived as the variational derivative of the electrostatic energy with respect to $u$. 

To the best of our knowledge, existence and uniqueness of solutions for this system is unknown. Due to the nonlinearities, smooth solutions may not be expected in general. We therefore proceed to define weak solutions: Fix $T>0$ and assume that $g$ admits an extension $\tilde g: [0,T]\times\Omega\to\R$ with $\tilde g \in L^\infty([0,T];H^1(\Omega))\cap W^{1,2}([0,T];H^1(\Omega))$. Further assume that $q$ admits an extension $\tilde q:\Omega\to\R^{d\times d}$ with $\tilde q\in (H^1(\Omega))^{d\times d}$ which is trace-free and symmetric. Also, assume that $Q_0:\Omega\to\R^{d\times d} \in (H^1(\Omega))^{d\times d}$ takes values in the symmetric, trace-free $d\times d$ matrices. Denote 
$$
(A)_S:=\frac{A+A^\top}{2}
$$
the symmetric part of a matrix $A\in \R^{d\times d}$.
Then we define weak solutions as follows:
\begin{definition}\label{def:weaksol}
A pair $(Q,u)$ with $Q:[0,T]\times\Omega\to\R^{d\times d}$ trace-free and symmetric and $u:[0,T]\times\Omega\to\R$ is called a weak solution of \eqref{eq:system_strong_formulation} if
\[
    \tilde Q\in L^\infty(0,T;H_0^1(\Omega)),\quad \tilde Q_t\in L^2([0,T]\times\Omega),\quad \tilde u \in L^\infty(0,T;H_0^1(\Omega)),
\]
with $Q = \tilde Q + \tilde q$ and $u = \tilde u + \tilde g$, and
\begin{subequations}
\label{eq:weak_formulation}
\begin{equation}
\begin{split}
    &\int_0^T \int_\Omega Q:\Phi_t\, dx\, dt + \int_\Omega Q_0(x):\Phi(0,x)\, dx\\ 
    &= L\int_0^T\int_\Omega \nabla Q : \nabla \Phi\, dx\, dt + \int_0^T\int_\Omega \frac{\partial\mathcal{F}_B(Q)}{\partial Q}:\Phi\, dx\, dt\\ 
    &\hspace{10ex}- \frac{\varepsilon_2}{2}\int_0^T\int_\Omega \left(\P(Q)\odot\nabla u\nabla u^\top - \frac1d\operatorname{tr}(\P(Q)\odot\nabla u\nabla u^\top)I\right):\Phi\, dx\, dt,\\
    &\hspace{10ex}-\varepsilon_3 \int_0^T\int_\Omega \left(\nabla u\cdot\div((\Phi)_S) - \frac1d\nabla u\cdot\nabla\tr\Phi\right)\, dx\, dt\\
\end{split}
\end{equation}
\begin{equation}
\label{eq:div_weak_form}
    \int_0^T\int_\Omega (\varepsilon_1\nabla u + \varepsilon_2 \T_R(Q)\nabla u + \varepsilon_3\div Q)\cdot\nabla \psi\, dx\, dt = 0
\end{equation}
\end{subequations}
for all smooth $\Phi = (\Phi_{ij})_{ij=1}^{d} : [0,T]\times\Omega\to\mathbb{R}^{d\times d}$ and $\psi: [0, T]\times\Omega\to\mathbb{R}$ which are compactly supported within $[0,T)\times\Omega$.

\end{definition}

In this work, we will construct a fully discrete numerical scheme based on a finite element discretization in space for system~\eqref{eq:system_strong_formulation}. We will prove uniform stability with respect to the discretization parameters, and in the case $\varepsilon_3=0$ (effect of polarization is zero), convergence to a weak solution as in Definition~\ref{def:weaksol} as the discretization parameters vanish. Thus, the convergence proof also implies the existence of a weak solution to~\eqref{eq:system_strong_formulation} when $\varepsilon_3=0$. The convergence in the case with nonzero polarization introduces additional mathematical challenges and we therefore leave it for future research. 

\subsection{Related works}
There have been a couple of numerical works investigating the dynamics of liquid crystals under the influence of external fields. Most rigorous results with error and/or convergence analyses consider equilibrium situations or explicit solutions in simplified situations, e.g.,~\cite{Nochetto2018,Berezin1973,Davis1998,Helfrich1973,Zhou2015,Borthagaray2021}. Most works for nonequilibrium situations are mainly experimental, without stability or convergence guarantees,~\cite{Aursand2016,Tovkach2017,MacDonald2020,Lee2002,Mori1999,Luckhurst2003}.
An exception is~\cite{Schimming2021}, however in that model, the electric field is given and not computed through a PDE like in~\eqref{eq:system_strong_formulation}, and in~\cite{weber2021convergent}, a convergent numerical scheme for the director field model by Oseen and Frank with an elliptic equation for the electric field was designed and analyzed.
Energy-stable and convergent numerical schemes for the Q-tensor gradient flow (without electric field) have been designed and analyzed in~\cite{Zhao2016,Shen2019,Cai2017,gudibanda_weber_yue_2022,Yue2023}. To the best of our knowledge, the present work is the first proving convergence of a fully discrete scheme for a coupled Q-tensor model under the influence of an electric field.

\subsection{Outline} The structure of this article is as follows: In Section~\ref{sec:numscheme}, we present the numerical scheme for~\eqref{eq:system_strong_formulation} and prove basic properties such as well-posedness, energy-stability, and preservation of the symmetry and trace-free constraint of $Q$ at the discrete level. In Section~\ref{sec:conv}, we show that the scheme converges up to a subsequence to a weak solution of~\eqref{eq:system_strong_formulation} when $\varepsilon_3 = 0$. Then in the last section, Section~\ref{sec:num}, we present numerical experiments for the scheme. We illustrate that numerically, the truncation operator $\T_R$ is necessary for certain boundary conditions for the elliptic equation, since in this case the norms of the approximations of $Q$ grow over time without bound, which would cause the elliptic equation to become ill-posed. One of the experiments also shows that with this model, it is possible to simulate the Fr\'{e}edericksz transition~\cite{Freedericksz1927}, a phase transition commonly observed in liquid crystals that interact with electric or magnetic fields.


\section{The numerical scheme}\label{sec:numscheme}
We start by introducing the numerical scheme to discretize~\eqref{eq:system_strong_formulation}. To obtain a suitable time discretization, we will use convex splitting for the bulk potential. 

\subsection{Convex Splitting}
We rewrite the Landau-de Gennes bulk potential $\mathcal{F}_B:\R^{d\times d}\to \R$ as
\[
    \mathcal{F}_B(Q) = \frac{a}{2}\langle Q, Q\rangle_F - \frac{b}{3}\langle Q^2, Q\rangle_F + \frac{c}{4}(\langle Q, Q\rangle_F)^2 + A_0,
\]
where we used $\langle\cdot, \cdot\rangle_F$ to denote the Frobenius inner product: $\langle A,B\rangle_F:= \tr(A^\top B)$.
When $Q$ is symmetric, this is equal to
\[
    \mathcal{F}_B(Q) = \frac{a}{2}\tr(Q^2) - \frac{b}{3}\tr(Q^3) + \frac{c}{4}(\tr(Q^2))^2 + A_0.
\]
The following lemma is similar to Lemma 2.4 in \cite{Zhao2016}. We modify its proof here to make its statement more precise.

\begin{lemma}[Convex Splitting of $\mathcal{F}_B$]
We can write this bulk potential as
\[
    \mathcal{F}_B(Q) = \mathcal{F}_1(Q) - \mathcal{F}_2(Q),
\]
where
\begin{equation*}
    \begin{split}
        \mathcal{F}_1(Q) &= \frac{\beta_1}{2}\langle Q, Q\rangle_F - \frac{b}{3}\langle Q^2, Q\rangle_F + \frac{\beta_2}{4}(\langle Q, Q\rangle_F)^2 + A_0,\\
        \mathcal{F}_2(Q) &= \frac{\beta_1 - a}{2}\langle Q, Q\rangle_F + \frac{\beta_2-c}{4}(\langle Q, Q\rangle_F)^2.
    \end{split}
\end{equation*}
If $\beta_1 \ge \max\{|b|,a\}$ and $\beta_2\ge\max\{|b|,c\}$, then $\mathcal{F}_1(Q)$ and $\mathcal{F}_2(Q)$ are convex functions.
\end{lemma}
\begin{proof}
We begin by showing that $\mathcal{F}_2$ is convex by showing its Hessian is positive semidefinite. In what follows, we will use Einstein summation convention. We have
\begin{align*}
    \frac{\partial\mathcal{F}_2}{\partial Q_{ij}}
        &= \frac{\beta_1-a}{2}(\langle Q, E_{ij}\rangle_F + \langle E_{ij}, Q\rangle_F) + \frac{\beta_2-c}{2}\langle Q, Q\rangle_F(\langle Q,E_{ij}\rangle_F + \langle E_{ij}, Q\rangle_F)\\
        &= (\beta_1 - a)Q_{ij} + (\beta_2 - c)Q_{ks}Q_{ks}Q_{ij}.
\end{align*}
Thus we have the entries of the Hessian are given by
\begin{align*}
    \frac{\partial^2\mathcal{F}_2}{\partial Q_{mn}\partial Q_{ij}}
        &= (\beta_1-a)\delta_{im}\delta_{jn} + (\beta_2-c)(2Q_{mn}Q_{ij} + \delta_{im}\delta_{jn}Q_{ks}Q_{ks}).
\end{align*}
Multiplying this expression by $x_{mn}$ and $x_{ij}$ and summing over $i,j,m,n$, we have
\begin{align*}
    x_{mn}\frac{\partial^2\mathcal{F}_2}{\partial Q_{mn}\partial Q_{ij}}x_{ij}
        &= (\beta_1-a)x_{mn}x_{mn} + (\beta_2-c)(2Q_{mn}x_{mn}Q_{ij}x_{ij} + x_{mn}x_{mn}Q_{ks}Q_{ks}) \ge 0,
\end{align*}
where in the last inequality we have used $\beta_1 \ge a$ and $\beta_2 \ge c$. This shows that $\mathcal{F}_2(Q)$ is convex. Now we show that $\mathcal{F}_1(Q)$ is convex. We begin by computing
\begin{align*}
    \frac{\partial\mathcal{F}_1}{\partial Q_{ij}}
        &= \frac{\beta_1}{2}(\langle E_{ij}, Q\rangle_F + \langle Q, E_{ij}\rangle_F) - \frac{b}{3}\left(\left\langle \frac{\partial}{\partial Q_{ij}} Q^2, Q\right\rangle_F + \langle Q^2, E_{ij}\rangle_F\right)\\ 
        &\hspace{10ex}+ \frac{\beta_2}{2}\langle Q, Q\rangle_F(\langle E_{ij}, Q\rangle_F + \langle Q, E_{ij}\rangle_F)\\
        &= \beta_1 Q_{ij} - \frac{b}{3}(Q_{ik}Q_{jk} + Q_{ki}Q_{kj} + Q_{ik}Q_{kj}) + \beta_2 Q_{ks}Q_{ks}Q_{ij}.
\end{align*}
Now computing the entries of the Hessian, we have 
\begin{align*}
    \frac{\partial^2\mathcal{F}_1}{\partial Q_{mn}\partial Q_{ij}}
        &= \beta_1 \delta_{im}\delta_{jn} - \frac{b}{3}\bigg(\delta_{im}\delta_{kn}Q_{jk} + \delta_{jm}\delta_{kn}Q_{ik} + \delta_{km}\delta_{in}Q_{kj}\bigg.\\ 
        &\hspace{5ex}\bigg.+ \delta_{km}\delta_{jn}Q_{ki} + \delta_{im}\delta_{kn}Q_{kj} + \delta_{km}\delta_{jn}Q_{ik}\bigg) + \beta_2(2\delta_{km}\delta_{sn}Q_{ks}Q_{ij} + \delta_{im}\delta_{jn}Q_{ks}Q_{ks})\\
        &= \beta_1\delta_{im}\delta_{jn} - \frac{b}{3}(\delta_{im}Q_{jn} + \delta_{jm}Q_{in} + \delta_{in}Q_{mj} + \delta_{jn}Q_{mi} + \delta_{im}Q_{nj} + \delta_{jn}Q_{im})\\ 
        &\hspace{5ex}+ \beta_2(2Q_{mn}Q_{ij} + \delta_{im}\delta_{jn}Q_{ks}Q_{ks}).
\end{align*}
Multiplying this by $x_{mn}$ and $x_{ij}$ and summing over $i,j,m,n$ gives
\begin{align*}
    x_{mn}\frac{\partial^2\mathcal{F}_1}{\partial Q_{mn}\partial Q_{ij}}x_{ij}
        &= \beta_1 x_{mn}x_{mn} - \frac{b}{3}\bigg(x_{mn}x_{mj}Q_{jn} + x_{mn}x_{im}Q_{in} + x_{mn}x_{nj}Q_{mj}\bigg.\\ 
        &\hspace{20ex}+ \bigg.x_{mn}x_{in}Q_{mi} + x_{mn}x_{mj}Q_{nj} + x_{mn}x_{in}Q_{im}\bigg)\\ 
        &\hspace{5ex}+ \beta_2(2Q_{mn}x_{mn}Q_{ij}x_{ij} + x_{mn}x_{mn}Q_{ks}Q_{ks}).
\end{align*}
Note that
\begin{align*}
    x_{mn}Q_{jn}x_{mj} = \sum_m\sum_n\sum_j x_{mn}Q_{jn}x_{mj} = \sum_{j,n} Q_{jn}\sum_{m}x_{mn}x_{mj} = \sum_{j,n} Q_{jn} B_{jn} \le \rvert Q\lvert_F\lvert B\rvert_F,
\end{align*}
where $B_{jn} := \sum_m x_{mn}x_{mj}$ and where we have applied the Cauchy-Schwarz inequality. Using the Cauchy-Schwarz inequality once more for $B$, we have
\begin{equation*}
    \lvert B\rvert_F = \left(\sum_{j,n} B_{jn}^2\right)^{1/2} \le \left(\sum_{j,n} \left(\sum_{m_1} x_{m_1n}^2\right)\left(\sum_{m_2} x_{m_2j}^2\right)\right)^{1/2}\\ 
    = \sum_{m,n} x_{mn}^2 = x_{mn}x_{mn},
\end{equation*}
and thus we have
\[
    \lvert x_{mn}Q_{jn}x_{mj}\rvert \le \lvert Q\rvert_Fx_{mn}x_{mn}.
\]
We can similarly bound the other terms of this form by $\lvert Q\rvert_Fx_{mn}x_{mn}$. Thus, using this bound and dropping the term $2Q_{mn}x_{mn}Q_{ij}x_{ij} \ge 0$, we obtain
\begin{align*}
    x_{mn}\frac{\partial^2\mathcal{F}_1}{\partial Q_{mn}\partial Q_{ij}}x_{ij}
        &\ge \beta_1 x_{mn}x_{mn} - 2\lvert b\rvert\lvert Q\rvert_Fx_{mn}x_{mn}+ \beta_2x_{mn}x_{mn}\lvert Q\rvert_F^2\\
        &\ge \lvert b\rvert x_{mn}x_{mn}(1 - \lvert Q\rvert_F)^2\ge 0,
\end{align*}
since $\beta_1,\beta_2 \ge \lvert b\rvert$. This shows that $\mathcal{F}_1(Q)$ is convex, as desired.
\end{proof}

\subsection{Definition of the Numerical Scheme}
Now we are in a position to define the numerical scheme. To discretize in time, we will use a fully implicit discretization and for the spatial discretization, we will use finite elements. We let $\Delta t>0$ and denote $t^n := n\Delta t$. Also, $N\in\N$ is such that $T = t^N$ for a given final time of computation $T>0$. Then we let $\mathfrak{T}_h = \{K_\alpha\}_{\alpha\in\Lambda}$ be a quasi-uniform triangulation of $\Omega$ with mesh size $h$. Each $K$ is a closed subset of $\Omega$. We assume for simplicity that $\Omega$ is a convex polygon/polyhedron so that we do not need to deal with the approximation of curved boundaries. When $\Omega$ is 2-dimensional, each $K$ is a triangle, while when $\Omega$ is 3-dimensional they are tetrahedrons. We define the spaces 
\begin{subequations}
	\begin{align}
	\mathbb{X}_h &= \{f_h \in H_0^1(\Omega; \R^{d\times d}): \left.(f_h)_{ij}\right\rvert_K \in P_1(K) \text{ for every } K\in\mathfrak{T}_h,\, 1\le i,j\le d\},\\
	\mathbb{Y}_h &= \{v_h \in H_0^1(\Omega): \left.v_h\right\rvert_K \in P_1(K) \text{ for every } K\in\mathfrak{T}_h\}.
	\end{align}
\end{subequations}
We let $\tilde g^n= \tilde g(\Delta t n)$ and denote by $\tilde g_h^n$ the $L^2(\Omega)$-orthogonal projection of the extension of the boundary data $\tilde g^n$ onto continuous and piecewise linear functions on $\mathfrak{T}_h$. We analogously define $\tilde q_h$ to be the $L^2(\Omega)$-orthogonal projection of the extension of the boundary data $\tilde q$ onto continuous, piecewise linear and matrix valued functions on $\mathfrak{T}_h$. Now the fully discrete formulation of the numerical scheme is as follows: 

Find $\tilde Q_h^n \in \mathbb{X}_h$ and $\tilde u_h^n \in \mathbb{Y}_h$ for $n = 1,\dots,N$ such that
\begin{subequations}
	\label{eq:fully_discrete_scheme_definition}
	\begin{equation}\label{eq:Qfullydiscrete}
	\begin{split}
	- \int_\Omega D_t^+ Q_h^n : \Phi_h\, dx &= L \int_\Omega \nabla Q_h^{n+1/2}:\nabla\Phi_h\, dx +  \int_\Omega \left(\frac{\partial\mathcal F_1(Q_h^{n+1})}{\partial Q} - \frac{\partial\mathcal F_2(Q_h^n)}{\partial Q}\right):\Phi_h\, dx\\ 
	&- \frac{\varepsilon_2}{2} \int_\Omega\left((\tilde\P_h^n\odot \nabla u_h^n(\nabla u_h^{n+1})^\top)_S - \frac1d\operatorname{tr}(\tilde\P_h^n\odot \nabla u_h^n(\nabla u_h^{n+1})^\top)I\right):\Phi_h\, dx\\ 
	&- \varepsilon_3 \int_\Omega \left(\nabla u_h^{n+1/2}\cdot\operatorname{div}((\Phi_h)_S) - \frac1d\nabla u_h^{n+1/2}\cdot\nabla\operatorname{tr}\Phi_h\right)\, dx,
	\end{split}
	\end{equation}
	\begin{equation}\label{eq:ufullydiscrete}
	\int_\Omega (\varepsilon_1\nabla u_h^{n} + \varepsilon_2 \T_R(Q_h^n)\nabla u_h^{n} + \varepsilon_3\operatorname{div}(Q_h^{n}))\cdot\nabla\psi_h\, dx = 0
	\end{equation}
\end{subequations}
for all $\Phi_h \in \mathbb{X}_h$ and $\psi_h \in \mathbb{Y}_h$, where $Q_h^n = \tilde Q_h^n + \tilde q_h$ and $u_h^n = \tilde u_h^n + \tilde g_h^n$. Here $D_t^+$ is the discrete time derivative defined by
\[
    D_t^+ Q_h^n = \frac{Q_h^{n+1} - Q_h^n}{\Delta t},
\]
and the superscript $n+1/2$ is used to denote the average of approximations at time steps $n$ and $n+1$:
\begin{equation*}
Q^{n+1/2}_h=\frac{Q^n_h+Q^{n+1}_h}{2},\quad u^{n+1/2}_h=\frac{u^n_h+u^{n+1}_h}{2}, 
\end{equation*}
and $\tilde\P_h^n$ is understood as $\tilde\P_h^n = \tilde\P(Q_h^{n+1}, Q_h^n)$, where
\[
(\tilde\P_h^n)_{ij} = \begin{cases}
\frac{\T_R(Q^{n+1}_h)_{ij} - \T_R(Q_h^n)_{ij}}{(Q^{n+1}_h)_{ij}-(Q^n_h)_{ij}} &\text{if } (Q_h^{n+1})_{ij} - (Q^n_h)_{ij} \ne 0\\
\frac{\partial\T_R(Q_h^n)_{ij}}{\partial Q} &\text{otherwise}.
\end{cases}
\]
We will sometimes make the dependence of $\tilde \P^n_h$ on $Q^{n+1}_h$ and $Q^n_h$ more explicit by instead writing $\tilde \P(Q^{n+1}_h, Q^n_h)$.
Before proving stability and well-posedness of the above scheme, we will show that it preserves the trace-free and symmetry properties of $Q_h^n$. To do so, we require the map $\mathcal{L}: \mathbb{X}_h\times \mathbb{Y}_h\to \mathbb{X}_h\times \mathbb{Y}_h$,  $\hat{x} = \mathcal{L}x$
\[
    \{\tilde Q_h^{n+1}, \tilde u_h^{n+1}\} \xmapsto[]{\mathcal{L}} \{\widehat{Q}_h^{n+1}, \widehat{u}_h^{n+1}\},
\]
where the quantities with hats solve
\begin{subequations}
    \label{eq:fixed_point_L_map}
    \begin{equation}
    \label{eq:fixed_point_Q_equation}
    \begin{split}
		&- \int_\Omega \frac{\widehat{Q}_h^{n+1} + \tilde q_h - Q_h^n}{\Delta t} : \Phi_h\, dx\\ 
		&= L \int_\Omega \frac{\nabla\widehat{Q}_h^{n+1}+\nabla \tilde q_h+\nabla Q_h^n}{2}:\nabla\Phi_h\, dx +  \int_\Omega \left(\frac{\partial\mathcal F_1(Q_h^{n+1})}{\partial Q} - \frac{\partial\mathcal F_2(Q_h^n)}{\partial Q}\right):\Phi_h\, dx\\ 
		&- \frac{\varepsilon_2}{2} \int_\Omega\left((\tilde\P(Q_h^{n+1},Q_h^n)\odot \nabla u_h^n(\nabla \widehat{u}_h^{n+1})^\top)_S - \frac1d\operatorname{tr}(\tilde\P(Q_h^{n+1},Q_h^n)\odot \nabla u_h^n(\nabla \widehat{u}_h^{n+1})^\top)I\right):\Phi_h\, dx\\ 
		&- \frac{\varepsilon_2}{2} \int_\Omega\left((\tilde\P(Q_h^{n+1},Q_h^n)\odot \nabla u_h^n(\nabla \tilde g_h^{n+1})^\top)_S - \frac1d\operatorname{tr}(\tilde\P(Q_h^{n+1},Q_h^n)\odot \nabla u_h^n(\nabla \tilde g_h^{n+1})^\top)I\right):\Phi_h\, dx\\
		&- \varepsilon_3 \int_\Omega \left(\frac{\nabla \widehat{u}_h^{n+1} + \nabla u_h^n}{2}\cdot\operatorname{div}((\Phi_h)_S) - \frac1d\frac{\nabla \widehat{u}_h^{n+1}+\nabla u^n}{2}\cdot\nabla\operatorname{tr}\Phi_h\right)\, dx\\
		&- \frac{\varepsilon_3}{2} \int_\Omega \left(\nabla \tilde g_h^{n+1}\cdot\operatorname{div}((\Phi_h)_S) - \frac1d\nabla \tilde g_h^{n+1}\cdot\nabla\operatorname{tr}\Phi_h\right),
\end{split}
\end{equation}
\begin{equation}
\label{eq:fixed_point_elliptic_definition}
\begin{split}
    &\int_\Omega (\varepsilon_1\nabla \widehat{u}_h^{n+1} + \varepsilon_2 \T_R(Q_h^{n+1})\nabla \widehat{u}_h^{n+1})\cdot\nabla\psi_h\, dx\\ 
    &= -\int_\Omega (\varepsilon_1\nabla \tilde g_h^{n+1} + \varepsilon_2 \T_R(Q_h^{n+1})\nabla \tilde g_h^{n+1} + \varepsilon_3\operatorname{div}(Q_h^{n+1}))\cdot\nabla\psi_h\, dx
    \end{split}
    \end{equation}
\end{subequations}
for all $\Phi_h \in \mathbb{X}_h$ and $\psi_h\in\mathbb{Y}_h$, where $Q_h^{n+1} = \tilde Q_h^{n+1} + \tilde q_h$.

We will prove that this map is well defined in Theorem~\ref{thm:existence_of_solutions}, but first we will use the map to prove the following result, which implies a symmetry and trace-free preservation property of scheme~\eqref{eq:fully_discrete_scheme_definition} needed for proving solvability of the scheme:

\begin{theorem}[Trace-Free and Symmetry Preservation Generalized]
If $Q_h^n$ is trace-free and symmetric and $\lambda\in (0,1]$, we have that every $\widehat{Q}_h^{n+1}$ which satisfies $\lambda^{-1}\widehat{x} = \mathcal{L}\widehat{x}$ is trace-free and symmetric. In particular, for $\lambda = 1$, this shows that solutions $Q_h^{n+1}$ to (\ref{eq:fully_discrete_scheme_definition}) are trace-free and symmetric.
\end{theorem}
\begin{proof}
	Let $\lambda \in (0,1]$ and let $\widehat{Q}_h^{n+1}$ satisfy $\lambda^{-1}\widehat{x} = \mathcal{L}\widehat{x}$. We begin by showing that $\widehat{Q}_h^{n+1}$ is symmetric. Take $\Phi_h = \widehat{Q}_h^{n+1} - (\widehat{Q}_h^{n+1})^\top := V_h^{n+1}$ in \eqref{eq:fixed_point_Q_equation}. Note that $V_h^{n+1}$ is skew-symmetric. Writing $\overline Q_h^{n+1} = \widehat{Q}_h^{n+1} + \tilde q_h$, it follows that
	\[
	\int_\Omega\left((\tilde\P(\overline Q_h^{n+1},Q_h^n)\odot \nabla u_h^n(\nabla \widehat{u}_h^{n+1})^\top)_S - \frac{1}{d}\operatorname{tr}(\tilde\P(\overline Q_h^{n+1},Q_h^n)\odot \nabla u_h^n(\nabla \widehat{u}_h^{n+1})^\top)I\right):V_h^{n+1}\, dx = 0
	\]
	since $(\tilde\P(\overline Q_h^{n+1},Q_h^n)\odot \nabla u_h^n(\nabla \widehat{u}_h^{n+1})^\top)_S - \frac{1}{d}\operatorname{tr}(\tilde\P(\overline Q_h^{n+1},Q_h^n)\odot \nabla u_h^n(\nabla \widehat{u}_h^{n+1})^\top)I$ is symmetric and $V_h^{n+1}$ is skew-symmetric. For the same reason, 
	\begin{equation*}
	\int_\Omega\left((\tilde\P(\overline Q_h^{n+1},Q_h^n)\odot \nabla u_h^n(\nabla \tilde g_h^{n+1})^\top)_S - \frac{1}{d}\operatorname{tr}(\tilde\P(\overline Q_h^{n+1},Q_h^n)\odot \nabla u_h^n(\nabla \tilde g_h^{n+1})^\top)I\right):V_h^{n+1}\, dx = 0.
	\end{equation*}
	Furthermore, we have
	\begin{equation*}
	\int_\Omega \left(\frac{\nabla \widehat{u}_h^{n+1}+\nabla u_h^n}{2}\cdot\operatorname{div}((V_h^{n+1})_S) - \frac{1}{d}\frac{\nabla \widehat{u}_h^{n+1}+\nabla u_h^n}{2}\cdot\nabla\operatorname{tr}V_h^{n+1}\right)\, dx = 0
	\end{equation*}
	and
	\begin{equation*}
	\int_\Omega \left(\nabla\tilde g_h^{n+1}\cdot \operatorname{div}((V_h^{n+1})_S) - \frac1d\nabla\tilde g_h^{n+1}\cdot\nabla\operatorname{tr}V_h^{n+1}\right)\, dx = 0
	\end{equation*}
	since $(V_h^{n+1})_S = 0$ and $\operatorname{tr}V_h^{n+1} = 0$. We also have
	\[
	\int_\Omega \tilde q_h : V_h^{n+1}\, dx = \int_\Omega \nabla \tilde q_h : \nabla V_h^{n+1}\, dx = \int_\Omega Q_h^n:V_h^{n+1}\, dx = \int_\Omega \nabla Q_h^n : \nabla V_h^{n+1}\, dx = 0
	\]
	because $\tilde q_h$ and $Q_h^n$ are symmetric and $V_h^{n+1}$ is skew-symmetric. Using these facts again gives
	\[
	\int_\Omega \frac{\partial \mathcal{F}_2(Q_h^n)}{\partial Q}:V_h^{n+1}\, dx = \int_\Omega \left((\beta_1 - a)Q_h^n + (\beta_2 - c)\operatorname{tr}((Q_h^n)^2)Q_h^n\right):V_h^{n+1}\, dx = 0.
	\]
	Thus we have
	\begin{equation}
	\frac{1}{\lambda\Delta t}\int_\Omega \widehat{Q}_h^{n+1}:V_h^{n+1}\, dx + \frac{L}{2\lambda}\int_\Omega \nabla \widehat{Q}_h^{n+1}: \nabla V_h^{n+1}\, dx + \int_\Omega \frac{\partial \mathcal{F}_1(\overline{Q}_h^{n+1})}{\partial Q}: V_h^{n+1}\, dx = 0.
	\end{equation}
	Now because $V_h^{n+1}$ is skew-symmetric and $\left(\frac{\partial \mathcal{F}_1(\overline{Q}_h^{n+1})}{\partial Q}\right)^\top = \frac{\partial \mathcal{F}_1((\overline{Q}_h^{n+1})^\top)}{\partial Q}$, we also have
	\begin{equation}
	-\frac{1}{\lambda\Delta t}\int_\Omega (\widehat{Q}_h^{n+1})^\top:V_h^{n+1}\, dx - \frac{L}{2\lambda}\int_\Omega \nabla (\widehat{Q}_h^{n+1})^\top: \nabla V_h^{n+1}\, dx - \int_\Omega \frac{\partial \mathcal{F}_1((\overline{Q}_h^{n+1})^\top)}{\partial Q}: V_h^{n+1}\, dx = 0.
	\end{equation}
	Adding the previous two equations together gives
	\begin{equation}
	\frac{1}{\lambda\Delta t}\|V_h^{n+1}\|^2 + \frac{L}{2\lambda}\|\nabla V_h^{n+1}\|^2 + \left\langle \frac{\partial \mathcal{F}_1(\overline{Q}_h^{n+1})}{\partial Q} - \frac{\partial \mathcal{F}_1((\overline{Q}_h^{n+1})^\top)}{\partial Q}, \overline{Q}_h^{n+1} - (\overline{Q}_h^{n+1})^\top\right\rangle = 0,
	\end{equation}
 where we have used $V_h^{n+1} = \widehat{Q}_h^{n+1} - (\widehat{Q}_h^{n+1})^\top = \overline{Q}_h^{n+1} - (\overline{Q}_h^{n+1})^\top$ since $\tilde q_h^{n+1}$ is symmetric.
	And since $\mathcal{F}_1$ is convex, it follows that
	\begin{align*}
	&\left\langle \frac{\partial \mathcal{F}_1(\overline{Q}_h^{n+1})}{\partial Q} - \frac{\partial \mathcal{F}_1((\overline{Q}_h^{n+1})^\top)}{\partial Q}, \overline{Q}_h^{n+1} - (\overline{Q}_h^{n+1})^\top\right\rangle \\
	&= \left\langle \frac{\partial \mathcal{F}_1(\overline{Q}_h^{n+1})}{\partial Q} , \overline{Q}_h^{n+1} - (\overline{Q}_h^{n+1})^\top\right\rangle - \left\langle \frac{\partial \mathcal{F}_1((\overline{Q}_h^{n+1})^\top)}{\partial Q}, \overline{Q}_h^{n+1} - (\overline{Q}_h^{n+1})^\top\right\rangle\\
	&\ge -(\mathcal{F}_1((\overline{Q}_h^{n+1})^\top) - \mathcal{F}_1(\overline{Q}_h^{n+1})) - (\mathcal{F}_1(\overline{Q}_h^{n+1}) - \mathcal{F}_1((\overline{Q}_h^{n+1})^\top))\\
	&= 0.
	\end{align*}
	Thus
	\[
	\frac{1}{\lambda\Delta t} \|V_h^{n+1}\|^2 + \frac{L}{2\lambda}\|\nabla V_h^{n+1}\|^2 \le 0,
	\]
	which implies that $V_h^{n+1} \equiv 0$ so that $\widehat{Q}_h^{n+1}$ is symmetric.
	
	We now show that $\widehat{Q}_h^{n+1}$ is trace-free. Taking $\Phi_h = \operatorname{tr}(\widehat{Q}_h^{n+1})I$ as a test function in~\eqref{eq:fixed_point_Q_equation}, we have
	\begin{align*}
	&-\frac{1}{\lambda\Delta t}\int_\Omega \lvert \operatorname{tr}\widehat{Q}_h^{n+1}\rvert^2\, dx\\ 
	&\hspace{5ex}= L\int_\Omega \frac{\lambda^{-1}\nabla \widehat{Q}_h^{n+1}+\nabla\tilde q_h+\nabla Q_h^n}{2} : \nabla(\operatorname{tr}(\widehat{Q}_h^{n+1})I)\, dx + \int_\Omega \left(\frac{\partial\mathcal{F}_1(\overline{Q}_h^{n+1})}{\partial Q} - \frac{\partial\mathcal{F}_2(Q_h^n)}{\partial Q}\right) : \operatorname{tr}(\widehat{Q}_h^{n+1})I\, dx\\
	&\hspace{5ex}= \frac{L}{2\lambda}\int_\Omega \lvert \nabla \operatorname{tr}(\widehat{Q}_h^{n+1})\rvert^2\, dx + \int_\Omega \left(\frac{\partial\mathcal{F}_1(\overline{Q}_h^{n+1})}{\partial Q} - \frac{\partial\mathcal{F}_2(Q_h^n)}{\partial Q}\right) : \operatorname{tr}(\widehat{Q}_h^{n+1})I\, dx,
	\end{align*}
	where for the last equality we used $\nabla A : \nabla (\operatorname{tr}(B)I) = \nabla \operatorname{tr}(A) \cdot \nabla \operatorname{tr}(B)$ and the fact that $\tilde q_h$ and $Q_h^n$ are trace-free. The remaining terms involving $\Grad u^n_h$ and $\Grad u^{n+1}_h$ cancel out due to the structure of the test function. Then computing the second integral above, we have
	\[
	\int_\Omega \frac{\partial\mathcal{F}_2(Q_h^n)}{\partial Q} : \operatorname{tr}(\widehat{Q}_h^{n+1})I\, dx = \int_\Omega \left((\beta_1 - a)Q_h^n + (\beta_2 - c)\operatorname{tr}((Q_h^n)^2)Q_h^n\right) : \operatorname{tr}(\widehat{Q}_h^{n+1})I\, dx = 0
	\]
	since $\operatorname{tr}(Q_h^n) = 0$. Furthermore,
	\begin{align*}
	&\int_\Omega \frac{\partial\mathcal{F}_1(\overline{Q}_h^{n+1})}{\partial Q} : \operatorname{tr}(\widehat{Q}_h^{n+1})I\, dx \\
	&\hspace{5ex}= \int_\Omega \left(\beta_1 \overline{Q}_h^{n+1} - b((\overline{Q}_h^{n+1})^2 - \frac{1}{d}\operatorname{tr}((\overline{Q}_h^{n+1})^2)I) + \beta_2\operatorname{tr}((\overline{Q}_h^{n+1})^2)\overline{Q}_h^{n+1}\right) : \operatorname{tr}(\widehat{Q}_h^{n+1})I\, dx\\
	&\hspace{5ex}= \int_\Omega \left(\beta_1\lvert \operatorname{tr}(\overline{Q}_h^{n+1})\rvert^2 + \beta_2\operatorname{tr}((\overline{Q}_h^{n+1})^2)\lvert\operatorname{tr}(\overline{Q}_h^{n+1})\rvert^2\right)\, dx,
	\end{align*}
	where we have used $\operatorname{tr}(\widehat{Q}_h^{n+1}) = \operatorname{tr}(\overline{Q}_h^{n+1})$ since $\operatorname{tr}(\tilde q_h) = 0$. Using this fact again, it follows that
	\[
	\int_\Omega \left(\frac{1}{\lambda\Delta t} + \beta_1 + \beta_2\operatorname{tr}((\overline{Q}_h^{n+1})^2)\right)\lvert\operatorname{tr}(\overline{Q}^{n+1})\rvert^2\, dx + \frac{L}{2\lambda}\int_\Omega \lvert \nabla \operatorname{tr}(\overline{Q}_h^{n+1})\rvert^2\, dx = 0.
	\]
	Now since $\overline{Q}_h^{n+1}$ is symmetric, we have that $\operatorname{tr}((\overline{Q}_h^{n+1})^2) = \lvert \overline{Q}_h^{n+1}\rvert_F^2 \ge 0$. We also have $\lambda, \beta_1, \beta_2 > 0$. Thus,
	\[
	\frac{1}{\Delta t}\int_\Omega \lvert \operatorname{tr}(\overline{Q}_h^{n+1})\rvert^2\, dx = \frac{1}{\Delta t}\int_\Omega \lvert \operatorname{tr}(\widehat{Q}_h^{n+1})\rvert^2\, dx \le 0,
	\]
	which shows that $\operatorname{tr}(\widehat{Q}_h^{n+1}) \equiv 0$, as desired.
\end{proof}

Note that it follows from this result that $\tilde \P^n_h$ is also symmetric. Next, we derive a uniform energy estimate for the scheme.

\begin{theorem}[Energy Stability]
	\label{thm:fully_discrete_energy_stability}
	Assume that $\Delta t < \min\{\frac{L}{2d|\varepsilon_3|}, \frac{\varepsilon_1 - |\varepsilon_2|R}{3(\varepsilon_1 + |\varepsilon_2|R^2)}\}$ and $\varepsilon_1 > |\varepsilon_2|R$. Then the approximations defined by the scheme (\ref{eq:fully_discrete_scheme_definition}) satisfy for all $N\in\mathbb{N}$
	\begin{align*}
	&\int_\Omega \mathcal{F}_B(Q_h^N)\, dx + \frac14\left(L - 2d\Delta t |\varepsilon_3|\right)\|\nabla Q_h^N\|_{L_2}^2 + \frac{1}{6}\left(\varepsilon_1 - |\varepsilon_2|R\right)\|\nabla u_h^N\|_{L^2}^2\\
	&\le C(1+t^N\operatorname{exp}(Ct^N))\left( \|\nabla Q_h^0\|_{L^2}^2 + \mathcal{F}_B(Q_h^0) + (1+R)\|\nabla u_h^0\|_{L^2}^2 + \|\nabla \tilde g_h^0\|_{L^2}^2\right)\\ 
	&+ C(1 + t^N\operatorname{exp}(Ct^N))\Delta t \sum_{n=0}^{N-1} \|D_t^+ \nabla\tilde g_h^n\|_{L^2}^2 + C\|\nabla\tilde g_h^N\|_{L^2}^2 + C\operatorname{exp}(Ct^N)\Delta t\sum_{n=0}^{N-1} \|\nabla \tilde g_h^n\|_{L^2}^2
	\end{align*}
	for some constant $C>0$ not depending on $\Delta t$ or $h$.
\end{theorem}
\begin{proof}
	The proof of this result can be found in Appendix~\ref{app:energystability}.
\end{proof}
\begin{remark}[Physical parameter range]
	The condition $\epsilon_1>|\epsilon_2|R$ may seem restrictive. Some typical values for the constants $\epsilon_1$ and $\epsilon_2$ can be found for example in the book by Stewart~\cite{Stewart2008} on page 330. These are $(\epsilon_1,\epsilon_2)=(4.7,-0.7)$, $(\epsilon_1,\epsilon_2)=(5.538,-0.167)$ and $(\epsilon_1,\epsilon_2)=(18.5,11.5)$. If $R\leq 3/2$, which would still allow the entries of the Q-tensor to attain their whole physical range, each of these parameter combinations would satisfy the assumptions.
\end{remark}
As a corollary, we obtain:
\begin{corollary}
	\label{corollary:fully_discrete_DtQ_L2_bound}
	Assume the hypotheses of Theorem~\ref{thm:fully_discrete_energy_stability}. Then 
	\[
	\Delta t \sum_{n=0}^{N-1}\int_\Omega \lvert D_t^+ Q_h^n\rvert^2\, dx < C
	\]
	for all $h,\Delta t >0$ sufficiently small for some $C<\infty$.
\end{corollary}

\begin{proof}
This follows directly from equation~\eqref{eq:L2_bound_on_derivative_Qn} and Theorem~\ref{thm:fully_discrete_energy_stability}.
\end{proof}


\subsection{Solvability of the scheme}
In this section, we will show that solutions to the scheme~\eqref{eq:fully_discrete_scheme_definition} exist. Before doing so, we require the following lemma.

\begin{lemma}\label{lem:PLipschitz}
Let $B\in\R^{d\times d}$. Then $\tilde\P(\cdot, B)$ is Lipschitz.
\end{lemma}
\begin{proof}
Recall the definition of $\tilde\P$,
\[
    \tilde\P(Q, B)_{ij} = \begin{cases}
        \frac{\T_R(Q)_{ij} - \T_R(B)_{ij}}{Q_{ij} - B_{ij}} &\text{if } Q_{ij} \ne B_{ij}\\
        \frac{\partial\T_R(B)_{ij}}{\partial Q} &\text{if } Q_{ij} = B_{ij}.
    \end{cases}
\]
Clearly $\tilde\P$ is continuous. To show that $\tilde\P$ is Lipschitz, we first show that its derivative is continuously differentiable. For $Q_{ij}\ne B_{ij}$, since $\T_R$ is smooth, by the quotient rule we have 
\begin{align*}
    \frac{\partial\tilde\P(Q,B)_{ij}}{\partial Q_{ij}} 
    &= \frac{1}{Q_{ij} - B_{ij}}\left[\frac{\partial\T_R(Q)_{ij}}{\partial Q_{ij}} - \frac{\T_R(B)_{ij} - \T_R(Q)_{ij}}{B_{ij} - Q_{ij}}\right]\\
    &= \frac{1}{Q_{ij} - B_{ij}}\left[\frac{\partial\T_R(Q)_{ij}}{\partial Q_{ij}} - \frac{\partial\T_R(B)_{ij}}{\partial Q_{ij}} - \frac12 \frac{\partial^2\T_R(B)_{ij}}{\partial Q_{ij}^2}(Q_{ij} - B_{ij}) + o(Q_{ij} - B_{ij})\right]\\
    &= \frac{1}{Q_{ij} - B_{ij}}\left[\frac{\partial\T_R(Q)_{ij}}{\partial Q_{ij}} - \frac{\partial\T_R(B)_{ij}}{\partial Q_{ij}}\right] - \frac12\frac{\partial^2\T_R(B)_{ij}}{\partial Q_{ij}^2} + \frac{o(Q_{ij} - B_{ij})}{Q_{ij} - B_{ij}}.
\end{align*}
As $Q_{ij}\to B_{ij}$, we see this has limit $\frac12 \frac{\partial^2\T_R(B)_{ij}}{\partial Q_{ij}^2}$. Now applying the limit definition of the derivative and again writing Taylor approximations, we see that for $Q_{ij} = B_{ij}$, we also have
\[
    \frac{\partial\tilde\P(Q,B)_{ij}}{\partial Q_{ij}} = \frac12 \frac{\partial^2 \T_R(B)_{ij}}{\partial Q_{ij}^2}.
\]
Thus $\tilde\P$ is continuously differentiable since we assumed that $\mathcal{T}_R\in C^2(\R^d\times\R^d)$. Now using the fact that $\T_R$ is bounded with bounded derivatives, we see that $\tilde{\mathcal{P}}$ has bounded derivative and is therefore Lipschitz, as desired.
\end{proof}

We now prove that solutions to the scheme exist using the Leray-Schauder fixed point theorem.
\begin{theorem}[Existence of Solutions]
\label{thm:existence_of_solutions}
Assume $\tilde Q_h^n \in \mathbb{X}_h$ and $\tilde u_h^n \in \mathbb{Y}_h$ are given, and let $Q_h^n = \tilde Q_h^n + \tilde q_h$ and $u_h^n = \tilde u_h^n + \tilde g_h^n$. Assume that $\varepsilon_1 > |\varepsilon_2|R + |\varepsilon_3|$ and $\lvert\varepsilon_3\rvert < \frac{L}{12}$. Then the scheme (\ref{eq:fully_discrete_scheme_definition}) admits solutions for all $n=1,2,\dots,N$. 
\end{theorem}
\begin{proof}
Consider the map $\hat x = \mathcal{L}x$ given by (\ref{eq:fixed_point_L_map}). We now show that $\mathcal{L}$ satisfies the hypotheses of the Leray-Schauder fixed point theorem to assert the existence of a solution.
To apply the Leray-Schauder theorem, we need to show that the mapping is well-posed, compact, continuous,  as a mapping from $\mathbb{X}_h\times \mathbb{Y}_h$ to $\mathbb{X}_h\times \mathbb{Y}_h$, and that the set
\begin{equation*}
	K:=\{(\widehat{Q}_h^{n+1},\widehat{u}_h^{n+1})\, |\, (\widehat{Q}_h^{n+1},\widehat{u}_h^{n+1})= \lambda \mathcal{L}(\tilde{Q}_h^{n+1},\tilde{u}_h^{n+1})\quad \text{for some }\ 0\leq \lambda \leq 1\}
\end{equation*}
is bounded~\cite[page 540]{Evans2010} (we also need $0\in \mathbb{X}_h\times \mathbb{Y}_h$, which is clearly true).

\begin{itemize}
    \item[--] \textbf{Well-posedness.} We must show that the operator $\mathcal{L}$ is well defined. We begin by solving the elliptic equation~\eqref{eq:fixed_point_elliptic_definition}. This equation has a unique solution for $\widehat{u}_h^{n+1}$ due to the Lax-Milgram theorem.
    The integral on the left hand side of~\eqref{eq:fixed_point_elliptic_definition} defines a bilinear form $B : H_0^1(\Omega)\times H_0^1(\Omega)\to \R$, which is continuous and coercive as long as $\epsilon_1> |\epsilon_2| R$, and the right hand side defines a bounded linear form $F:H_0^1(\Omega)\to\R$ since since $\tilde g \in H^1(\Omega)$. 
    Thus, the Lax-Milgram lemma yields that there is a unique solution $\widehat{u}_h^{n+1}$ to~\eqref{eq:fixed_point_elliptic_definition}. Given $\widehat{u}_h^{n+1}$, a similar argument shows that a unique $\widehat{Q}^{n+1}_h$ exists for~\eqref{eq:fixed_point_Q_equation}.

    \item[--] \textbf{Boundedness} Next, we bound solutions of $\mathcal L\widehat{x} = \lambda^{-1}\widehat{x}$ in terms of $\tilde Q_h^n, \tilde u_h^n,$ and given data uniformly over $\lambda\in(0,1]$. We begin by taking $\psi_h = \lambda\widehat{u}_h^{n+1}$ in the definition~\eqref{eq:fixed_point_elliptic_definition} of $\mathcal{L}\widehat{x} = \lambda^{-1}\widehat{x}$ and apply the Cauchy-Schwarz inequality and Young's inequality to obtain
    \begin{equation}
        \begin{split}
            0
                &= \int_\Omega (\varepsilon_1\lvert \nabla \widehat{u}_h^{n+1}\rvert^2 + \varepsilon_2 (\nabla \widehat{u}_h^{n+1})^\top \T_R(\overline{Q}_h^{n+1})\nabla \widehat{u}_h^{n+1})\, dx\\ 
                &+ \lambda\int_\Omega (\varepsilon_1 \nabla\tilde g_h^{n+1}\cdot\nabla\widehat{u}_h^{n+1} + \varepsilon_2(\nabla\widehat{u}_h^{n+1})^\top \T_R(\overline{Q}_h^{n+1})\nabla\tilde g_h^{n+1} + \varepsilon_3 \operatorname{div}(\overline{Q}_h^{n+1})\cdot\nabla\widehat{u}_h^{n+1})\, dx\\
                &\ge (\varepsilon_1 - R|\varepsilon_2|)\|\nabla\widehat{u}_h^{n+1}\|_{L^2}^2 - \frac{\varepsilon_1\lambda}{2\delta_1}\|\nabla\tilde g_h^{n+1}\|_{L^2}^2 - \frac{\varepsilon_1\delta_1\lambda}{2}\|\nabla\widehat{u}_h^{n+1}\|_{L^2}^2 - \frac{|\varepsilon_2|\lambda R^2}{2\delta_2}\|\nabla\tilde g_h^{n+1}\|_{L^2}^2\\ 
                &\hspace{10ex}- \frac{|\varepsilon_2|\delta_2\lambda}{2}\|\nabla\widehat{u}_h^{n+1}\|_{L^2}^2 - \frac{|\varepsilon_3|\lambda}{2}\|\operatorname{div}\overline{Q}_h^{n+1}\|_{L^2}^2 - \frac{|\varepsilon_3|\lambda}{2}\|\nabla\widehat{u}_h^{n+1}\|_{L^2}^2,
        \end{split}
    \end{equation}
    where $\delta_1,\delta_2>0$ are to be chosen and $\overline{Q}_h^{n+1} = \widehat{Q}_h^{n+1} + \tilde q_h$. Now noting that $\lambda\le 1$ and $\|\operatorname{div}(\overline{Q}_h^{n+1})\|_{L^2}^2 \le d\|\nabla\overline{Q}_h^{n+1}\|_{L^2}^2$, we have
    \begin{equation}
    \label{eq:leray_schauder_boundedness_u_inequality}
        \begin{split}
            0              &\ge \left(\varepsilon_1 - R|\varepsilon_2| - \frac{\varepsilon_1\delta_1}{2} - \frac{|\varepsilon_2|\delta_2}{2} - \frac{|\varepsilon_3|}{2}\right)\|\nabla\widehat{u}_h^{n+1}\|_{L^2}^2 - \frac{\varepsilon_1}{2\delta_1}\|\nabla\tilde g_h^{n+1}\|_{L^2}^2\\ 
                &\hspace{10ex} - \frac{|\varepsilon_2|R^2}{2\delta_2}\|\nabla\tilde g_h^{n+1}\|_{L^2}^2 - d|\varepsilon_3|\|\nabla\widehat{Q}_h^{n+1}\|_{L^2}^2 - d|\varepsilon_3|\|\nabla\tilde q_h\|_{L^2}^2\\
        \end{split}
    \end{equation}
    Now take $\Phi_h = \lambda\widehat{Q}_h^{n+1}$ in \eqref{eq:fixed_point_Q_equation}. Considering this term by term, we have
    \begin{equation}
    \begin{split}
        &-\int_\Omega \frac{\lambda^{-1}\widehat{Q}_h^{n+1} +\tilde q_h - Q_h^n}{\Delta t} : \lambda\widehat{Q}_h^{n+1}\, dx\\ 
        &\le -\frac{1}{2\Delta t}\int_\Omega\left(\lvert\widehat{Q}_h^{n+1}\rvert^2 - \lvert Q_h^n\rvert^2\right)\, dx + \frac{\delta_3}{2\Delta t}\int_\Omega \lvert \widehat{Q}_h^{n+1}\rvert^2\, dx + \frac{1}{2\delta_3\Delta t}\int_\Omega \lvert \tilde q_h\rvert^2\, dx,
    \end{split}
    \end{equation}
    where we have used $\lambda\le 1$ and $\delta_3>0$ is to be chosen. Similarly, we have
    \begin{equation}
    \begin{split}
        &\frac{L}{2}\int_\Omega (\nabla \lambda^{-1}\widehat{Q}_h^{n+1} + \nabla \tilde q_h + \nabla Q_h^n):\nabla\lambda\widehat{Q}_h^{n+1}\, dx\\ 
        &\ge \frac{L}{4}\int_\Omega\left(\lvert\nabla\widehat{Q}_h^{n+1}\rvert^2 - \lvert\nabla Q_h^n\rvert^2\right)\, dx - \frac{L\delta_4}{4}\|\nabla\widehat{Q}_h^{n+1}\|_{L^2}^2 - \frac{L}{4\delta_4}\|\nabla\tilde q_h\|_{L^2}^2,
    \end{split}
    \end{equation}
    where $\delta_4>0$ is to be chosen. Now consider
    \[
        \int_\Omega \left(\frac{\partial\mathcal{F}_1(\widehat Q_h^{n+1} + \tilde q_h)}{\partial Q} - \frac{\partial\mathcal{F}_2(Q_h^n)}{\partial Q}\right):\lambda\widehat{Q}_h^{n+1}\, dx.
    \]
    For the first term, we have
    \begin{equation}
    \begin{split}
    &\lambda\int_\Omega \frac{\partial\mathcal{F}_1(\widehat{Q}_h^{n+1} + \tilde q_h)}{\partial Q}:\widehat{Q}_h^{n+1}\, dx\\
    &= \lambda\beta_1\|\widehat{Q}_h^{n+1}+\tilde q_h\|_{L^2}^2 - \lambda b\int_\Omega \operatorname{tr}((\widehat{Q}_h^{n+1} + \tilde q_h)^3)\, dx + \lambda\beta_2\int_\Omega \operatorname{tr}((\widehat{Q}_h^{n+1} + \tilde q_h)^2)^2\, dx\\
    & \hspace{10ex}-  \lambda\int_\Omega \Bigg(\beta_1(\widehat{Q}_h^{n+1} + \tilde q_h) - b\left((\widehat{Q}_h^{n+1} + \tilde q_h)^2 - \frac{1}{d}\operatorname{tr}((\widehat{Q}^{n+1} + \tilde q_h)^2)I\right)\\ 
    &\hspace{25ex}+ \beta_2\operatorname{tr}((\widehat{Q}_h^{n+1} + \tilde q_h)^2)(\widehat{Q}_h^{n+1}+\tilde q_h)\Bigg):\tilde q_h\, dx.
    \end{split}
    \end{equation}
    Then
    \begin{equation}
    \begin{split}
        &-\lambda b\int_\Omega \operatorname{tr}((\widehat{Q}_h^{n+1} + \tilde{q}_h)^3)\, dx + \lambda\beta_2\int_\Omega \operatorname{tr}((\widehat{Q}_h^{n+1} + \tilde{q}_h)^2)^2\, dx\\
        &\hspace{10ex}-  \lambda\int_\Omega \Bigg(\beta_1(\widehat{Q}_h^{n+1} + \tilde q_h) - b\left((\widehat{Q}_h^{n+1} + \tilde q_h)^2 - \frac{1}{d}\operatorname{tr}((\widehat{Q}^{n+1} + \tilde q_h)^2)I\right)\\ 
    &\hspace{25ex}+ \beta_2\operatorname{tr}((\widehat{Q}_h^{n+1} + \tilde q_h)^2)(\widehat{Q}_h^{n+1}+\tilde q_h)\Bigg):\tilde q_h\, dx\\
    &\ge \lambda\beta_2\int_\Omega \lvert \widehat{Q}_h^{n+1} + \tilde q_h\rvert^4\, dx - \lambda\lvert b\rvert\int_\Omega \lvert \widehat{Q}_h^{n+1} + \tilde q_h\rvert^3\, dx - \lambda\beta_1\|\tilde q_h\|_{L^\infty}\|\widehat{Q}_h^{n+1}+ \tilde q_h\|_{L^1}\\ 
    &\hspace{10ex} - \lambda\lvert b\rvert \|\tilde q_h\|_{L^\infty}\|(\widehat{Q}_h^{n+1}+\tilde q_h)^2\|_{L^1} - \lambda\beta_2\|\tilde q_h\|_{L^\infty}\|\operatorname{tr}((\widehat{Q}_h^{n+1} + \tilde q_h)^2)(\widehat{Q}_h^{n+1} + \tilde q_h)\|_{L^1}\\
    &\ge \lambda\int_\Omega \left(\beta_2\lvert\widehat{Q}_h^{n+1} + \tilde q_h\rvert^4 - C\lvert\widehat{Q}_h^{n+1} + \tilde q_h\rvert^3 - C\lvert\widehat{Q}_h^{n+1} + \tilde q_h\rvert^2 - C\lvert\widehat{Q}_h^{n+1} + \tilde q_h\rvert\right)\, dx\\
    &= \lambda\int_\Omega \left( \beta_2\lvert\widehat{Q}_h^{n+1} + \tilde q_h\rvert^4 - C\lvert\widehat{Q}_h^{n+1} + \tilde q_h\rvert^3 - C\lvert\widehat{Q}_h^{n+1} + \tilde q_h\rvert^2 - C\lvert\widehat{Q}_h^{n+1} + \tilde q_h\rvert\right)\chi_{\{\lvert \widehat{Q}_h^{n+1} + \tilde q_h\rvert > C_1\}}\, dx\\
    &\hspace{5ex}+ \lambda\int_\Omega \left( \beta_2\lvert\widehat{Q}_h^{n+1} + \tilde q_h\rvert^4 - C\lvert\widehat{Q}_h^{n+1} + \tilde q_h\rvert^3 - C\lvert\widehat{Q}_h^{n+1} + \tilde q_h\rvert^2 - C\lvert\widehat{Q}_h^{n+1} + \tilde q_h\rvert\right)\chi_{\{\lvert \widehat{Q}_h^{n+1} + \tilde q_h\rvert \le C_1\}}\, dx\\
    &\ge \lambda\int_\Omega \left( \beta_2\lvert\widehat{Q}_h^{n+1} + \tilde q_h\rvert^4 - C\lvert\widehat{Q}_h^{n+1} + \tilde q_h\rvert^3 - C\lvert\widehat{Q}_h^{n+1} + \tilde q_h\rvert^2 - C\lvert\widehat{Q}_h^{n+1} + \tilde q_h\rvert\right)\chi_{\{\lvert \widehat{Q}_h^{n+1} + \tilde q_h\rvert \le C_1\}}\, dx\\
    &\ge -C\lambda\lvert\Omega\rvert,
    \end{split}
    \end{equation}
    where $\chi_E$ is the characteristic function for the set $E$, $C_1 = \max\{1, 3C, (3C)^{1/2}, (3C)^{1/3}\}/\beta_2$, and $C>0$ is a constant which may change from line to line.
    Now for the second term, we have using Cauchy-Schwarz and Young's inequality,
    \begin{equation}
    \begin{split}
        -\lambda \int_\Omega \frac{\partial\mathcal{F}_2(Q_h^n)}{\partial Q}:\widehat{Q}_h^{n+1}\, dx &
        \ge -\frac{\delta_5}{2} \left\|\frac{\partial\mathcal{F}_2(Q_h^n)}{\partial Q}\right\|_{L^2}^2 - \frac{1}{2\delta_5}\left\|\widehat{Q}_h^{n+1}\right\|_{L^2}^2\\
        &\geq -\frac{\delta_5}{2} \left(\lvert\beta_1-a\rvert\|Q_h^n\|_{L^2} + \lvert\beta_2 - c\rvert\|Q_h^n\|_{L^6}^3\right)^2 - \frac{1}{2\delta_5}\left\|\widehat{Q}_h^{n+1}\right\|_{L^2}^2,
    \end{split}
    \end{equation}
    where we have used $\lambda\le 1$ and where $\delta_5>0$ is to be chosen. Here we note that since $Q_h^n\in H_0^1(\Omega)$, we have that $Q_h^n \in L^6(\Omega)$ by the Sobolev embedding theorem. Thus, together we have
    \begin{equation}
    \begin{split}
        &\int_\Omega \left(\frac{\partial\mathcal{F}_1(\widehat Q_h^{n+1}+\tilde q_h)}{\partial Q} - \frac{\partial\mathcal{F}_2(Q_h^n)}{\partial Q}\right):\lambda\widehat{Q}_h^{n+1}\, dx\\ 
        &\ge - C\lvert\Omega\rvert -\frac{\delta_5}{2} \left(\lvert\beta_1-a\rvert\|Q_h^n\|_{L^2} + \lvert\beta_2 - c\rvert\|Q_h^n\|_{L^6}^3\right)^2 - \frac{1}{2\delta_5}\left\|\widehat{Q}_h^{n+1}\right\|_{L^2}^2,
    \end{split}
    \end{equation}
    where we have again used $\lambda\le 1$. Now we have, using Cauchy-Schwarz and Young's inequality,
    \begin{align*}
            &-\frac{\varepsilon_2}{2}\int_\Omega \left(\tilde\P(\widehat{Q}_h^{n+1}+\tilde q_h,Q^n)\odot\nabla u_h^n(\lambda^{-1}\nabla\widehat{u}_h^{n+1})^\top\right):\lambda\widehat{Q}_h^{n+1}\, dx\\
            &= -\frac{\varepsilon_2}{2} \int_\Omega \left(\tilde\P(\widehat{Q}_h^{n+1}+\tilde q_h,Q^n)\odot\nabla u_h^n(\nabla\widehat{u}_h^{n+1})^\top\right):(\widehat{Q}_h^{n+1}+\tilde q_h-Q_h^n+Q_h^n - \tilde q_h)\, dx\\
            &= -\frac{\varepsilon_2}{2}\int_\Omega (\T_R(\widehat{Q}_h^{n+1}+\tilde q_h) - \T_R(Q_h^n)):\nabla u_h^n(\nabla\widehat{u}_h^{n+1})^\top\, dx\\ 
            &\hspace{10ex}- \frac{\varepsilon_2}{2}\int_\Omega (\tilde \P(\widehat{Q}_h^{n+1}+\tilde q_h,Q_h^n)\odot\nabla u_h^n(\nabla\widehat{u}_h^{n+1})^\top):(Q_h^n-\tilde q_h)\, dx\\
            &\ge -\frac{R^2\lvert\varepsilon_2\rvert}{4\delta_6}\|\nabla u_h^n\|_{L^2}^2 - \frac{\lvert\varepsilon_2\rvert\delta_6}{4}\|\nabla\widehat{u}_h^{n+1}\|_{L^2}^2\\
            &\hspace{10ex}-\frac{\lvert\varepsilon_2\rvert}{4\delta_6}\|(\tilde\P(\widehat{Q}_h^{n+1}+\tilde q_h,Q_h^n)\odot (Q_h^n-\tilde q_h)) \nabla u_h^n\|_{L^2}^2 - \frac{\lvert\varepsilon_2\rvert\delta_6}{4}\|\nabla\widehat{u}_h^{n+1}\|_{L^2}^2\\
            & \geq -\frac{R^2\lvert\varepsilon_2\rvert}{4\delta_6}\|\nabla u_h^n\|_{L^2}^2 - \frac{\lvert\varepsilon_2\rvert\delta_6}{4}\|\nabla\widehat{u}_h^{n+1}\|_{L^2}^2 -\frac{\lvert\varepsilon_2\rvert\delta_6}{4}\|\nabla\widehat{u}_h^{n+1}\|_{L^2}^2 -
            \frac{\lvert\varepsilon_2\rvert \CTR}{4\delta_6}\|(Q_h^n-\tilde q_h) \nabla u_h^n\|_{L^2}^2,
    \end{align*}
    where $\delta_6>0$ is yet to be chosen and we have the assumptions on $\mathcal{T}_R$,~\eqref{eq:conditionsonTr} in the last line.
    Now we have
    \begin{equation}
        \begin{split}
            &-\frac{\varepsilon_2}{2}\int_\Omega (\tilde\P(\widehat{Q}_h^{n+1}+\tilde q_h,Q_h^n)\odot\nabla u_h^n(\nabla\tilde g_h^{n+1})^\top):\lambda\widehat{Q}_h^{n+1}\, dx\\
              &\ge -\frac{\lvert\varepsilon_2\rvert\CTR}{4\delta_7}\|\nabla\tilde g_h^{n+1}\|_{L^\infty}^2\|\nabla u_h^n\|_{L^2}^2 - \frac{\lvert\varepsilon_2\rvert\delta_7}{4}\|\widehat{Q}_h^{n+1}\|_{L^2}^2,
        \end{split}
    \end{equation}
    where $\delta_7>0$ is to be chosen.
    Note that $\nabla\tilde g_h^{n+1}$ is bounded because $\tilde g_h^{n+1}$ is piecewise linear on each $K\in\mathfrak{T}_h$. Next, we have with Cauchy-Schwarz and Young's inequality,
    \begin{equation}
        \begin{split}
            &-\varepsilon_3\int_\Omega \frac{\nabla\lambda^{-1}\widehat{u}_h^{n+1}+\nabla u_h^n}{2}\cdot\operatorname{div}(\lambda\widehat{Q}_h^{n+1})\, dx\\ 
            &\ge -\frac{\lvert\varepsilon_3\rvert}{4}\|\nabla\widehat{u}_h^{n+1}\|_{L^2}^2 - \frac{\lvert\varepsilon_3\rvert d}{4}\|\nabla \widehat{Q}_h^{n+1}\|_{L^2}^2 - \frac{\lvert \varepsilon_3\rvert}{4}\|\nabla u_h^n\|_{L^2}^2 - \frac{\lvert\varepsilon_3\rvert\lambda d}{4}\|\nabla \widehat{Q}_h^{n+1}\|_{L^2}^2,
        \end{split}
    \end{equation}
    where we have used that $\|\operatorname{div}(\widehat{Q}_h^{n+1})\|_{L^2}^2 \le d\|\nabla\widehat{Q}_h^{n+1}\|_{L^2}^2$. Using this inequality again along with Cauchy-Schwarz and Young's inequality, we have
    \begin{equation}
        \begin{split}
            -\frac{\varepsilon_3}{2} \int_\Omega \nabla\tilde g_h^{n+1}\cdot\operatorname{div}(\lambda\widehat{Q}_h^{n+1})\, dx &\ge -\frac{\lvert\varepsilon_3\rvert}{4\delta_8}\|\nabla\tilde g_h^{n+1}\|_{L^2}^2 - \frac{\lvert\varepsilon_3\rvert d\delta_8}{4}\|\nabla\widehat{Q}_h^{n+1}\|_{L^2}^2,
        \end{split}
    \end{equation}
    where $\delta_8>0$ is yet to be chosen. The remaining terms in (\ref{eq:fixed_point_L_map}) are zero because $\widehat{Q}_h^{n+1}$ is trace-free. Thus, combining all of these inequalities gives
    \begin{equation}
        \begin{split}
            &-\frac{1}{2\Delta t}\int_\Omega \left(\lvert\widehat{Q}_h^{n+1}\rvert^2 - \lvert Q_h^n\rvert^2\right)\, dx + \frac{\delta_3}{2\Delta t}\int_\Omega \lvert\widehat{Q}_h^{n+1}\rvert^2\, dx + \frac{1}{2\delta_3 \Delta t}\int_\Omega \lvert \tilde q_h\rvert^2\, dx\\
            &\ge \frac{L}{4}\int_\Omega \left(\lvert \nabla\widehat{Q}_h^{n+1}\rvert^2 - \lvert\nabla Q_h^n\rvert^2\right)\, dx - \frac{L\delta_4}{4}\|\nabla\widehat{Q}_h^{n+1}\|_{L^2}^2 - \frac{L}{4\delta_4}\|\nabla\tilde q_h\|_{L^2}^2 -C\lvert\Omega\rvert\\ 
            &\hspace{10ex} - \frac{\delta_5}{2}\left(\lvert\beta_1-a\rvert\|Q_h^n\|_{L^2} + \lvert\beta_2-c\rvert\|Q_h^n\|_{L^6}^3\right)^2 - \frac{1}{2\delta_5}\|\widehat{Q}_h^{n+1}\|_{L^2}^2 - \frac{R^2|\varepsilon_2|}{4\delta_6}\|\nabla u_h^n\|_{L^2}^2\\ 
            &\hspace{10ex}- \frac{|\varepsilon_2|\delta_6}{4}\|\nabla \widehat{u}_h^{n+1}\|_{L^2}^2 - \frac{|\varepsilon_2|\CTR}{4\delta_6}\|(Q_h^n - \tilde q_h)\nabla u_h^n\|_{L^2}^2 - \frac{|\varepsilon_2|\delta_6}{4}\|\nabla\widehat{u}_h^{n+1}\|_{L^2}^2\\
            &\hspace{10ex} - \frac{|\varepsilon_2|\CTR}{4\delta_7}\|\nabla\tilde g_h^{n+1}\|_{L^\infty}^2\|\nabla u_h^n\|_{L^2}^2 - \frac{|\varepsilon_2|\delta_7}{4}\|\widehat{Q}_h^{n+1}\|_{L^2} - \frac{|\varepsilon_3|}{4}\|\nabla\widehat{u}_h^{n+1}\|_{L^2}^2 - \frac{|\varepsilon_3|d}{4}\|\nabla\widehat{Q}_h^{n+1}\|_{L^2}^2\\ 
            &\hspace{10ex} - \frac{|\varepsilon_3|}{4}\|\nabla u_h^n\|_{L^2}^2 - \frac{|\varepsilon_3|d}{4}\|\nabla \widehat{Q}_h^{n+1}\|_{L^2}^2 - \frac{|\varepsilon_3|}{4\delta_8}\|\nabla\tilde g_h^{n+1}\|_{L^2}^2 - \frac{|\varepsilon_3|d\delta_8}{4}\|\nabla \widehat{Q}_h^{n+1}\|_{L^2}^2.
        \end{split}
    \end{equation}
    Adding $1/2$ times \eqref{eq:leray_schauder_boundedness_u_inequality} to this and combining like terms, we obtain
    \begin{equation}
        \begin{split}
            &\left(\frac{1}{2\Delta t} - \frac{\delta_3}{2\Delta t} - \frac{1}{2\delta_5} - \frac{|\varepsilon_2|\delta_7}{4}\right)\|\widehat{Q}_h^{n+1}\|_{L^2}^2\\ 
            &\hspace{2ex}+ \left(\frac{L}{4} - \frac{L\delta_4}{4} - \frac{2|\varepsilon_3|d}{4} - \frac{|\varepsilon_3|d\delta_8}{4} - \frac{d|\varepsilon_3|}{2}\right)\|\nabla\widehat{Q}_h^{n+1}\|_{L^2}^2\\
            &\hspace{2ex}+ \left(\frac{\varepsilon_1}{2} - \frac{R|\varepsilon_2|}{2} - \frac{\varepsilon_1\delta_1}{4} - \frac{|\varepsilon_2|\delta_2}{4} - \frac{2|\varepsilon_3|}{4} - \frac{2|\varepsilon_2|\delta_6}{4}\right)\|\nabla\widehat{u}_h^{n+1}\|_{L^2}^2\\
            &\le \frac{1}{2\Delta t}\|Q_h^n\|_{L^2}^2 + \frac{L}{4}\|\nabla Q_h^n\|_{L^2}^2 + \frac{\delta_5}{2}\left(|\beta_1 - a|\cdot\|Q_h^n\|_{L^2} + |\beta_2 - c|\cdot\|Q_h^n\|_{L^6}^3\right)^2\\
            &\hspace{2ex}+ \frac{|\varepsilon_2|\CTR}{4\delta_6}\|(Q_h^n - \tilde q_h)\nabla u_h^n\|_{L^2}^2 + \left(\frac{R^2 |\varepsilon_2|}{4\delta_6} + \frac{|\varepsilon_2|\CTR}{4\delta_7}\|\nabla\tilde g_h^{n+1}\|_{L^\infty}^2 + \frac{|\varepsilon_3|}{4}\right)\|\nabla u_h^n\|_{L^2}^2\\
            &\hspace{2ex}+ \left(\frac{|\varepsilon_3|}{4\delta_8} + \frac{\varepsilon_1}{4\delta_1} + \frac{|\varepsilon_2|R^2}{4\delta_2}\right)\|\nabla\tilde g_h^{n+1}\|_{L^2}^2 + \left(\frac{1}{2\delta_3\Delta t} + \frac{L}{4\delta_4} + \frac{d|\varepsilon_3|}{2}\right)\|\tilde q_h\|_{H^1}^2 + C|\Omega|.
        \end{split}
    \end{equation}
    Take $\delta_3 = 1/4$, $\delta_5 = 4\Delta t$, and $\delta_7 = (2(1+|\varepsilon_2|)\Delta t)^{-1}$. Further take $\delta_4 = (L/4 - 3|\varepsilon_3|)/L$ and $\delta_8 = (L/4 - 3|\varepsilon_3|)/(3(1+|\varepsilon_3|))$. Also take $\delta_1 = (3+3\varepsilon_1)^{-1}\cdot(\varepsilon_1 - R|\varepsilon_2|-|\varepsilon_3|)$, $\delta_2 = (3+3|\varepsilon_2|)^{-1}(\varepsilon_1 - R|\varepsilon_2|-|\varepsilon_3|)$, and $\delta_6 = (6+6|\varepsilon_2|)^{-1}(\varepsilon_1 - R|\varepsilon_2| - |\varepsilon_3|)$. Now by our assumptions on $\varepsilon_1, \varepsilon_2, \varepsilon_3$, we have $|\varepsilon_3| < L/12$ and $\varepsilon_1 > R|\varepsilon_2| + |\varepsilon_3|$. Thus we see that $\|\nabla\widehat{u}_h^{n+1}\|_{L^2}^2$, $\|\nabla \widehat Q_h^{n+1}\|_{L^2}^2$, and $\|\widehat{Q}_h^{n+1}\|_{L^2}^2$ are bounded uniformly in $\lambda$ and the previous time data, as desired.

    \item[--] \textbf{$\mathcal{L}$ maps bounded sets into bounded sets.} Suppose $\|\tilde Q_h^{n+1}\|_{H^1(\Omega)}, \|\tilde u_h^{n+1}\|_{H^1(\Omega)} < C$. Then letting $Q_h^{n+1} = \tilde Q_h^{n+1} + \tilde q_h$ and taking $\psi_h = \widehat{u}_h^{n+1}$ in~\eqref{eq:fixed_point_elliptic_definition}, we have
    \begin{align*}
        &\varepsilon_1\|\nabla\widehat{u}_h^{n+1}\|_{L^2}^2 + \varepsilon_2\int_\Omega (\nabla\widehat{u}_h^{n+1})^\top\T_R(Q_h^{n+1})\nabla\widehat{u}_h^{n+1}\, dx\\
        &\le \varepsilon_1\|\nabla\tilde g_h^{n+1}\|_{L^2}\|\nabla\widehat{u}_h^{n+1}\|_{L^2} + \lvert\varepsilon_2\rvert\|\T_R(Q_h^{n+1})\nabla\tilde g_h^{n+1}\|_{L^2}\|\nabla\widehat{u}_h^{n+1}\|_{L^2} + \lvert\varepsilon_3\rvert\|\operatorname{div}Q_h^{n+1}\|_{L^2}\|\nabla\widehat{u}_h^{n+1}\|_{L^2}
    \end{align*}
    which implies
    \[
        (\varepsilon_1 - \lvert\varepsilon_2\rvert R)\|\nabla\widehat{u}_h^{n+1}\|_{L^2} \le \varepsilon_1\|\nabla\tilde g_h^{n+1}\|_{L^2} + \lvert\varepsilon_2\rvert\|\T_R(Q_h^{n+1})\nabla\tilde g_h^{n+1}\|_{L^2} + \lvert\varepsilon_3\rvert\|\operatorname{div}Q_h^{n+1}\|_{L^2}.
    \]
    The right hand side is bounded since $\|\tilde Q_h^{n+1}\|_{H^1(\Omega)} \le C$ and $\|\tilde q_h\|_{H^1(\Omega)} < \infty$, so we see by the Poincar\'e inequality that $\|\widehat{u}_h^{n+1}\|_{H^1(\Omega)} \le C'$.
    
    Now note that $\tilde \P$ is bounded and $Q_h^n, \tilde Q_h^{n+1}, u_h^n,  \widehat{u}_h^{n+1}, \tilde q_h, \tilde g_h^{n+1}$ are piecewise linear and continuous. Thus, taking $\Phi_h = \widehat{Q}_h^{n+1}$ in~\eqref{eq:fixed_point_Q_equation} and applying Cauchy-Schwarz, we obtain
    \[
        \frac{1}{\Delta t}\|\widehat{Q}_h^{n+1}\|_{L^2}^2 + \frac{L}{2}\|\nabla\widehat{Q}_h^{n+1}\|_{L^2}^2 \le C(\|\widehat{Q}_h^{n+1}\|_{L^2} + \|\nabla\widehat{Q}_h^{n+1}\|_{L^2} + \|\operatorname{div}\widehat{Q}_h^{n+1}\|_{L^2}).
    \]
    It follows that
    \[
        C_1\|\widehat{Q}_h^{n+1}\|_{H^1}^2 \le C_2\|\widehat{Q}_h^{n+1}\|_{H^1}
    \]
    so that $\widehat{Q}_h^{n+1}$ is bounded in $H^1(\Omega)$, as desired.

    \item[--] \textbf{Continuity.} We now show that the operator $\mathcal{L}$ is continuous. Because $\mathcal{L}$ is independent of the input $\tilde u_h^{n+1}$, it suffices to show that $\mathcal{L}$ is continuous with respect to $\tilde Q_h^{n+1}$. So let $\{\tilde Q_k\}_k \subseteq \mathbb{X}_h$ be such that $\tilde Q_k \to \tilde Q$ in $H_0^1(\Omega)$. For each $k\in\mathbb{N}$, since $\mathcal{L}$ is well-posed, there are $(\widehat{Q}_k,\widehat{u}_k) = \mathcal{L}(\tilde Q_k,\tilde u_k)$, where $\tilde u_k\in \mathbb{Y}_h$ is arbitrary (since $\mathcal{L}$ is independent of this input). Let $(\widehat{Q}, \widehat{u}) = \mathcal{L}(Q, \tilde u)$ with $\tilde u \in \mathbb{Y}_h$ again arbitrary. Then to show $\mathcal{L}$ is continuous, it suffices to show that $\widehat{Q}_k\to \widehat{Q}$ and $\widehat{u}_k\to\widehat{u}$ in $H^1(\Omega)$.
    
    First, because $\mathcal{L}$ maps bounded sets into bounded sets, we have that $\nabla \widehat{u}_k$ is uniformly bounded in $L^2(\Omega)$ because the sequence $\{\tilde Q_k\}_k$ is bounded in $H^1(\Omega)$. Now subtract the elliptic equations corresponding to $\nabla\widehat{u}_k$ and to $\nabla\widehat{u}$ to obtain
    \begin{align*}
        &\int_\Omega (\varepsilon_1(\nabla\widehat{u}_k - \nabla\widehat{u}) + \varepsilon_2(\T_R(\tilde Q_k + \tilde q_h)\nabla\widehat{u}_k - \T_R(\tilde Q + \tilde q_h)\nabla\widehat{u}))\cdot\nabla\psi_h\, dx\\
        &= -\int_\Omega (\varepsilon_2(\T_R(\tilde Q_k + \tilde q_h) - \T_R(\tilde Q + \tilde q_h))\nabla\tilde g_h^{n+1} + \varepsilon_3(\operatorname{div}(\tilde Q_k + \tilde q_h) - \operatorname{div}(\tilde Q + \tilde q_h)))\cdot\nabla\psi_h\, dx.
    \end{align*}
    Then take $\psi_h = \nabla\widehat{u}_k - \nabla\widehat{u}$. We get
    \begin{align*}
        &\varepsilon_1\|\nabla\widehat{u}_k - \nabla\widehat{u}\|_{L^2}^2 + \varepsilon_2\underbrace{\int_\Omega (\nabla\widehat{u}_k-\nabla\widehat{u})^\top\T_R(\tilde Q_k + \tilde q_h)(\nabla\widehat{u}_k-\nabla\widehat{u})\, dx}_{I}\\ 
                &=                - \varepsilon_2 \underbrace{\int(\nabla\widehat{u}_k-\nabla\widehat{u})^\top (\T_R(\tilde Q_k + \tilde q_h)-\T_R(\tilde Q + \tilde q_h))\nabla\widehat{u}\, dx}_{II}\\
                &\quad  -\underbrace{\int_\Omega (\varepsilon_2(\T_R(\tilde Q_k + \tilde q_h) - \T_R(\tilde Q + \tilde q_h))\nabla\tilde g_h^{n+1} + \varepsilon_3(\operatorname{div}(\tilde Q_k + \tilde q_h) - \operatorname{div}(\tilde Q + \tilde q_h)))\cdot(\nabla\widehat{u}_k-\nabla\widehat{u})\, dx}_{III}.
    \end{align*}
    Note that $I$ is bounded by $\pm R\|\nabla\widehat{u}_k-\nabla\widehat{u}\|_{L^2}^2$. For $II$, because $\widehat{u} \in \mathbb{Y}_h$ (as $\mathbb{Y}_h$ is finite dimensional and thus closed), we have that $\nabla \widehat{u}$ is essentially bounded. Combining this with the fact that $\T_R$ is Lipschitz, that $\nabla\widehat{u}_k - \nabla \widehat{u}$ is uniformly bounded in $L^2(\Omega)$, and $\tilde Q_k\to \tilde Q$ in $L^2(\Omega)$, we have that $II\to 0$ as $k\to\infty$.
     For term $III$, using Cauchy-Schwarz, we obtain an upper bound of
    \begin{multline*}
      |III|\leq  \lvert\varepsilon_2\rvert \|(\T_R(\tilde Q_k + \tilde q_h) - \T_R(\tilde Q + \tilde q_h))\nabla\tilde g_h^{n+1}\|_{L^2}\|\nabla\widehat{u}_k - \nabla\widehat{u}\|_{L^2} \\
      + \lvert\varepsilon_3\rvert\|\operatorname{div}(\tilde Q_k + \tilde q_h) - \operatorname{div}(\tilde Q + \tilde q_h)\|_{L^2}\|\nabla\widehat{u}_k - \nabla \widehat{u}\|_{L^2}.
    \end{multline*}
    Using the fact that $\tilde g_h^{n+1}\in\mathbb{Y}_h$, we have that $\nabla\tilde g_h^{n+1}$ is essentially bounded. Combining this with the fact that $\T_R$ is Lipschitz, that $\tilde Q_k\to \tilde Q$ in $H^1(\Omega)$, and the fact that $\nabla\widehat{u}_k-\nabla\widehat{u}$ is uniformly bounded in $L^2(\Omega)$, we have that this upper bound goes to $0$ as $k\to\infty$.
    Thus we have
    \begin{equation*}
    \lim_{k\to \infty} (\varepsilon_1-R|\varepsilon_2|)\norm{\Grad \widehat{u}_k-\Grad\widehat{u}}_{L^2}^2\leq \lim_{k\to \infty} (|II|+|III|) = 0,
    \end{equation*}
    which implies that $\Grad\widehat{u}_k\to \Grad \widehat{u}$ as $k\to\infty$ since we are assuming that $\varepsilon_1>R|\varepsilon_2|$. 
By the Poincar\'e inequality, $\widehat{u}_k\to \widehat{u}$ in $H^1(\Omega)$, as desired.
    
    We must now show that $\widehat{Q}_k\to \widehat{Q}$ as $k\to\infty$. We have that $\{\widehat{Q}_k\}_k$ is uniformly bounded in $H^1(\Omega)$, as before. Now subtracting the equation in the operator defining $\widehat{Q}_k$ and $\widehat{Q}$, we obtain
    \begin{align*}
        &-\int_\Omega \frac{\widehat{Q}_k - \widehat{Q}}{\Delta t}:\Phi_h\, dx\\
        &= L\int_\Omega \frac{\nabla\widehat{Q}_k - \nabla\widehat{Q}}{2}:\nabla\Phi_h\, dx + \int_\Omega \left(\frac{\partial\mathcal{F}_1(\tilde Q_k + \tilde q_h)}{\partial Q} - \frac{\partial\mathcal{F}_1(\tilde Q + \tilde q_h)}{\partial Q}\right):\Phi_h\, dx\\
        &- \frac{\varepsilon_2}{2}\int_\Omega \left((\tilde \P(\tilde Q_k + \tilde q_h,Q_h^n)\odot \nabla u_h^n(\nabla\widehat{u}_k)^\top)_S - \frac1d\operatorname{tr}(\tilde \P(\tilde Q_k + \tilde q_h,Q_h^n)\odot \nabla u_h^n(\nabla\widehat{u}_k)^\top)I\right):\Phi_h\, dx\\
        &+ \frac{\varepsilon_2}{2}\int_\Omega \left((\tilde \P(\tilde Q + \tilde q_h,Q_h^n)\odot \nabla u_h^n(\nabla\widehat{u})^\top)_S - \frac1d\operatorname{tr}(\tilde \P(\tilde Q + \tilde q_h,Q_h^n)\odot \nabla u_h^n(\nabla\widehat{u})^\top)I\right):\Phi_h\, dx\\
        &- \frac{\varepsilon_2}{2}\int_\Omega \left((\tilde \P(\tilde Q_k + \tilde q_h,Q_h^n)\odot \nabla u_h^n(\nabla\tilde g_h^{n+1})^\top)_S - \frac1d\operatorname{tr}(\tilde \P(\tilde Q_k + \tilde q_h,Q_h^n)\odot \nabla u_h^n(\nabla\tilde g_h^{n+1})^\top)I\right):\Phi_h\, dx\\
        &+ \frac{\varepsilon_2}{2}\int_\Omega \left((\tilde \P(\tilde Q + \tilde q_h,Q_h^n)\odot \nabla u_h^n(\nabla\tilde g_h^{n+1})^\top)_S - \frac1d\operatorname{tr}(\tilde \P(\tilde Q + \tilde q_h,Q_h^n)\odot \nabla u_h^n(\nabla\tilde g_h^{n+1})^\top)I\right):\Phi_h\, dx\\
        &- \varepsilon_3 \int_\Omega \left(\frac{\nabla\widehat{u}_k - \nabla \widehat{u}}{2}\cdot\operatorname{div}((\Phi_h)_S) - \frac1d \frac{\nabla\widehat{u}_k - \nabla\widehat{u}}{2}\cdot\nabla\operatorname{tr}\Phi_h\right)\, dx.
    \end{align*}
    Now take $\Phi_h = \widehat{Q}_k - \widehat{Q}$. Using that $\widehat{Q}-\widehat{Q}_k$ is trace-free, all terms involving traces of $\Phi_h$ cancel. By Cauchy-Schwarz, we then have
    \begin{align*}
        &\frac{1}{\Delta t}\|\widehat{Q}_k - \widehat{Q}\|_{L^2}^2 + \frac{L}{2}\|\nabla\widehat{Q}_k - \nabla\widehat{Q}\|_{L^2}^2\\ 
        &\le \underbrace{\left\|\frac{\partial\mathcal{F}_1(\tilde Q_k + \tilde q_h)}{\partial Q} - \frac{\partial\mathcal{F}_1(\tilde Q + \tilde q_h)}{\partial Q}\right\|_{L^2}\|\widehat{Q}_k - \widehat{Q}\|_{L^2}}_{I}\\
        &+ \frac{\lvert\varepsilon_2\rvert}{2}\underbrace{\left\|(\tilde \P(\tilde Q_k + \tilde q_h,Q_h^n)\odot \nabla u_h^n(\nabla\widehat{u}_k)^\top)_S - (\tilde \P(\tilde Q + \tilde q_h,Q_h^n)\odot \nabla u_h^n(\nabla\widehat{u})^\top)_S\right\|_{L^2}}_{II}\|\widehat{Q}_k - \widehat{Q}\|_{L^2}\\
        &+ \frac{\lvert\varepsilon_2\rvert}{2}\underbrace{\left\|(\tilde \P(\tilde Q_k + \tilde q_h,Q_h^n)\odot \nabla u_h^n(\nabla\tilde g_h^{n+1})^\top)_S - (\tilde \P(\tilde Q + \tilde q_h,Q_h^n)\odot \nabla u_h^n(\nabla\tilde g_h^{n+1})^\top)_S\right\|_{L^2}}_{III}\|\widehat{Q}_k - \widehat{Q}\|_{L^2}\\
        &+ \frac{\lvert\varepsilon_3\rvert}{2}\underbrace{\left\|\nabla\widehat{u}_k - \nabla\widehat{u}\right\|_{L^2}\|\operatorname{div}(\widehat{Q}_k - \widehat{Q})\|_{L^2}}_{IV}.
    \end{align*}
    Now in what follows, we repeatedly use the fact that all of the norms above involving $\widehat{Q}_k$ and $\widehat{Q}$ are bounded by a constant uniformly in $k$ since $\widehat{Q}_k$ are uniformly bounded in $H^1(\Omega)$. With this fact and the fact that $\tilde Q_k\to \tilde Q$ in $H^1(\Omega)$ and thus also in $L^6(\Omega)$ by the Sobolev embedding theorem, we have
    \[
      I=  \left\|\frac{\partial\mathcal{F}_1(\tilde Q_k + \tilde q_h)}{\partial Q} - \frac{\partial\mathcal{F}_1(\tilde Q + \tilde q_h)}{\partial Q}\right\|_{L^2}\|\widehat{Q}_k - \widehat{Q}\|_{L^2} \to 0
    \]
    as $k\to\infty$. To bound $II$, we have
    \begin{align*}
        &\left\|(\tilde \P(\tilde Q_k + \tilde q_h,Q_h^n)\odot \nabla u_h^n(\nabla\widehat{u}_k)^\top)_S - (\tilde \P(\tilde Q + \tilde q_h,Q_h^n)\odot \nabla u_h^n(\nabla\widehat{u})^\top)_S\right\|_{L^2}\\
        &\le \left\|\tilde \P(\tilde Q_k + \tilde q_h,Q_h^n)\odot \nabla u_h^n(\nabla\widehat{u}_k)^\top - \tilde \P(\tilde Q_k + \tilde q_h,Q_h^n)\odot \nabla u_h^n(\nabla\widehat{u})^\top\right\|_{L^2}\\ 
        &\hspace{10ex}+ \left\|\tilde \P(\tilde Q_k + \tilde q_h,Q_h^n)\odot \nabla u_h^n(\nabla\widehat{u})^\top - \tilde \P(\tilde Q + \tilde q_h,Q_h^n)\odot \nabla u_h^n(\nabla\widehat{u})^\top\right\|_{L^2}\\
        &\le C\|\nabla u_h^n(\nabla\widehat{u}_k - \nabla \widehat{u})^\top\|_{L^2} + \|(\tilde\P(\tilde Q_k + \tilde q_h,Q_h^n) - \tilde\P(\tilde Q + \tilde q_h,Q_h^n))\odot \nabla u_h^n(\nabla\widehat{u})^\top\|_{L^2},
    \end{align*}
    where in the last inequality we used that $\tilde\P$ is bounded. Now use the fact that $\nabla u_h^n$ and $\nabla\widehat{u}$ are bounded to obtain the upper bound
    \begin{align*}
        &\left\|(\tilde \P(\tilde Q_k + \tilde q_h,Q_h^n)\odot \nabla u_h^n(\nabla\widehat{u}_k)^\top)_S - (\tilde \P(\tilde Q + \tilde q_h,Q_h^n)\odot \nabla u_h^n(\nabla\widehat{u})^\top)_S\right\|_{L^2} \\
        &\le C\|\nabla\widehat{u}_k - \nabla\widehat{u}\|_{L^2} + C\|\tilde\P(\tilde Q_k + \tilde q_h,Q_h^n) - \tilde\P(\tilde Q + \tilde q_h,Q_h^n)\|_{L^2}\\
        &\le C\|\nabla\widehat{u}_k - \nabla\widehat{u}\|_{L^2} + C'\|\tilde Q_k - \tilde Q\|_{L^2},
    \end{align*}
    where we have applied the Lipschitz property of $\tilde\P$. Then since $\widehat{u}_k \to \widehat{u}$ in $H^1(\Omega)$ as previously shown and $\tilde Q_k\to \tilde Q$ in $H^1(\Omega)$, we have $II\to 0$
  as $k\to\infty$. By similar computations, we have $III\to 0$ as $k\to\infty$. 
    We also see that
    \[
    IV= \frac{\lvert\varepsilon_3\rvert}{2}\left\|\nabla\widehat{u}_k - \nabla\widehat{u}\right\|_{L^2}\|\operatorname{div}(\widehat{Q}_k - \widehat{Q})\|_{L^2}
    \to 0
    \]
    since $\nabla\widehat{u}_k \to \nabla\widehat{u}$ in $L^2$, as previously shown.
      Combining all of the limits we have computed gives
    \[
        \frac{1}{\Delta t}\|\widehat{Q}_k - \widehat{Q}\|_{L^2}^2 + \frac{L}{2}\|\nabla\widehat{Q}_k - \nabla\widehat{Q}\|_{L^2}^2 \to 0
    \]
    as $k\to\infty$. In particular, $\widehat{Q}_k \to \widehat{Q}$ in $H^1(\Omega)$, and as previously computed, we also have $\widehat{u}_k\to \widehat{u}$. This shows that $\mathcal{L}$ is continuous, as desired.
    
    \item[--] \textbf{Compactness} Compactness of the operator $\mathcal{L}$ follows since it is continuous and maps bounded sets into bounded sets in finite dimensional space.
\end{itemize}
Thus, we can apply the Leray-Schauder fixed point theorem to conclude the proof.
\end{proof}

\begin{remark}[General elastic constants]
	In general, the elastic energy density takes the form 
	\begin{equation*}
		\mathcal{F}_E(Q) = \frac{L_1}{2}|\Grad Q|^2 + \frac{L_2}{2}|\Div Q|^2 +\frac{L_3}{2}\sum_{i,j,k=1}^d \partial_i Q_{jk}\partial_k Q_{ji}.
	\end{equation*}
	The additional terms are quadratic and nonnegative.
	The variational derivative for this energy term is
	\begin{equation*}
		\left(\frac{\partial \mathcal{F}_E(Q)}{\partial Q}\right)_{ij} = L_1 \Delta Q_{ij} + \frac{L_2+L_3}{2} \left(\sum_{k=1}^d\left( \partial_{ik}Q_{jk} + \partial_{jk}Q_{ik}\right) - \frac{2}{d} \sum_{k,s=1}^d \partial_{ks}Q_{ks}\, \delta_{ij}\right).
	\end{equation*}
	We note that the terms related to the elastic constants $L_2$ and $L_3$ are linear. Therefore, they do not cause any additional difficulty in the discretization and stability and well-posedness analysis of the scheme. In fact, one could for example discretize them as in~\cite{Hirsch2024}, where stability and convergence of a liquid crystal model with inertia is proved. Using the discretization proposed there, one would obtain stability, well-posedness and convergence (as proved in the upcoming section), for this model too. We chose to omit the $L_2$ and $L_3$-term to improve the readability of the paper, since the treatment is very similar to the $L_1$-term.	
\end{remark}
\section{Convergence}\label{sec:conv}
We now show that up to a subsequence, the scheme \eqref{eq:fully_discrete_scheme_definition} converges to a weak solution of \eqref{eq:weak_formulation} when $\varepsilon_3=0$ and under some conditions for $\varepsilon_1$ and $\varepsilon_2$ (that were needed for stability and well-posedness in the previous section). To do so, we define piecewise constant interpolants in time:
\begin{subequations}
    \begin{equation}
        Q_{h,\Delta t}(t, x) = \sum_{n=0}^{N-1} Q_{h}^n(x)\chi_{S_n}(t),
    \end{equation}
    \begin{equation}
        u_{h,\Delta t}(t, x) = \sum_{n=0}^{N-1} u_{h}^n(x)\chi_{S_n}(t),
    \end{equation}
\end{subequations}
where $\chi_A$ is the characteristic function of a set $A$ and $S_n = [n\Delta t, (n+1)\Delta t)$. We will need the following version of the Aubin-Lions-Simon lemma to conclude pre-compactness of $\{Q_{h,\Delta t}\}$:

\begin{lemma}[Pre-Compactness of $\{Q_{h,\Delta t}\}$]
\label{conj:precompactness}
Assume that $\{Q_{h,\Delta t}\}\subset L^\infty(0,T; H^1(\Omega))$ and $\{D_t^+ Q_{h,\Delta t}\}\subset L^2([0,T]\times\Omega)$ uniformly in $\Delta t, h>0$ sufficiently small. Then the sequence $\{Q_{h,\Delta t}\}$ is precompact in $L^2([0,T]\times\Omega)$, that is,
\[
    Q_{h,\Delta t} \to Q\quad\text{in } L^2([0,T]\times\Omega)
\]
up to a subsequence.
\end{lemma}
\begin{proof}
We define piecewise linear interpolations in time:
\[
    \widehat{Q}_{h,\Delta t}(t, x) = \sum_{n=0}^{N-1} \chi_{S_n}(t)\left(Q_h^n(x) + \frac{Q_h^{n+1}(x) - Q_h^n(x)}{\Delta t}(t - n\Delta t)\right),
\]
where $S_n=[n\Delta t,(n+1)\Delta t)$. Then if $0 < t_1 < t_2 < T$, we have (omitting spatial dependence)
\begin{align*}
    \left\|\int_{t_1}^{t_2} \widehat{Q}_{h,\Delta t}(t)\, dt\right\|_{H^1(\Omega)} \le \int_{t_1}^{t_2} \|\widehat{Q}_{h,\Delta t}(t)\|_{H^1(\Omega)}\, dt \le 3\Delta t\sum_{n=0}^{N}\|Q_h^n\|_{H^1(\Omega)} \le C,
\end{align*}
where $C>0$ is some constant and where the last inequality follows since $\{Q_{h,\Delta t}\}\subset L^\infty(0,T;H^1(\Omega))$. Now $H^1(\Omega)$ is compactly embedded in $L^2(\Omega)$, so the set $\left\{\int_{t_1}^{t_2} \widehat{Q}_{h,\Delta t}(t)\, dt\right\}_{(h,\Delta t)}$ is relatively compact in $L^2(\Omega)$ for all $0<t_1<t_2<T$. 
Now 
\begin{align*}
    \int_0^{T-\Delta t} \lvert \widehat{Q}_{h,\Delta t}(t+\Delta t) - \widehat{Q}_{h,\Delta t}(t)\rvert^2\, dt
    &\le \int_0^{T-\Delta t} \left(\int_{t}^{t+\Delta t} \lvert \partial_s\widehat{Q}_{h,\Delta t}(s)\rvert\, ds\right)^2\, dt\\
    &\le \Delta t\int_0^{T-\Delta t} \int_t^{t+\Delta t}\lvert \partial_s \widehat{Q}_{h,\Delta t}(s)\rvert^2\, ds\, dt\\
    &\le C\Delta t,
\end{align*}
where the second inequality is H\"{o}lder's inequality, and the last inequality follows since $\{D_t^+ Q_{h,\Delta t}\}\subset L^2([0,T]\times\Omega)$ uniformly in $\Delta t, h>0$. Thus, we have that $\| \widehat{Q}_{h,\Delta t}(\cdot+\Delta t) - \widehat{Q}_{h,\Delta t}(\cdot)\|_{L^2(0,T-\Delta t; L^2(\Omega))} \to 0$ as $\Delta t\to 0$, uniformly in $h>0$. It follows from Theorem 1 in~\cite{Simon1986} that the family $\{\widehat{Q}_{h,\Delta t}\}_{(h,\Delta t)}$ is pre-compact in $L^2([0,T]\times\Omega)$. 

We now show that $\widehat{Q}_{h,\Delta t}$ and $Q_{h,\Delta t}$ have the same limit. We have
\begin{align*}
    \left\|\widehat{Q}_{h,\Delta t} - Q_{h,\Delta t}\right\|_{L^2([0,T]\times\Omega)}
        &= \left\|\sum_{n=0}^{N-1}\chi_{S_n}(t)\frac{Q_h^{n+1}(x) - Q_h^n(x)}{\Delta t}(t - n\Delta t)\right\|_{L^2([0,T]\times\Omega)}\\
        &\le \Delta t\left(\int_0^T\int_\Omega \left\lvert \sum_{n=0}^{N-1}\chi_{S_n} D_t^+ Q_h^n\right\rvert^2\, dx\, dt\right)^{1/2}\\
        &= \Delta t\left(\Delta t\sum_{n=0}^{N-1}\int_\Omega \lvert D_t^+ Q_h^n\rvert^2\, dx\right)^{1/2}\\
        &\le C\Delta t,
\end{align*}
where we have used Corollary~\ref{corollary:fully_discrete_DtQ_L2_bound}. So as $\Delta t\to 0$, we see that $\widehat{Q}_{h,\Delta t}$ and $Q_{h,\Delta t}$ have the same limits. Thus, $\{Q_{h,\Delta t}\}_{(h,\Delta t)}$ is precompact in $L^2([0,T]\times\Omega)$, as desired.
\end{proof}

Using this lemma and assuming some restrictions on $\varepsilon_1,\varepsilon_2,$ and $\varepsilon_3$, we now show that the scheme converges to a weak solution as the mesh is refined.

\begin{theorem}[Convergence to Weak Solutions]
Assume that  $\Delta t< \frac{\varepsilon_1 - |\varepsilon_2|R}{3(\varepsilon_1+|\varepsilon_2|R^2)}$, that $\varepsilon_1 > \lvert\varepsilon_2\rvert R$, and $\varepsilon_3=0$. Then the sequence $\{(Q_{h,\Delta t}, u_{h,\Delta t})\}$ of approximations computed by the scheme~\eqref{eq:fully_discrete_scheme_definition} converge up to a subsequence to a weak solution of~\eqref{eq:weak_formulation}.
\end{theorem}
\begin{proof}
	We divide the proof into several steps. We start by using the energy estimate to get precompactness of the approximating sequences. For $\{u_h\}$ this only implies weak compactness in $L^\infty([0,T]; H^1(\Omega))$. However, because the gradient of $u_h$ appears nonlinearly in the equations, we need to derive strong compactness in $L^2([0,T]\times\Omega)$. We do this as a next step by using the structure of the elliptic equation. Finally, we pass to the limit in the parabolic equation.

\medskip	
	
\textbf{Step 1: Precompactness of approximating sequences}
	
From Theorem~\ref{thm:fully_discrete_energy_stability}, the Poincar\'e inequality, the fact that $\|\tilde q_h\|_{H^1}, \|\tilde g_h(t)\|_{H^1} \le C$ uniformly in $h$, and Corollary~\ref{corollary:fully_discrete_DtQ_L2_bound}, we have 
\begin{equation}
\label{eq:compactness_bounds}
    \{Q_{h,\Delta t}\} \subset L^\infty(0,T; H^1(\Omega)),\quad
    \{u_{h,\Delta t}\} \subset L^\infty(0,T; H^1(\Omega)),\quad
    \{D_t^+ Q_{h,\Delta t}\} \subset L^2([0,T]\times\Omega)
\end{equation}
uniformly in $\Delta t, h>0$ sufficiently small. Then the Banach-Alaoglu theorem implies weak star convergence of subsequences as $\Delta t, h\to 0$, for simplicity denoted $\{(Q_h, u_h)\}_{h>0}$,
\begin{subequations}
\begin{align}
    Q_h \overset{\ast}{\rightharpoonup} Q \text{ in } L^\infty(0,T; H^1(\Omega)),&\quad u_h \overset{\ast}{\rightharpoonup} u \text{ in } L^\infty(0,T; H^1(\Omega)),\\
    D_t^+ Q_h \rightharpoonup \partial_t Q\text{ in }& L^2([0,T]\times\Omega).
\end{align}
\end{subequations}
The bounds \eqref{eq:compactness_bounds} allow us to apply the result of Lemma~\ref{conj:precompactness},
\begin{equation}
    Q_h \to Q\text{ in } L^2([0,T]\times\Omega)
\end{equation}
up to a subsequence, for simplicity still denoted $Q_h$. Because $Q_h$ converges to $Q$ in $L^2([0,T]\times\Omega)$, there is a subsequence which converges almost everywhere, and so we see that $Q$ must be trace-free and symmetric, as the $Q_h$ are trace-free and symmetric. In an abuse of notation, we will denote by $\Pi_h$ the $L^2(\Omega)$-orthogonal projection from $H_0^1(\Omega;\R^{d\times d})$ to $\mathbb X_h$ and also the $L^2(\Omega)$-projection from $H_0^1(\Omega)$ to $\mathbb Y_h$. Then this projection satisfies
\begin{equation}
\label{eq:projection_convergence}
    \lvert\Pi_h v - v\rvert_{H^1} \overset{h\to0}{\longrightarrow} 0
\end{equation}
for $v\in H^s(\Omega)$ or $v\in H^s(\Omega;\R^{d\times d})$, $s>1$. Furthermore, $\|\Pi_h\|_{\mathcal{L}(H^1(\Omega);H^1(\Omega))}$ is uniformly bounded in $h>0$ (see~\cite{Bramble2002} or~\cite[Lemma 11.18]{Ern2021}). Then for a given $\Phi,\psi\in C^1_c([0,T]\times\Omega)$, taking $\Pi_h\Phi$ and $\Pi_h\psi$ to be test functions in (\ref{eq:fully_discrete_scheme_definition}), we obtain (omitting dependences on $x$)
\begin{subequations}
\begin{equation}
\label{eq:scheme_parabolic_convergence}
\begin{split}
    &- \int_0^T\int_\Omega D_t^+ Q_h : \Pi_h\Phi\, dx\, dt \\
        &= L \int_0^T\int_\Omega \frac{\nabla Q_h(t+\Delta t) + \nabla Q_h(t)}{2}:\nabla\Pi_h\Phi\, dx\, dt\\ 
        &+  \int_0^T\int_\Omega \left(\frac{\partial\mathcal F_1(Q_h(t+\Delta t))}{\partial Q} - \frac{\partial\mathcal F_2(Q_h(t))}{\partial Q}\right):\Pi_h\Phi\, dx\, dt\\ 
        &- \frac{\varepsilon_2}{2} \int_0^T\int_\Omega\left((\tilde\P(Q_h(t+\Delta t), Q_h(t))\odot \nabla u_h(t)(\nabla u_h(t+\Delta t))^\top)_S\right.\\
        &\hspace{20ex}\left.- \frac1d\operatorname{tr}(\tilde\P(Q_h(t+\Delta t), Q_h(t))\odot \nabla u_h(t)(\nabla u_h(t+\Delta t))^\top)I\right):\Pi_h\Phi\, dx\, dt,
\end{split}
\end{equation}
\begin{equation}
\label{eq:scheme_elliptic_convergence}
    \int_0^T\int_\Omega (\varepsilon_1\nabla u_h + \varepsilon_2 \T_R(Q_h)\nabla u_h)\cdot\nabla\Pi_h\psi\, dx\, dt = 0.
\end{equation}
\end{subequations}

\medskip 

\textbf{Step 2: Convergence of Elliptic Equation}

We first consider the elliptic equation.  We have, using triangle inequality,
\begin{align*}
	&\left\lvert \int_0^T \int_\Omega\T_R(Q_h)\nabla u_h\cdot\nabla\Pi_h\psi\, dx\, dt - \int_0^T\int_\Omega \T_R(Q)\nabla u\cdot\nabla\psi\, dx\, dt\right\rvert\\
	&\le \underbrace{\left\lvert \int_0^T\int_\Omega (\T_R(Q_h)-\T_R(Q))\nabla u_h\cdot\nabla\Pi_h\psi\, dx\, dt\right\rvert}_{I} + \underbrace{\left\lvert\int_0^T\int_\Omega \left(\T_R(Q)\nabla u_h\cdot\nabla\Pi_h\psi - \T_R(Q)\nabla u\cdot\nabla\psi\right)\, dx\, dt\right\rvert}_{II}.
\end{align*}
Estimating $I$, we note that $(\T_R(Q_h)-\T_R(Q))\nabla u_h\cdot\nabla\Pi_h\psi = \tr((\nabla\Pi_h\psi)^\top(\T_R(Q_h)-\T_R(Q))\nabla u_h)$ and use the cyclic property of the trace followed by the Cauchy-Schwarz inequality and the fact that $\T_R$ is Lipschitz to obtain
\begin{align*}
	I
	&= \left\lvert \int_0^T\int_\Omega (\T_R(Q_h) - \T_R(Q)):\nabla u_h(\nabla\Pi_h\psi)^\top\, dx\, dt\right\rvert\\
	&\le \|\T_R(Q_h) - \T_R(Q)\|_{L^2([0,T]\times\Omega)}\|\nabla u_h(\nabla\Pi_h\psi)^\top\|_{L^2([0,T]\times\Omega)}\\
	&\le C\|Q_h - Q\|_{L^2([0,T]\times\Omega)}\|\nabla u_h(\nabla\Pi_h\psi)^\top\|_{L^2([0,T]\times\Omega)}\\
	&\le C\|Q_h - Q\|_{L^2([0,T]\times\Omega)} \stackrel{h,\Delta t\to 0}{\longrightarrow} 0,
\end{align*}
where we have used in the last line that $\nabla u_h$ is bounded in $L^\infty(0,T;L^2(\Omega))$ and $\Pi_h\psi \in C_c([0,T]\times\overline{\Omega})$ is uniformly bounded in $h>0$.
Then to estimate $II$, we have
\begin{align*}
	II 
	&\le \left\lvert \int_0^T\int_\Omega\T_R(Q)\nabla u_h\cdot(\nabla\Pi_h\psi - \nabla \psi)\, dx\, dt\right\rvert + \left\lvert \int_0^T\int_\Omega \T_R(Q)(\nabla u_h - \nabla u)\cdot\nabla \psi\, dx\, dt\right\rvert.
\end{align*}
The first term converges to zero since $|\Pi_h \psi - \psi|_{H^1}\to 0$, $\mathcal{T}_R(Q_h)$ is uniformly bounded and $\norm{\Grad u_h}_{L^2}\leq C$ for all $h>0$.
The second term converges to $0$ due to the weak star convergence of $\nabla u_h$ to $\nabla u$. Thus $II \to 0$ and so we obtain
\begin{equation}
	\label{eq:elliptic2_convergence}
	\int_0^T \int_\Omega \T_R(Q_h)\nabla u_h\cdot\nabla\Pi_h\psi\, dx\, dt \to \int_0^T\int_\Omega \T_R(Q)\nabla u\cdot\nabla\psi\, dx\, dt.
\end{equation}
In the same way, one can show for the other term in the elliptic equation,
\begin{equation}
\label{eq:elliptic1_convergence}
    \int_0^T \int_\Omega\nabla u_h\cdot\nabla\Pi_h\psi\, dx\, dt \longrightarrow \int_0^T\int_\Omega \nabla u\cdot \nabla\psi\, dx\, dt
\end{equation}
as $\Delta t,h\to 0$.
And so combining \eqref{eq:scheme_elliptic_convergence}, \eqref{eq:elliptic1_convergence}, and \eqref{eq:elliptic2_convergence}, we have
\begin{equation}
\label{eq:scheme_elliptic_convergence_limit}
    \int_0^T\int_\Omega (\varepsilon_1\nabla u + \varepsilon_2 \T_R(Q)\nabla u)\cdot\nabla\psi\, dx\, dt = 0.
\end{equation}
Since it follows from the energy inequality that $u\in L^2(0,T;H^1(\dom))$, and smooth functions are dense in $H^1(\Omega)$, this identity holds for any $\psi\in L^2(0,T;H^1(\dom))$.

\medskip

\textbf{Step 3: Strong Convergence of $\nabla u_h$}

For what follows, we will need to show that $\nabla u_h \to \nabla u$ in $L^2([0,T]\times\Omega)$. To do this, take $\psi = u_h - \tilde g_h$ in (\ref{eq:scheme_elliptic_convergence}) and $\psi = u - \tilde g$ in (\ref{eq:scheme_elliptic_convergence_limit}) to obtain
\begin{align*}
\int_0^T\int_\Omega (\varepsilon_1\nabla u_h + \varepsilon_2\T_R(Q_h)\nabla u_h)\cdot\nabla (u_h - \tilde g_h)\, dx\, dt &= 0,\\
\int_0^T\int_\Omega (\varepsilon_1\nabla u + \varepsilon_2\T_R(Q)\nabla u)\cdot\nabla (u - \tilde g)\, dx\, dt &= 0.
\end{align*}
Subtracting these two equations gives
\begin{equation}
\label{eq:u_difference_strong_convergence}
\begin{split}
&\int_0^T\int_\Omega \left(\varepsilon_1 (\lvert\nabla u\rvert^2 - \lvert\nabla u_h\rvert^2) + \varepsilon_2(\nabla u^\top\T_R(Q)\nabla u - \nabla u_h^\top\T_R(Q_h)\nabla u_h)\right)\, dx\, dt\\ 
&= \int_0^T\int_\Omega \left(\varepsilon_1(\nabla\tilde g\cdot\nabla u - \nabla\tilde g_h\cdot\nabla u_h) + \varepsilon_2(\nabla\tilde g^\top \T_R(Q)\nabla u - \nabla\tilde g_h^\top\T_R(Q_h)\nabla u_h)\right)\, dx\, dt.
\end{split}
\end{equation}
For the left hand side, we first have
\begin{equation*}
\int_0^T\int_\Omega \varepsilon_1 (\lvert\nabla u\rvert^2 - \lvert \nabla u_h\rvert^2)\, dx\, dt
= \int_0^T\int_\Omega \varepsilon_1(-\lvert \nabla u - \nabla u_h\rvert^2 + 2\nabla u^\top(\nabla u - \nabla u_h))\, dx\, dt.
\end{equation*}
Furthermore, we have
\begin{align*}
&\int_0^T\int_\Omega \varepsilon_2 (\nabla u^\top \T_R(Q)\nabla u - \nabla u_h^\top \T_R(Q_h)\nabla u_h)\, dx\, dt\\
&= \int_0^T\int_\Omega \varepsilon_2\left(\nabla u^\top \T_R(Q)(\nabla u-\nabla u_h) - (\nabla u - \nabla u_h)^\top \T_R(Q_h)(\nabla u-\nabla u_h)\right.\\ 
&\hspace{12ex}\left.+ \nabla u^\top \T_R(Q_h)(\nabla u-\nabla u_h) + \nabla u^\top(\T_R(Q)-\T_R(Q_h))\nabla u_h\right)\, dx\, dt.
\end{align*}
Thus, we can rewrite~\eqref{eq:u_difference_strong_convergence} as
\begin{multline}
\label{eq:notsure}
\int_0^T\int_\Omega \varepsilon_1\lvert \nabla u - \nabla u_h\rvert^2
+\varepsilon_2(\nabla u - \nabla u_h)^\top \T_R(Q_h)(\nabla u-\nabla u_h)dx dt\\
=2\varepsilon_1\underbrace{\int_0^T\int_\Omega  \nabla u^\top(\nabla u - \nabla u_h) dx dt}_{I}+\varepsilon_2\underbrace{\int_0^T\int_\Omega \nabla u^\top \T_R(Q)(\nabla u-\nabla u_h)  dx dt}_{II}\\
+\varepsilon_2\underbrace{\int_0^T\int_\Omega  \nabla u^\top \T_R(Q_h)(\nabla u-\nabla u_h) dx dt}_{III} +\varepsilon_2\underbrace{\int_0^T\int_\Omega \nabla u^\top(\T_R(Q)-\T_R(Q_h))\nabla u_h dxdt}_{IV}
\\ 
- \varepsilon_1\underbrace{\int_0^T\int_\Omega (\nabla\tilde g\cdot\nabla u - \nabla\tilde g_h\cdot\nabla u_h)dx dt}_{V} - \varepsilon_2\underbrace{\int_0^T\int_\Omega(\nabla\tilde g^\top \T_R(Q)\nabla u - \nabla\tilde g_h^\top\T_R(Q_h)\nabla u_h)dx dt}_{VI}.
\end{multline}
We first note that we can lower bound the left hand side by 
\begin{equation}\label{eq:u_strong_convergence_lhs_bound}
\int_0^T\int_\Omega \varepsilon_1\lvert \nabla u - \nabla u_h\rvert^2
+\varepsilon_2(\nabla u - \nabla u_h)^\top \T_R(Q_h)(\nabla u-\nabla u_h)dx dt\geq (\varepsilon_1-R|\varepsilon_2|)\norm{\Grad u - \Grad u_h}_{L^2([0,T]\times\Omega)}^2.
\end{equation}
Then we proceed to estimating the terms on the right hand side. The first two terms, $I$ and $II$, converge to zero by the weak convergence of $\Grad u_h$ to $\Grad u$ and the boundedness of $\T_R$. 

For $III$, we have that $\Grad u_h\rightharpoonup\Grad u$ in $L^2([0,T]\times\Omega)$ so that $(\Grad u_h-\Grad u)\Grad u^\top\rightharpoonup 0$ in $L^1([0,T]\times\Omega)$, and $\T_R(Q_h)\to\T_R(Q)$ a.e. since $Q_h\to Q$ a.e. along the subsequence we have taken. Also, $\T_R(Q_h)$ is bounded in $L^\infty([0,T]\times\Omega)$ due to truncation. Thus, we can apply Lemma A.1 in \cite{weber2021convergent} entrywise to obtain $\T_R(Q_h):(\nabla u-\nabla u_h)\nabla u^\top\rightharpoonup 0$ in $L^1([0,T]\times\Omega)$ which implies
\[
III=\int_0^T\int_\Omega \nabla u^\top \T_R(Q_h)(\nabla u-\nabla u_h)\, dx\, dt  = \int_0^T\int_\Omega \T_R(Q_h):(\nabla u-\nabla u_h)\nabla u^\top\, dx\, dt \to 0.
\]

Term $IV$ is similar. We have that since $Q_h\to Q$ a.e. along the subsequence we have taken, then $\T_R(Q) - \T_R(Q_h) \to 0$ a.e. Furthermore, $\T_R(Q) - \T_R(Q_h)$ is bounded in $L^\infty([0,T]\times\Omega)$ due to the truncation. By weak convergence of $\nabla u_h$ in $L^2([0,T]\times\Omega)$ and the fact that $\nabla u$ is in $L^2([0,T]\times\Omega)$, we have that $\nabla u_h\nabla u^\top$ converges weakly to $\nabla u\nabla u^\top$ in $L^1([0,T]\times\Omega)$. Thus, applying Lemma A.1 in \cite{weber2021convergent}, we obtain $(\T_R(Q)-\T_R(Q_h)):\nabla u_h\nabla u^\top \rightharpoonup 0$ in $L^1([0,T]\times\Omega)$ which implies
\[
IV = \int_0^T\int_\Omega \nabla u^\top(\T_R(Q) - \T_R(Q_h))\nabla u_h\, dx\, dt = \int_0^T\int_\Omega (\T_R(Q) - \T_R(Q_h)):\nabla u_h\nabla u^\top\, dx\, dt\to 0.
\]
Next, consider $V$:
\begin{align*}
|V|&=\left\lvert\int_0^T\int_\Omega (\nabla\tilde g\cdot\nabla u - \nabla \tilde g_h\cdot\nabla u_h)\, dx\, dt\right\rvert = \left\lvert\int_0^T\int_\Omega (\nabla\tilde g\cdot(\nabla u - \nabla u_h) - (\nabla \tilde g_h - \nabla \tilde g)\cdot\nabla u_h)\, dx\, dt\right\rvert\\
&\le \|\nabla\tilde g_h-\nabla\tilde g\|_{L^2([0,T]\times\Omega)}\|\nabla u_h\|_{L^2([0,T]\times\Omega)} + \left\lvert \int_0^T\int_\Omega \nabla\tilde g\cdot(\nabla u-\nabla u_h)\, dx\, dt\right\rvert.
\end{align*}
Then since $\nabla \tilde g_h\to \nabla\tilde g$ in $L^2([0,T]\times\Omega)$, $\nabla u_h$ is bounded in $L^2([0,T]\times\Omega)$, and $\nabla u_h \rightharpoonup \nabla u$ in $L^2([0,T]\times\Omega)$, we have that this bound goes to $0$ as $\Delta t, h\to 0$. For $VI$, we have
\begin{align*}
|VI|&=\left\lvert \int_0^T\int_\Omega (\nabla \tilde g^\top \T_R(Q)\nabla u - \nabla \tilde g_h^\top \T_R(Q_h)\nabla u_h)\, dx\, dt\right\rvert\\
&\le \left\lvert \int_0^T\int_\Omega \nabla \tilde g^\top \T_R(Q)(\nabla u - \nabla u_h)\, dx\, dt\right\rvert + \left\lvert \int_0^T\int_\Omega \nabla \tilde g^\top (\T_R(Q) - \T_R(Q_h))\nabla u_h\, dx\, dt\right\rvert\\ 
&\hspace{10ex}+ \left\lvert \int_0^T\int_\Omega (\nabla\tilde g^\top - \nabla \tilde g_h^\top)\T_R(Q_h)\nabla u_h\, dx\, dt\right\rvert.
\end{align*}
The first of these converges to $0$ by weak convergence of $\nabla u_h$. The second term converges to $0$ by the same argument as the analogous term dealt with before. The third term is bounded by $\|\nabla\tilde g - \nabla \tilde g_h\|_{L^2([0,T]\times\Omega)}\|\T_R(Q_h)\nabla u_h\|_{L^2([0,T]\times\Omega)}$ and hence goes to $0$ as $\Delta t,h\to 0$. 
It follows that all the terms $I$ - $VI$ in~\eqref{eq:notsure} converge to zero as $h,\Delta t\to 0$ up to subsequence.
Combining this with~\eqref{eq:u_strong_convergence_lhs_bound}, we have that
\begin{equation}
\begin{split}
0 &\ge \lim_{\Delta t,h\to 0} \left\{\varepsilon_1\|\nabla u - \nabla u_h\|_{L^2([0,T]\times\Omega)}^2 -\lvert\varepsilon_2\rvert R\|\nabla u - \nabla u_h\|_{L^2([0,T]\times\Omega)}^2\right\},
\end{split}
\end{equation}
which implies that $\|\nabla u-\nabla u_h\|_{L^2([0,T]\times\Omega)}\to 0$ as $\Delta t, h\to 0$ since $\varepsilon_1 > \lvert\varepsilon_2\rvert R$.

\medskip

\textbf{Step 4: Convergence of Parabolic Equation}

We turn now to estimates for \eqref{eq:scheme_parabolic_convergence}. First, we have
\begin{align*}
    &\left\lvert \int_0^T \int_\Omega D_t^+ Q_h:\Pi_h\Phi\, dx\, dt - \int_0^T\int_\Omega \partial_t Q:\Phi\, dx\, dt\right\rvert\\
        &\le \left\lvert \int_0^T\int_\Omega D_t^+ Q_h: (\Pi_h\Phi - \Phi)\, dx\, dt\right\rvert + \left\lvert\int_0^T\int_\Omega (D_t^+ Q_h - \partial_t Q):\Phi\, dx\, dt\right\rvert.
\end{align*}
The second of these terms converges to $0$ by weak star convergence of $D_t^+ Q_h$ to $\partial_t Q$. For the first term, we have
\begin{align*}
    \left\lvert \int_0^T\int_\Omega D_t^+ Q_h: (\Pi_h\Phi - \Phi)\, dx\, dt\right\rvert
        &\le \|D_t^+ Q_h\|_{L^2([0,T]\times\Omega)}\|\Pi_h\Phi - \Phi\|_{L^2([0,T]\times\Omega)}.
\end{align*}
Then by the Poincar\'e inequality and since $\lvert\Pi_h\Phi(t) - \Phi(t) \rvert_{H^1(\Omega)} \to 0$ for each $t$ as $\Delta t, h\to 0$ and is uniformly bounded by a constant, we have by Lebesgue's dominated convergence theorem and the fact that $D_t^+ Q_h$ is uniformly bounded in $L^2([0,T]\times\Omega)$ that that $\left\lvert \int_0^T\int_\Omega D_t^+ Q_h: (\Pi_h\Phi - \Phi)\, dx\, dt\right\rvert \to 0$ as $\Delta t, h\to 0$. Thus,
\begin{equation}
    \int_0^T \int_\Omega D_t^+ Q_h:\Pi_h\Phi\, dx\, dt \to \int_0^T\int_\Omega \partial_t Q:\Phi\, dx\, dt
\end{equation}
as $\Delta t, h\to 0$. By a similar argument, we obtain
\begin{equation}
    \int_0^T\int_\Omega \frac{\nabla Q_h(t)}{2}:\nabla\Pi_h\Phi\, dx\, dt \to \frac{1}{2} \int_0^T\int_\Omega \nabla Q:\nabla\Phi\, dx\, dt
\end{equation}
as $\Delta t, h\to 0$. Then for $\Delta t$ sufficiently small (since $\Phi$ has compact support), we have
\begin{align*}
    &\left\lvert \int_0^T\int_\Omega \frac{\nabla Q_h(t+\Delta t)}{2}:\nabla\Pi_h\Phi(t)\, dx\, dt - \int_0^T \int_\Omega \frac{\nabla Q}{2}: \nabla\Phi(t)\, dx\, dt\right\rvert\\ 
    &\le \left\lvert \int_0^T\int_\Omega \frac{\nabla Q_h(t)}{2}:(\nabla\Pi_h\Phi(t-\Delta t) - \nabla\Phi(t))\, dx\, dt\right\rvert + \left\lvert \int_0^T \int_\Omega \frac{\nabla Q_h - \nabla Q}{2}: \nabla\Phi(t)\, dx\, dt\right\rvert.
\end{align*}
The second term above converges to $0$ by weak star convergence of $\nabla Q_h$. The first term we bound by
\begin{align*}
    &\left\lvert \int_0^T\int_\Omega \frac{\nabla Q_h(t)}{2}:(\nabla\Pi_h\Phi(t-\Delta t) - \nabla\Phi(t))\, dx\, dt\right\rvert \le \frac{1}{2}\|\nabla Q_h\|_{L^2([0,T]\times\Omega)}\|\nabla\Pi_h\Phi(\cdot-\Delta t)-\nabla\Phi\|_{L^2([0,T]\times\Omega)}.
\end{align*}
Now we have
\[
    \|\nabla\Pi_h\Phi(\cdot-\Delta t)-\nabla\Phi\|_{L^2([0,T]\times\Omega)} \le \|\nabla\Pi_h\Phi(\cdot-\Delta t)-\nabla\Pi_h\Phi\|_{L^2([0,T]\times\Omega)} + \|\nabla\Pi_h\Phi-\nabla\Phi\|_{L^2([0,T]\times\Omega)},
\]
and the second term on the right hand side converges to $0$, as before. For the first term, we have
\begin{align*}
    \|\nabla\Pi_h\Phi(\cdot-\Delta t)-\nabla\Pi_h\Phi\|_{L^2([0,T]\times\Omega)}
        &\le CT^{1/2} \esssup_{t\in[0,T]} \|\Pi_h\Phi(t-\Delta t) - \Pi_h\Phi(t)\|_{H^1(\Omega)}\\
        &\le CT^{1/2}\|\Pi_h\|_{\mathcal{L}(H^1(\Omega);H^1(\Omega))}o_{\Delta t\to 0}(1)
\end{align*}
where we have used that $\Phi$ is continuously differentiable.
Thus we see that this bound converges to $0$ as $\Delta t\to 0$, as $\|\Pi_h\|_{\mathcal{L}(H^1(\Omega);H^1(\Omega))}$ is uniformly bounded in $h>0$. Combining this with the previous bounds and the fact that $\|\nabla Q_h\|_{L^2([0,T]\times\Omega)}$ is uniformly bounded in $h>0$, we obtain
\begin{equation}
    \int_0^T\int_\Omega \frac{\nabla Q_h(t+\Delta t)}{2}:\nabla\Pi_h\Phi(t)\, dx\, dt \to \frac12\int_0^T \int_\Omega \nabla Q: \nabla\Phi(t)\, dx\, dt
\end{equation}
and conclude
\begin{equation}
   L \int_0^T\int_\Omega \frac{\nabla Q_h(t+\Delta t) + \nabla Q_h(t)}{2}:\nabla\Pi_h\Phi\, dx\, dt \to L \int_0^T\int_\Omega \nabla Q:\nabla\Phi\, dx\, dt
\end{equation}
as $\Delta t,h\to 0$. Now we have
\begin{align*}
    &\left\lvert\int_0^T\int_\Omega \frac{\partial\mathcal F_2(Q_h)}{\partial Q}:\Pi_h\Phi\, dx\, dt - \int_0^T\int_\Omega \frac{\partial\mathcal F_2(Q)}{\partial Q}:\Phi\, dx\, dt \right\rvert\\ 
    &= \left\lvert \int_0^T\int_\Omega \left[(\beta_1 - a)Q_h:\Pi_h\Phi + (\beta_2 - c)\operatorname{tr}(Q_h^2)Q_h:\Pi_h\Phi - (\beta_1 - a)Q:\Phi - (\beta_2-c)\operatorname{tr}(Q^2)Q:\Phi\right]\, dx\, dt \right\rvert\\
    &\le (\beta_1 - a)\underbrace{\left\lvert \int_0^T\int_\Omega (Q_h:\Pi_h\Phi - Q:\Phi)\, dx\, dt\right\rvert}_{I} + (\beta_2 - c)\underbrace{\left\lvert\int_0^T\int_\Omega (\operatorname{tr}(Q_h^2)Q_h:\Pi_h\Phi - \operatorname{tr}(Q^2)Q:\Phi)\, dx\, dt\right\rvert}_{II}.
\end{align*}
Similar to arguments above, we have $I\to 0$ as $\Delta t,h\to 0$. For $II$, we have
\begin{align*}
    II
        &\le \underbrace{\left\lvert \int_0^T\int_\Omega \operatorname{tr}(Q_h^2)Q_h:(\Pi_h\Phi - \Phi) \, dx\, dt\right\rvert}_{I'} + \underbrace{\left\lvert \int_0^T\int_\Omega (\operatorname{tr}(Q_h^2)Q_h:\Phi - \operatorname{tr}(Q^2)Q:\Phi)\, dx\, dt\right\rvert}_{II'}.
\end{align*}
Because $\nabla Q_h \in L^\infty(0,T;L^2(\Omega))$ uniformly in $\Delta t, h>0$, we have that $Q_h \in L^6([0,T]\times\Omega)$ by the Sobolev embedding theorem, so $\operatorname{tr}(Q_h^2)Q_h$ is bounded in $L^2([0,T]\times\Omega)$. It follows as in our previous arguments that $I' \to 0$ as $\Delta t,h\to 0$. Then for $II'$, we have
\begin{align*}
    II'
        &= \left\lvert \int_0^T\int_\Omega (\operatorname{tr}(Q_h^2)Q_h - \operatorname{tr}(Q_hQ)Q_h + \operatorname{tr}(Q_hQ)Q_h - \operatorname{tr}(Q^2)Q_h + \operatorname{tr}(Q^2)Q_h - \operatorname{tr}(Q^2)Q):\Phi\, dx\, dt\right\rvert\\
        &\leq C\int_0^T\int_{\Omega} |Q_h-Q|\left(|Q|^2 + |Q_h|^2\right) |\Phi| dx dt\\
        & \leq C \norm{Q_h-Q}_{L^2([0,T]\times \Omega)}\left(\norm{Q}_{L^4([0,T]\times\Omega)}^2 +\norm{Q_h}_{L^4([0,T]\times\Omega)}^2 \right)\norm{\Phi}_{L^\infty([0,T]\times\Omega)}.
\end{align*}
Since $\Phi$ is bounded, $Q$ and $Q_h$ are bounded in $L^4([0,T]\times\Omega)$, and $Q_h \to Q$ in $L^2([0,T]\times\Omega)$, we have that $II' \to 0$ and we conclude that $II\to 0$ and thus 
\begin{equation}
    \int_0^T\int_\Omega \frac{\partial\mathcal F_2(Q_h)}{\partial Q}:\Pi_h\Phi\, dx\, dt \to \int_0^T\int_\Omega \frac{\partial\mathcal F_2(Q)}{\partial Q}:\Phi\, dx\, dt
\end{equation}
as $\Delta t, h\to 0$.

Now since $\mathcal{F}_1$ is of the same order as $\mathcal{F}_2$, for $\Delta t$ sufficiently small, we can similarly show as above, additionally using the continuity of $L^2([0,T]\times\Omega)$ shifts, that 
\begin{align*}
    \int_0^T\int_\Omega \frac{\partial\mathcal F_1(Q_h(t+\Delta t))}{\partial Q}:\Pi_h\Phi\, dx\, dt &= \int_0^T\int_\Omega \frac{\partial\mathcal F_1(Q_h)}{\partial Q}:\Pi_h\Phi(t-\Delta t)\, dx\, dt \to \int_0^T \int_\Omega \frac{\partial\mathcal{F}_1(Q)}{\partial Q}:\Phi\, dx\, dt.
\end{align*}
as $\Delta t,h\to 0$.
Denoting $\Phi_h = \Pi_h\Phi$, we now consider
\begin{align*}
    &\left\lvert\int_0^T\int_\Omega\left((\tilde\P(Q_h(t+\Delta t), Q_h(t))\odot \nabla u_h(t)(\nabla u_h(t+\Delta t))^\top)_S\right):\Phi_h\, dx\, dt - \int_0^T\int_\Omega (\P(Q)\odot\nabla u(\nabla u)^\top)_S:\Phi\, dx\, dt\right\rvert.
\end{align*}
We will show that this quantity converges to zero as $\Delta t, h\to 0$, from which it will follow that the trace term associated with it in~\eqref{eq:scheme_parabolic_convergence} also converges to the desired quantity by a very similar argument. We can also ignore the symmetric part in the above expression, for if we can show convergence without the symmetric part, the corresponding transpose will also converge as we can just move the transpose onto $\Phi_h$. We have
\begin{align*}
    &\left\lvert \int_0^T \int_\Omega (\tilde\P(Q_h(t+\Delta t),Q_h(t))\odot\nabla u_h(t)\nabla u_h(t+\Delta t)^\top - \P(Q)\odot\nabla u_h(t)\nabla u_h(t+\Delta t)^\top):\Pi_h\Phi\, dx\, dt\right\rvert\\
    &\le C(\Phi)\sum_{i,j=1}^d \left\|(\tilde \P(Q_h(t+\Delta t), Q_h(t))_{ij} - \P(Q)_{ij})\frac{\partial u_h(t)}{\partial x_i}\right\|_{L^2([0,T]\times\Omega)}\left\|\frac{\partial u_h(t+\Delta t)}{\partial x_j}\right\|_{L^2([0,T]\times\Omega)}.
\end{align*}
By strong convergence of $Q_h$, we have that $\tilde \P(Q_h(\cdot+\Delta t), Q_h)_{ij} \to \P(Q)_{ij}$ pointwise a.e. in $[0,T]\times\Omega$. By Egoroff's theorem, we have that $\tilde \P(Q_h(\cdot+\Delta t), Q_h)_{ij} \to \P(Q)_{ij}$ almost uniformly. So for any $\delta>0$, there is $\omega_{\delta} \subseteq [0,T]\times\Omega$ such that $[0,T]\times\Omega\setminus\omega_{\delta}$ has measure at most $\delta$ and $\tilde \P(Q_h(\cdot+\Delta t), Q_h)_{ij} \to \P(Q)_{ij}$ uniformly on $\omega_\delta$. Then 
\begin{align*}
    &\left\|(\tilde \P(Q_h(t+\Delta t), Q_h(t))_{ij} - \P(Q)_{ij})\frac{\partial u_h(t)}{\partial x_i}\right\|_{L^2([0,T]\times\Omega)}\\
    &\le \left\|(\tilde \P(Q_h(t+\Delta t), Q_h(t))_{ij} - \P(Q)_{ij})\frac{\partial u_h(t)}{\partial x_i}\right\|_{L^2(\omega_\delta)} +  2\left\|(1-\chi_{\omega_\delta})\frac{\partial u_h(t)}{\partial x_i}\right\|_{L^2([0,T]\times\Omega)}.
\end{align*}
Then for each fixed $\delta$, we see by the uniform convergence on $\omega_\delta$ and the fact that $\partial u_h/\partial x_i$ is bounded in $L^2([0,T]\times\Omega)$ uniformly in $\Delta t,h$ that the first term above goes to 0 as $\Delta t,h\to0$. By strong convergence of $\nabla u_h$ to $\nabla u$, we have the second term goes to $2\|(1-\chi_{\omega_\delta})\partial u/\partial x_i\|_{L^2([0,T]\times\Omega)}$. As $\delta\to 0$, we see that this goes to 0 by the dominated convergence theorem.

Thus we see that
 \begin{align*}
    &\left\lvert \int_0^T \int_\Omega (\tilde\P(Q_h(t+\Delta t),Q_h(t))\odot\nabla u_h(t)\nabla u_h(t+\Delta t)^\top - \P(Q)\odot\nabla u_h(t)\nabla u_h(t+\Delta t)^\top):\Pi_h\Phi\, dx\, dt\right\rvert \to 0
\end{align*}
as $\Delta t,h\to 0$. As in previous arguments, we also have
\begin{align*}
    &\left\lvert\int_0^T\int_\Omega (\P(Q)\odot\nabla u_h(t)\nabla u_h(t+\Delta t)^\top):(\Pi_h\Phi - \Phi)\, dx\, dt \right\rvert \to 0.
\end{align*}
Then we have
\begin{align*}
    &\left\lvert \int_0^T\int_\Omega (\P(Q)\odot \nabla u_h(t)\nabla u_h(t+\Delta t)^\top):\Phi\, dx\, dt-\int_0^T\int_\Omega (\P(Q)\odot \nabla u(t)\nabla u_h(t+\Delta t)^\top):\Phi\, dx\, dt\right\rvert\\
    &\le C\sum_{i,j=1}^d \left\|\frac{\partial}{\partial x_i}(u_h-u)\right\|_{L^2([0,T]\times\Omega)}\left\|\frac{\partial}{\partial x_j} u_h(\cdot+\Delta t)\Phi_{ij}\right\|_{L^2([0,T]\times\Omega)},
\end{align*}
which converges to $0$ by the $L^2([0,T]\times\Omega)$-convergence of $\nabla u_h$ to $\nabla u$ and because $\Phi_{ij}$ has compact support so that the norm of $\partial_{x_j}u_h(\cdot+\Delta t)\Phi_{ij}$ is well defined. Now we have
\begin{align*}
    &\left\lvert \int_0^T\int_\Omega (\P(Q)\odot \nabla u(t)\nabla u_h(t+\Delta t)^\top):\Phi\, dx\, dt - \int_0^T\int_\Omega (\P(Q)\odot \nabla u\nabla u^\top):\Phi\, dx\, dt\right\rvert\\
    &\le C\sum_{i,j=1}^d\left\|\frac{\partial u}{\partial x_i}\right\|_{L^2([0,T]\times\Omega)}\left(\left\|\Phi_{ij}\left(\frac{\partial u_h(\cdot+\Delta t)}{\partial x_j} - \frac{\partial u(\cdot+\Delta t)}{\partial x_j}\right)\right\|_{L^2([0,T]\times\Omega)}\right.\\ 
    &\hspace{30ex}\left.+ \left\|\Phi_{ij}\left(\frac{\partial u(\cdot+\Delta t)}{\partial x_j} - \frac{\partial u}{\partial x_j}\right)\right\|_{L^2([0,T]\times\Omega)}\right),
\end{align*}
which converges to $0$ by the $L^2([0,T]\times\Omega)$-convergence of $\nabla u_h$ to $\nabla u$ and by continuity of $L^2([0,T]\times\Omega)$ shifts. Thus, we have
\begin{equation}
    \int_0^T\int_\Omega\left((\tilde\P(Q_h(t+\Delta t), Q_h(t))\odot \nabla u_h(t)(\nabla u_h(t+\Delta t))^\top)_S\right):\Phi_h\, dx\, dt \to \int_0^T\int_\Omega (\P(Q)\odot\nabla u(\nabla u)^\top)_S:\Phi\, dx\, dt
\end{equation}
as $\Delta t,h\to 0$, as desired.

Combining all the estimates above, we see that $(Q, u)$ is a weak solution of~\eqref{eq:weak_formulation}, as desired.
\end{proof}

\section{Numerical Results}\label{sec:num}
We now present numerical results computed with our scheme~\eqref{eq:fully_discrete_scheme_definition}. All code to reproduce our results can be found at \url{https://github.com/maxhirsch/Maxwell-Liquid-Crystals}. We compute the truncation $\T_R(Q)$ as
\begin{equation}\label{eq:TRnumerics}
    \T_R(Q)_{ij} = Q_{ij} \widehat{H}\left(\frac{R}{2}-Q_{ij}\right) \widehat{H}\left(\frac{R}{2}+Q_{ij}\right) + \frac{R}{2}\widehat{H}\left(Q_{ij} - \frac{R}{2}\right) - \frac{R}{2}\widehat{H}\left(-Q_{ij}-\frac{R}{2}\right),
\end{equation}
where $\widehat{H}$ is a smooth approximation of the Heaviside function and is given by
\begin{equation}\label{eq:TRnumerics2}
 \widehat{H}(x) = \frac{1}{\pi}\arctan(5x) + \frac12.
\end{equation}
This choice of $\T_R(Q)$ is made to truncate entries of $Q$ to values close to $R/2$ while being very nearly the identity when entries are in $[-R/2, R/2]$. For our experiments, the largest entries of $\T_R(Q)$ are about $R/2+0.05$. Furthermore, we compute $\P(Q)$ using the exact derivative of $\T_R$ instead of a discrete approximation so that the implementation could utilize automatically computed Newton iterations. Our choice of $\T_R$ is also such that its derivatives are bounded, as required by \eqref{eq:conditionsonTr}. Thus $\T_R$ satisfies \eqref{eq:conditionsonTr} with a slightly larger truncation parameter.
For our numerical experiments, we take
\[
    a = -0.3,\quad b = -4,\quad c = 4,\quad \beta_1 = 8,\quad \beta_2 = 8.
\]
Furthermore, we take $\Omega = [-0.5,0.5]^2$ with a $30\times 30$ mesh. Other parameters and the boundary and initial conditions differ for the various experiments. We first consider an example with constant initial director angle. In the second experiment, we show the effect of increasing the electric field magnitude. After that, the third experiment demonstrates the need for the truncation operator. Next, we show that our model captures a phenomenon in liquid crystals known as the Fr\'{e}edericksz transition. We end with an experiment to numerically determine the convergence rate.

\subsection{Constant Initial Director Angle}\label{sec:constant-initial-director}
For this numerical experiment, we take
\[
    L = 1,\quad \varepsilon_1 = 2.5,\quad \varepsilon_2 = 0.5,\quad \varepsilon_3 = 0.01,\quad R=2.
\]
For boundary and initial conditions, we take
\[
    \tilde g(t, x, y) = 10\sin(2\pi t + 0.2)(x+0.5)\sin(\pi y)
\]
and
\[
    Q_0(x, y) = \boldsymbol{\textrm{d}}\boldsymbol{\textrm{d}}^\top - \frac12\operatorname{tr}(\boldsymbol{\textrm{d}}\boldsymbol{\textrm{d}}^\top)I,\quad\text{where}\quad \boldsymbol{\textrm{d}} = \begin{bmatrix}
        (x+0.5)(x-0.5)(y+0.5)(y-0.5)\\
        (x+0.5)(x-0.5)(y+0.5)(y-0.5)
    \end{bmatrix}
\]
is the director for the liquid crystal molecule orientation. The boundary data $\tilde q$ is taken according to $Q_0$ on the boundary. In words, the initial directors have the same angle at each point in space. We see that the initial directors are continuous in $\overline{\Omega}$. We also take $T = 2$ and $\Delta t = 0.01$. The resulting solution $(Q_h, u_h)$ is shown in Figure~\ref{fig:experiment_1_solution}. The maximum magnitude of any entry of a Q-tensor at time $t$ is shown in Figure~\ref{fig:experiment_1_qtensor_entries}, along with the maximum eigenvalue of any Q-tensor and the time-dependent coefficient in $\tilde g$.
\begin{figure}[h]
    \centering
    \subfloat[$t=0$]{\includegraphics[width=0.33\textwidth]{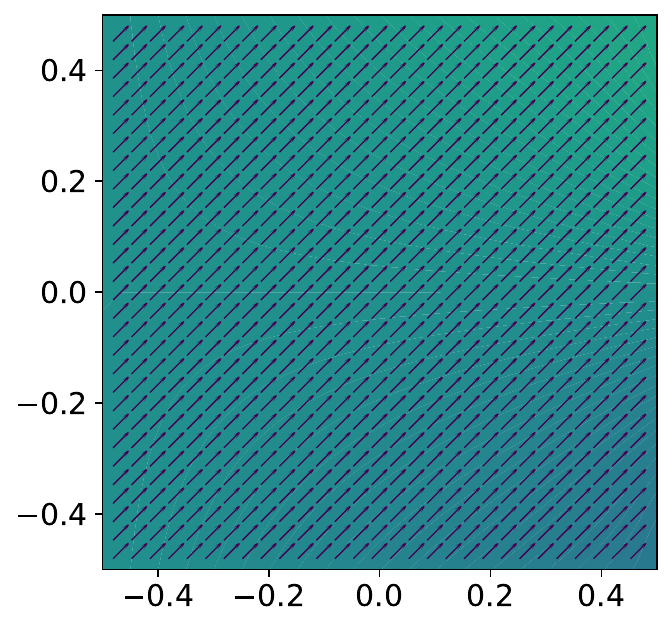}}
    \subfloat[$t=0.25$]{\includegraphics[width=0.33\textwidth]{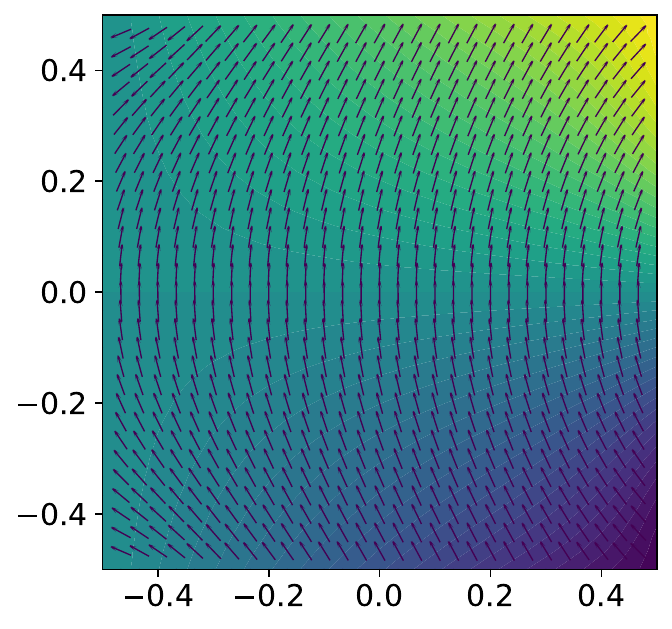}}
    \subfloat[$t=0.5$]{\includegraphics[width=0.33\textwidth]{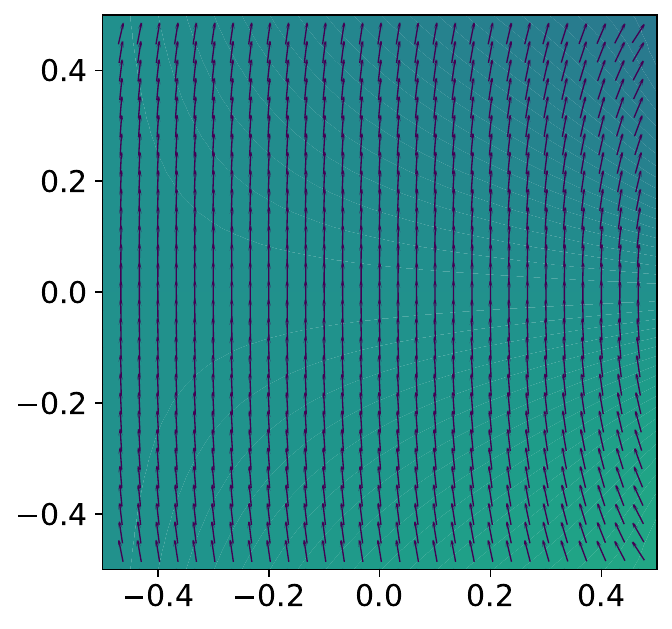}}
    \qquad
    \subfloat[$t=0.75$]{\includegraphics[width=0.33\textwidth]{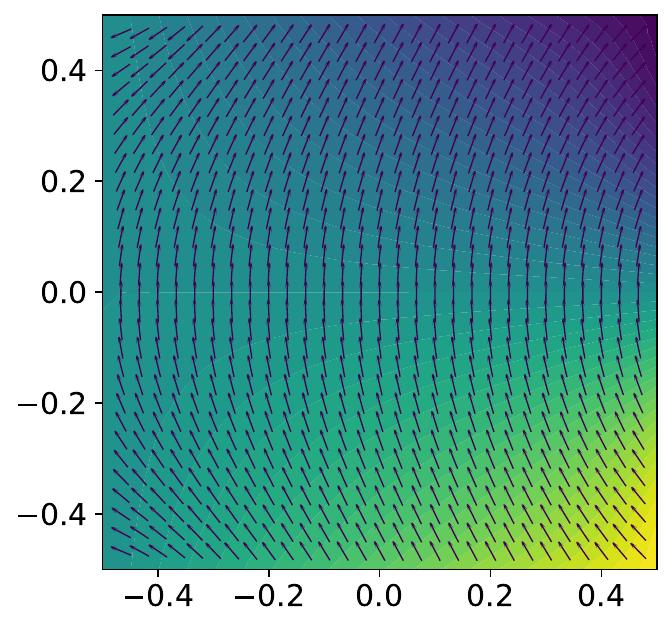}}
    \qquad 
    \subfloat[$t=2$]{\includegraphics[width=0.405\textwidth]{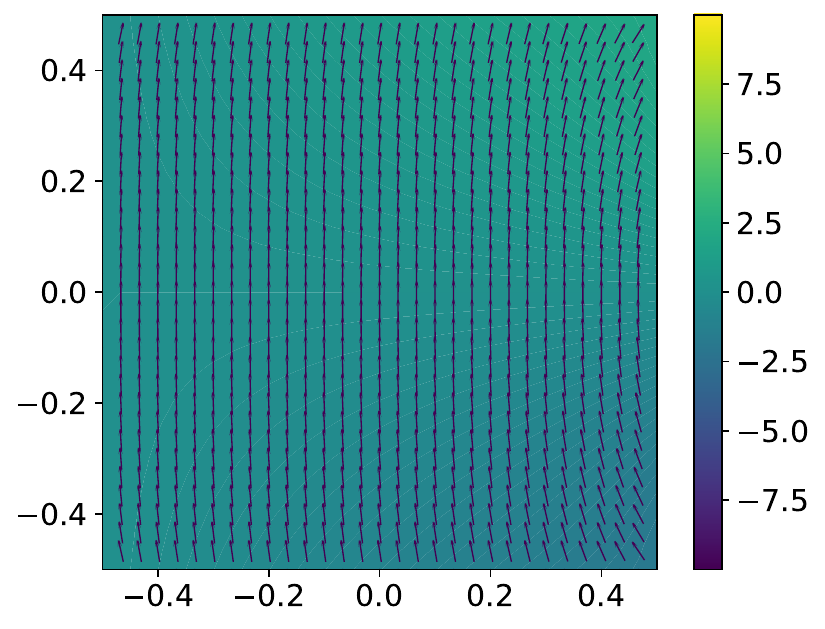}}
    \caption{Director field imposed on colored contour plot of electric potential $u$ for the constant initial director experiment in Section~\ref{sec:constant-initial-director}.}
    \label{fig:experiment_1_solution}
\end{figure}
\begin{figure}[h]
    \centering
    \includegraphics[scale=0.6]{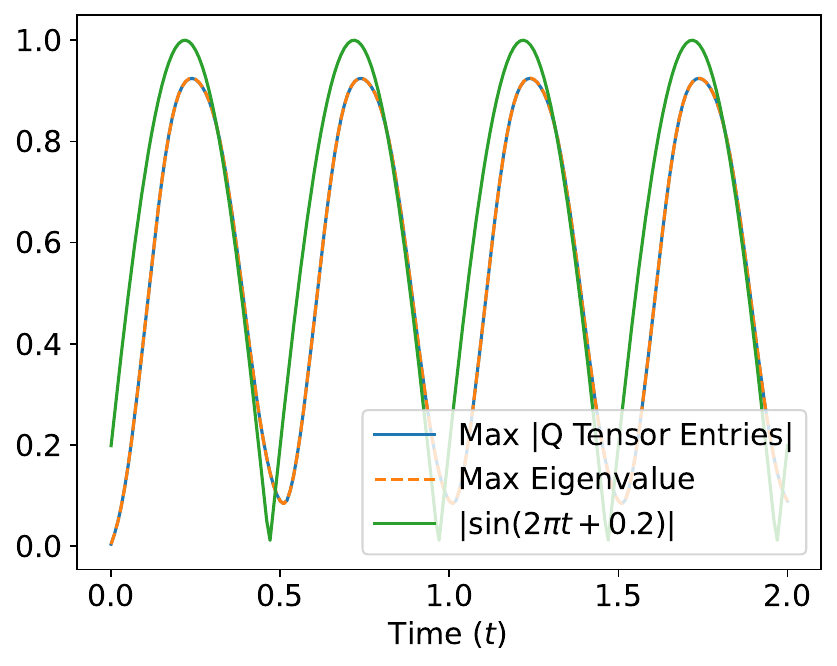}
    \caption{Maximum magnitude Q-tensor entry and maximum eigenvalue over time for the constant initial director experiment in Section~\ref{sec:constant-initial-director}.}
    \label{fig:experiment_1_qtensor_entries}
\end{figure}

The plot of solutions shows the director fields imposed on a color gradient of the electric potential. We see that the directors of the liquid crystals tend to align parallel to the gradient of the electric potential. When the electric potential changes from a large value in the top right of the domain to the bottom right and has value roughly $0$ throughout $\Omega$ (as in the transition from $t=0.25$ to $t=0.5$ to $t=0.75$), we see that the Q-tensor directors align in the vertical direction. They repeat this behavior of aligning with the gradient of the electric potential for all time, and we again see the directors in the vertical direction at time $t=2$. 

In Figure~\ref{fig:experiment_1_qtensor_entries}, we see the maximum magnitude Q-tensor entry and the maximum eigenvalue changing periodically according to the boundary data for the electric potential. The maximum eigenvalue is almost exactly the maximum magnitude Q-tensor entry. Of particular note is that the change in the magnitude of the electric potential is followed by a similar change in the Q-tensor entries. This suggests that the electric potential magnitude directly affects the magnitude of the Q-tensor entries. Because of this behavior, it seems that an electric potential with large enough magnitude could exhibit the need for truncation of the Q-tensors. To explore this idea further, we consider a modification of this experiment.

\subsection{Electric Field Increasing Magnitude}\label{sec:increasing-magnitude}
We now repeat the previous experiment except with 
\[
    \tilde g(t, x, y) = 10t\cdot\sin(2\pi t + 0.2)(x+0.5)\sin(\pi y)
\]
so that the relative maxima (minima) increase (decrease) as time increases. The results are shown in Figures~\ref{fig:experiment_2_solution} and~\ref{fig:experiment_2_qtensor_entries}. In particular, the behavior is similar to the previous experiment, except the local maxima in time of the Q-tensor maximum magnitude entries are increasing. We also see the lag between the magnitude of the electric potential changing and the Q-tensor entry magnitudes changing as before. So it seems that the magnitudes of the entries of the Q-tensors are related to the electric potential magnitude. We also see the maximum Q-tensor entry surpass $R/d = R/2 = 1.0$, the value at which each entry is truncated. This suggests that truncation is actually needed for the scheme unless we can assume the electric potential magnitude is small. We explore this further in the next experiment.

\begin{figure}
    \centering
    \subfloat[$t=0$]{\includegraphics[width=0.33\textwidth]{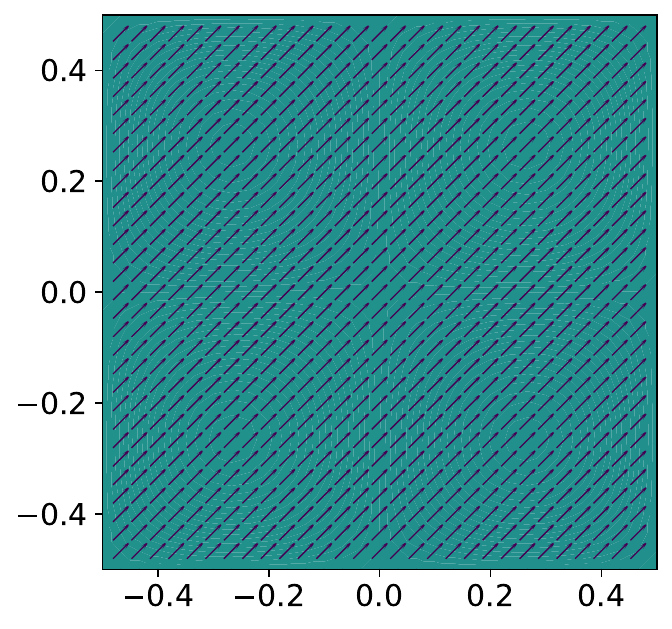}}
    \subfloat[$t=0.25$]{\includegraphics[width=0.33\textwidth]{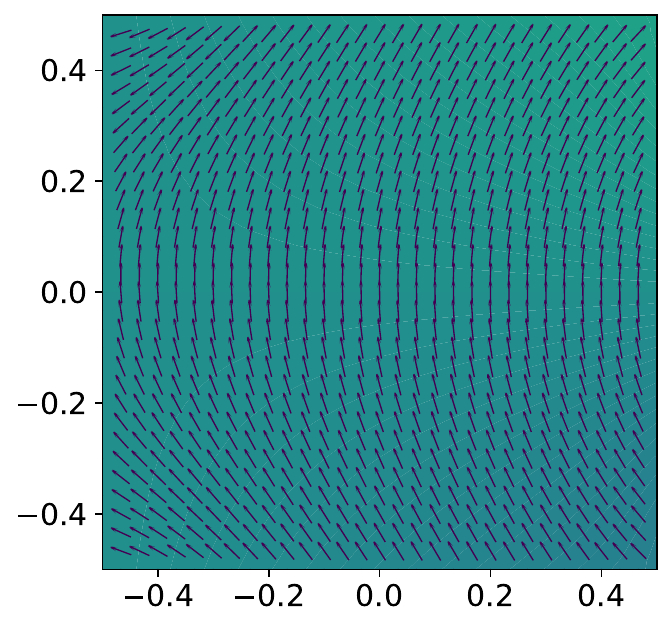}}
    \subfloat[$t=0.75$]{\includegraphics[width=0.33\textwidth]{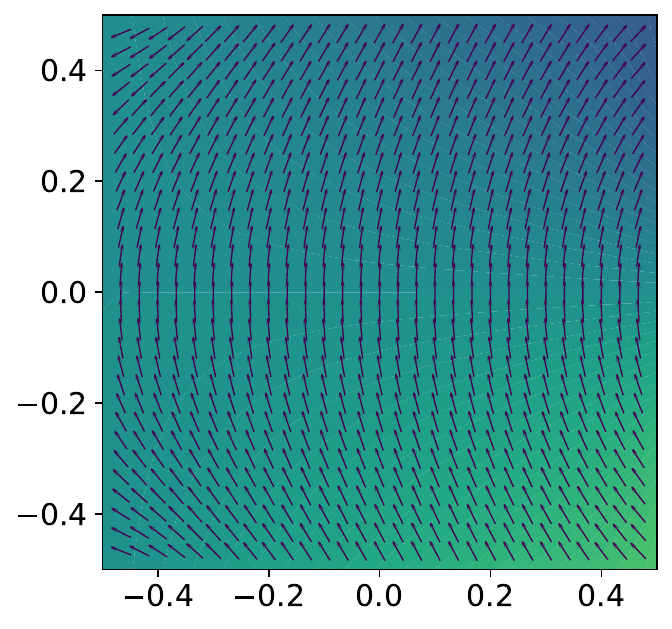}}
    \qquad
    \subfloat[$t=1.25$]{\includegraphics[width=0.33\textwidth]{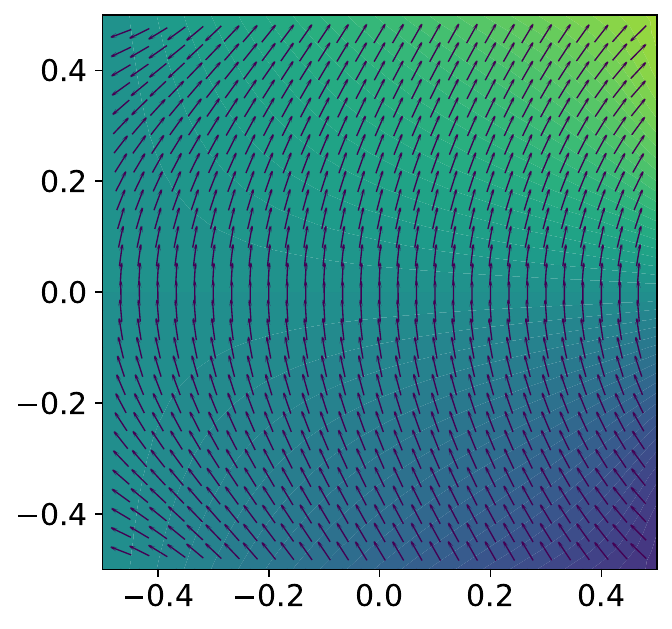}}
    \qquad 
    \subfloat[$t=1.75$]{\includegraphics[width=0.405\textwidth]{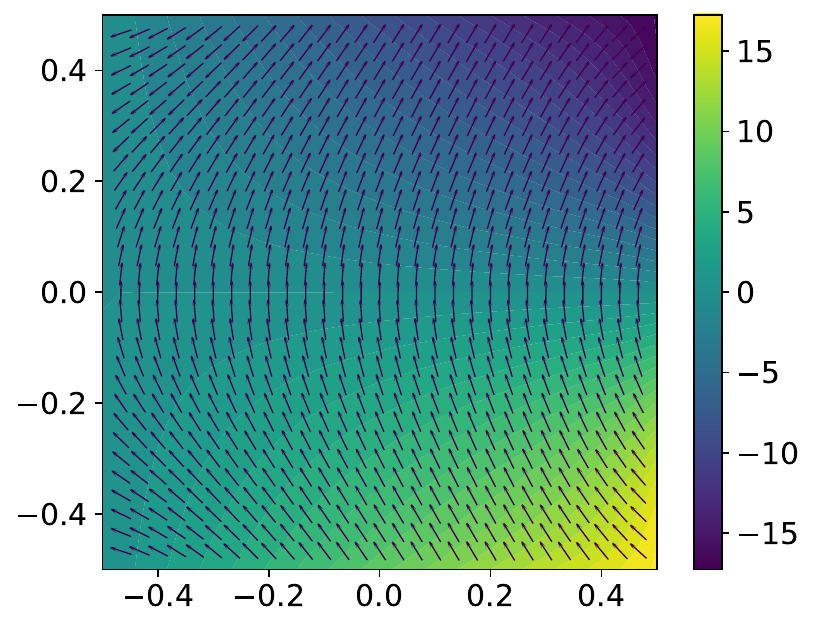}}
    \caption{Director field imposed on colored contour plot of electric potential $u$ for the increasing magnitude electric field experiment in Section~\ref{sec:increasing-magnitude}.}
    \label{fig:experiment_2_solution}
\end{figure}

\begin{figure}
    \centering
    \includegraphics[scale=0.6]{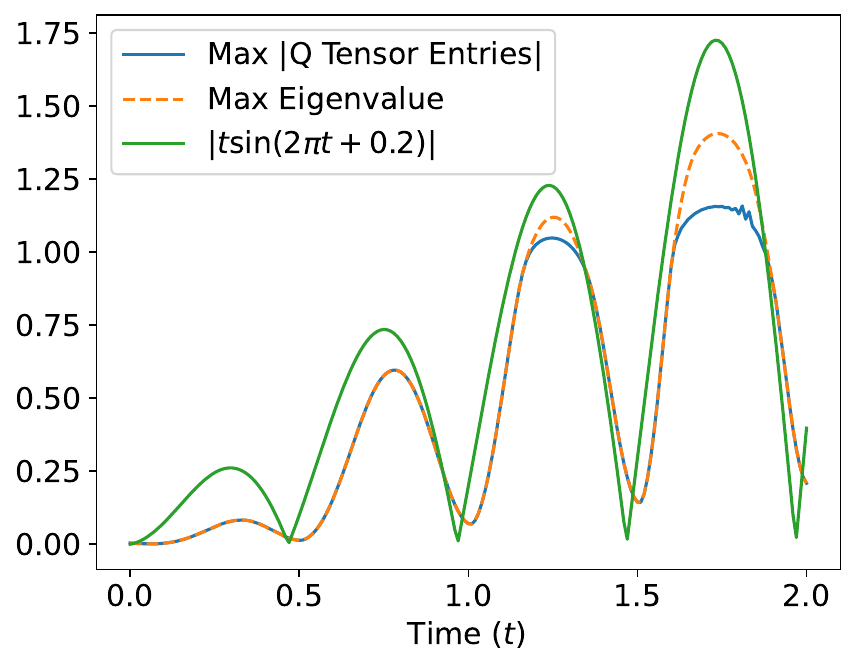}
    \caption{Maximum magnitude Q-tensor entry over time for the increasing magnitude electric field experiment in Section~\ref{sec:increasing-magnitude}.}
    \label{fig:experiment_2_qtensor_entries}
\end{figure}

\subsection{Truncation}\label{sec:truncation-experiment}
In this section, we demonstrate the need for truncation by showing that the Q-tensors can increase over time and obtain values larger than the given truncation. In particular, this will show that truncation is indeed necessary for the elliptic equation \eqref{seq:elliptic} to be solvable for all time steps. Indeed, due to our choice of truncation, when the entries of the Q-tensor are below $R/2$, we have $\T_R(Q)\approx Q$, and the scheme is like solving a scheme with no truncation. Then if the Q-tensor entries surpass $R/2$, we see that a scheme without truncation allows the Q-tensor entries to increase past the point at which the scheme is solvable. First we set 
\[
    L = 0.001,\quad \varepsilon_1 = 2.5,\quad \varepsilon_2 = 0.5,\quad \varepsilon_3 = 5\times10^{-5},\quad R = 2,
\]
and we set the electric field boundary condition
\[
    \tilde g(t, x, y) = 0.7 t \cdot (x+0.5)\sin(\pi y).
\]
We also take $T=10$ and $\Delta t = 0.01$. We repeat this with $T=100$. We chose the parameter $\varepsilon_3$ to be $\varepsilon_3 = L/20$ so that the condition of Theorem~\ref{thm:existence_of_solutions} is satisfied. In a second experiment, we take
\[
    L = 0.001,\quad \varepsilon_1 = 1.0,\quad \varepsilon_2 = 0.25,\quad \varepsilon_3 = 5\times10^{-5},\quad R = 3,
\]
and we keep the electric field boundary condition the same. We change the final time to $T=20$ and keep $\Delta t = 0.01$. The results for these two experiments are shown in Figure~\ref{fig:truncation-experiment}. The plots in this figure show the time evolution of the maximum magnitude eigenvalue in space, the maximum magnitude Q-tensor entry, and the value $R/2$ at which the Q-tensor entries are truncated. In each case, we see that the maximum eigenvalue and the maximum Q-tensor entry surpass the truncation level and continue to increase. We note that $R=2$ corresponds roughly to the physically allowable Q-tensors, so anything larger is non-physical. Thus, we see that as time increases, we obtain less and less physical Q-tensors. Furthermore, in Figure~\ref{fig:truncation-experiment}(B), the maximum Q-tensor entry and maximum eigenvalue each become larger than $\varepsilon_1 / (d\varepsilon_2) = 2.5$, which is the value below which the entries of $Q$ must remain for the existence of solutions. Thus, the Q-tensor entries become so large that the truncation is actually important for solvability of the scheme.

\begin{figure}
    \centering
    \subfloat[$R=2$, $\varepsilon_1 = 2.5$, $\varepsilon_2 = 0.5$]{\includegraphics[scale=0.5]{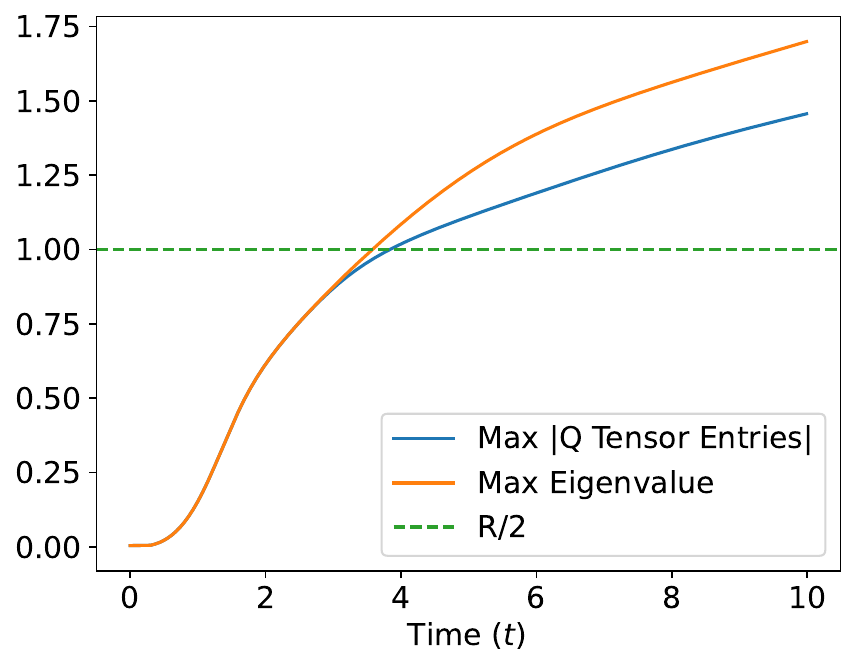}}
    \subfloat[$R=2$, $\varepsilon_1 = 2.5$, $\varepsilon_2 = 0.5$]{\includegraphics[scale=0.5]{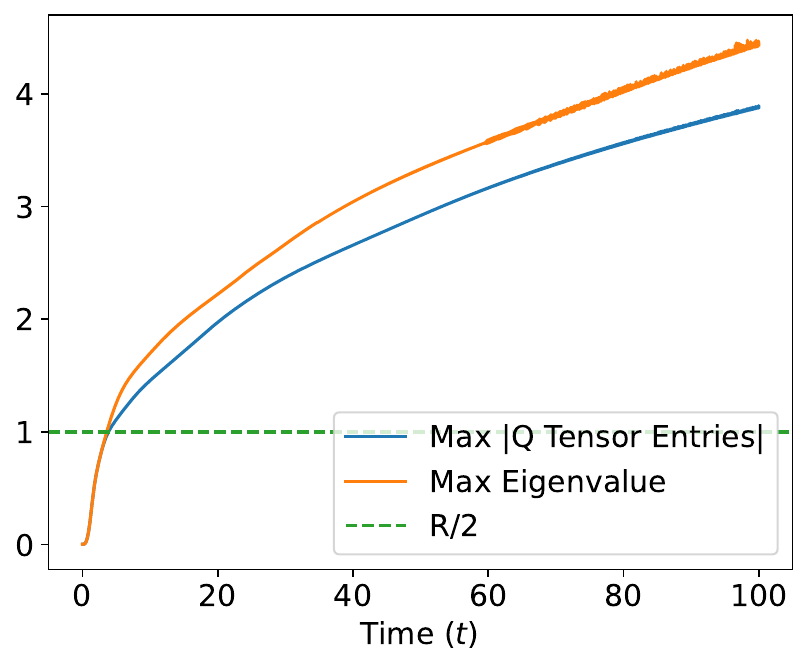}}\\
    \subfloat[$R=3$, $\varepsilon_1 = 1.0$, $\varepsilon_2 = 0.25$]{\includegraphics[scale=0.5]{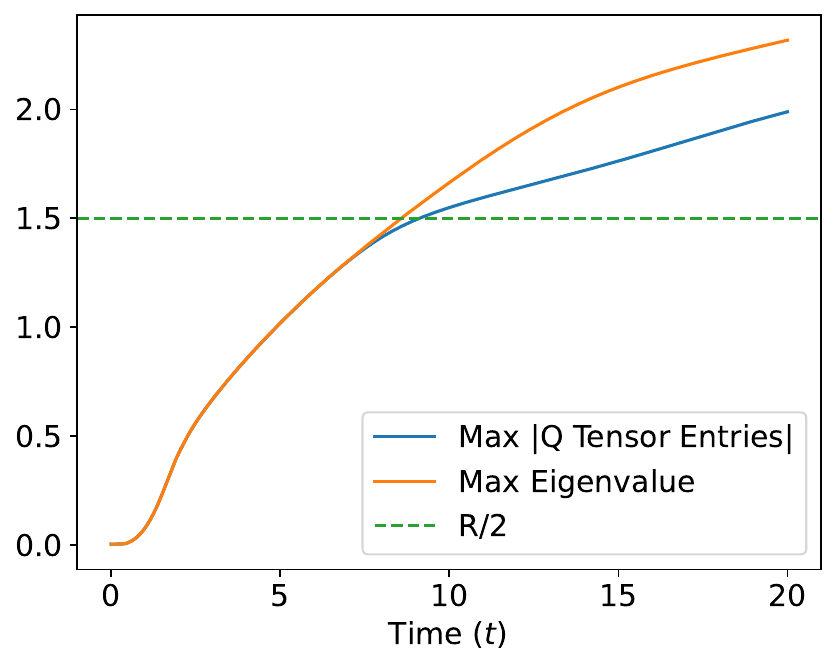}}
    \caption{Maximum absolute value of Q-tensor entries (in blue) and maximum eigenvalue of solutions (in orange) over time for the truncation experiment in Section~\ref{sec:truncation-experiment}. The dashed green line gives the value $R/d$ at which Q-tensor entries are truncated.}
    \label{fig:truncation-experiment}
\end{figure}

\subsection{Fr\'{e}edericksz transition}\label{sec:freedericksz-transition}
We now demonstrate a phenomenon in liquid crystals known as the Fr\'{e}edericksz transition. This transition is a change in behavior of the liquid crystal directors in which the directors align according to the anchoring of the boundary conditions for sufficiently small electric fields, and the directors align according to the electric field for sufficiently large electric fields~\cite{Freedericksz1927}. We can think of this transition as a sort of competition between the boundary conditions and the electric field. 

Here we will use the same parameters as in the previous experiments, and take
\[
    \tilde g(t, x, y) = \begin{cases}
        0 &\text{if } t < 0.5\\
        10x/3 &\text{if } 0.5 \le t < 1\\
        20x/3 &\text{if } 1\le t< 1.5\\
        10x &\text{if } t \ge 1.5,
    \end{cases}
\]
and
\[
    Q_0(x, y) = \boldsymbol{\textrm{d}}\boldsymbol{\textrm{d}}^\top - \frac12\operatorname{tr}(\boldsymbol{\textrm{d}}\boldsymbol{\textrm{d}}^\top)I,\quad\text{where}\quad \boldsymbol{\textrm{d}} = (0, 1).
\]
We take $Q$ to be $\begin{bmatrix} -0.5 & 0\\0 & 0.5\end{bmatrix}$ on the boundary so that the initial condition is continuous at $\partial\Omega$.

We see with the definition of $\tilde g$ that the electric field becomes stronger at times $t=0.5$, $t=1$, and $t=1.5$. The results of our simulation are shown in Figure~\ref{fig:experiment_3_solution}. We see that for $t\le 1.5$, the directors are all vertical, which is the orientation associated with the boundary condition. Once the electric potential is larger at time $t=1.5$, we see that the directors change so that a circular region in the middle of the domain has directors oriented horizontally, which is parallel to the gradient of the electric potential. Thus we observe that for large electric potentials, the directors align with the electric field, but for small electric potentials, the liquid crystals orient according to the boundary condition. We investigate this further in Figure~\ref{fig:experiment_3_transition}.

To produce this figure, we choose the extension of the electric potential boundary to be $s\tilde g(t,x,y)/10$, where $s$ is the parameter called ``Electric Field Strength'' on the axes. The plot gives the average absolute value of the director angles at time $t=2$ using this boundary extension. We see that there is a value of the strength $s$ below which the directors align vertically, that is, the same direction as the boundary condition for $Q$. For larger values of the strength $s$, we see that the directors begin to align in directions other than the vertical direction. Thus, our model reproduces the Fr\'{e}edericksz transition.

\begin{figure}
    \centering
    \subfloat[$t=0$]{\includegraphics[width=0.33\textwidth]{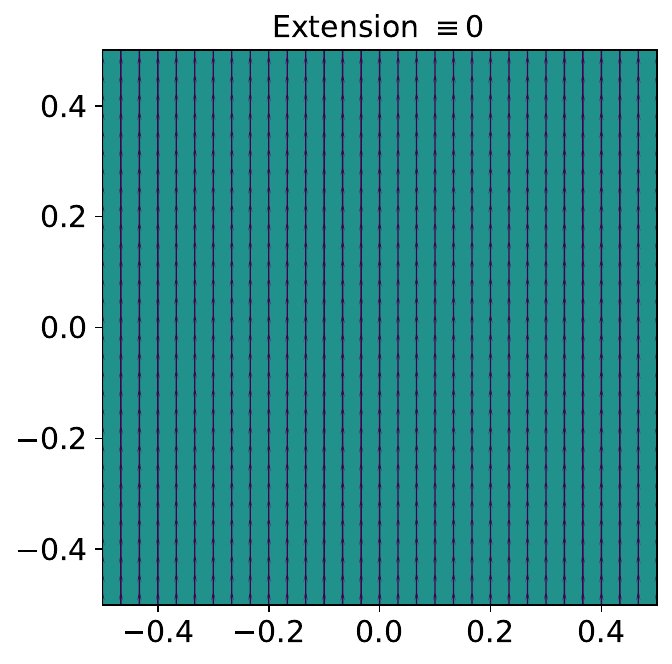}}
    \subfloat[$t=0.5$]{\includegraphics[width=0.33\textwidth]{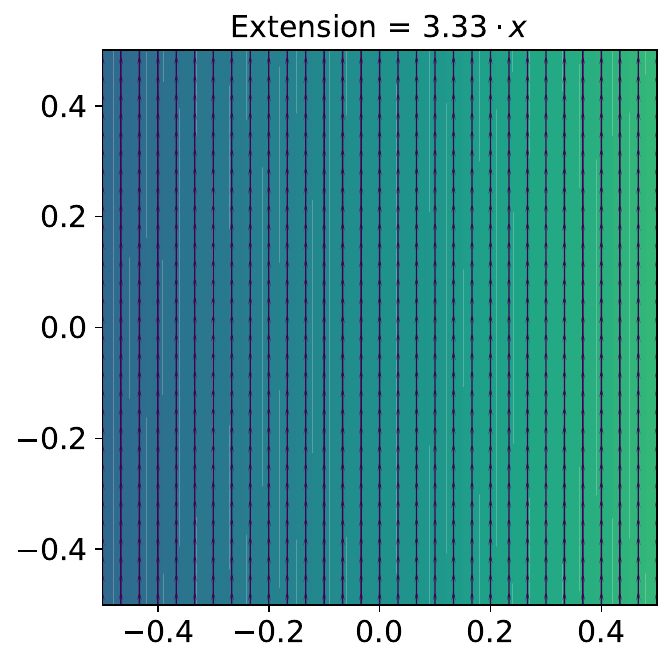}}
    \subfloat[$t=1$]{\includegraphics[width=0.33\textwidth]{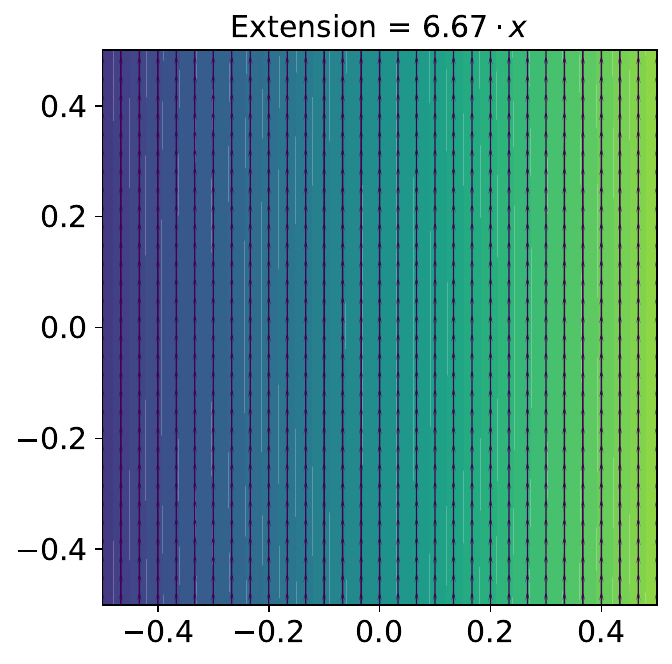}}
    \qquad 
    \subfloat[$t=1.5$]{\includegraphics[width=0.33\textwidth]{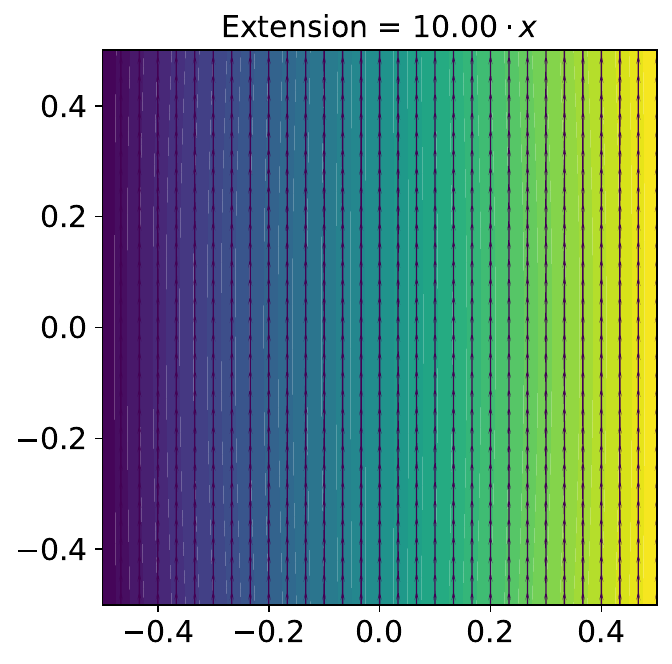}}
    \subfloat[$t=2$]{\includegraphics[width=0.3925\textwidth]{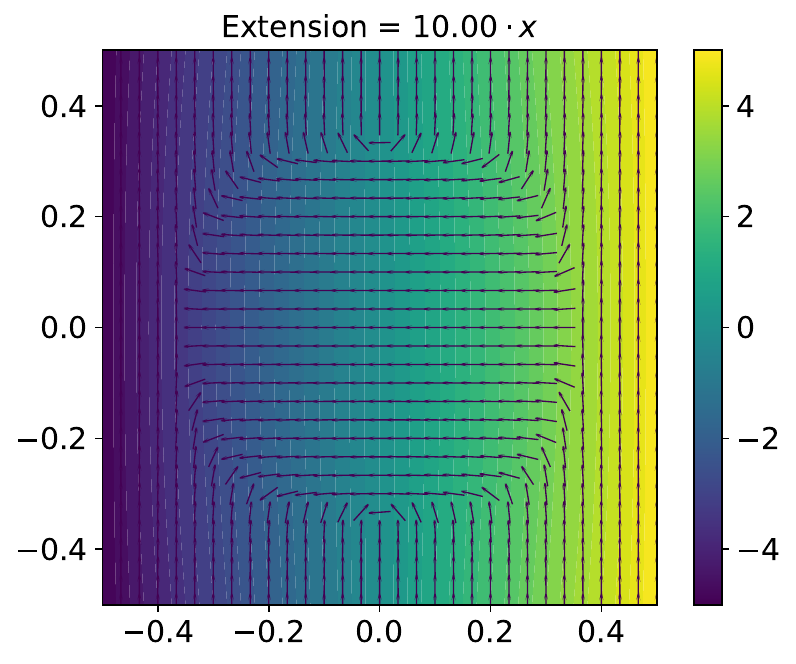}}
    \caption{Director field imposed on colored contour plot of electric potential $u$ for the Fr\'{e}edericksz transition experiment in Section~\ref{sec:freedericksz-transition}.}
    \label{fig:experiment_3_solution}
\end{figure}

\begin{figure}
    \centering
    \includegraphics[scale=0.6]{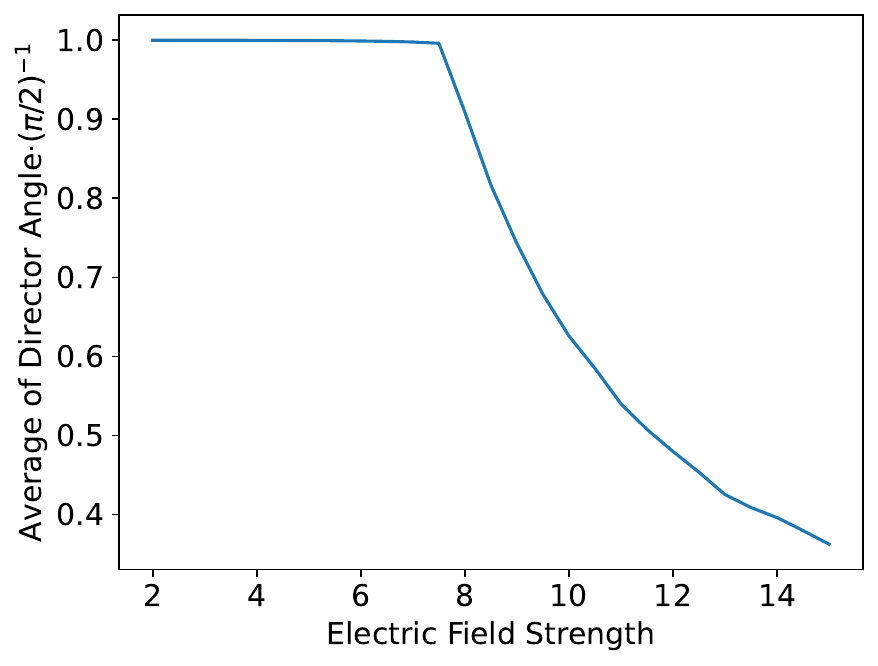}
    \caption{Scaled average director angles at final time for varying electric field strength for Fr\'{e}edericksz transition in Section~\ref{sec:freedericksz-transition}.}
    \label{fig:experiment_3_transition}
\end{figure}

\subsection{Convergence}
We now show the numerical convergence rates for our scheme. We take the same parameters, initial conditions, and boundary conditions as in Section \ref{sec:constant-initial-director}. The errors between the approximations $Q_h$ and $u_h$ compared to the reference solutions $Q^\ast$ and $u^\ast$ are measured in the $H^1$ norms defined by
\begin{align*}
    \mathcal{E}_Q(Q_h, Q^\ast)^2 &= \int_\Omega \left(|(Q_h)_{11} - Q^\ast_{11}|^2 + |(Q_h)_{12} - Q^\ast_{12}|^2 + |\nabla (Q_h)_{11} - \nabla Q^\ast_{11}|^2 + |\nabla (Q_h)_{12} - \nabla Q^\ast_{12}|^2\right)\, dx,\\ 
    \mathcal{E}_u(u_h, u^\ast)^2 &= \int_\Omega \left(|u_h - u^\ast|^2 + |\nabla u_h - \nabla u^\ast|^2\right)\, dx.
\end{align*}
For spatial convergence, we set $T = 0.4$ and compute a reference solution with $N = 1000$ time steps and $h = \frac{1}{300}$. We compare solutions using $N = 100$ time steps and $h \in \{\frac{1}{10}, \frac{1}{20}, \frac{1}{35}, \frac{1}{50}\}$. The convergence rates are given in Figure \ref{fig:spatial-convergence}. As expected, the convergence rate is around 1.

\begin{figure}[H]
\centering
\begin{tabular}{|c|c|c|c|c|}
    \hline
    $h$ & $\mathcal{E}_Q(Q_h, Q^\ast)$ & $Q_h$ rate & $\mathcal{E}_u(u_h, u^\ast)$ & $u_h$ rate \\
    \hline
    0.1 & $2.594\times 10^{-1}$ & NaN & $7.166\times 10^{-1}$ & NaN\\
    \hline
    0.05 & $1.326\times 10^{-1}$ & 0.968 & $3.627\times 10^{-1}$ & 0.982\\
    \hline
    0.0286 & $6.987 \times 10^{-2}$ & 1.145 & $1.851\times 10^{-1}$ & 1.202\\
    \hline
    0.02 & $5.602\times 10^{-2}$ & 0.619 & $1.441\times 10^{-1}$ & 0.703\\
    \hline
\end{tabular}
\caption{Errors and convergence rates for spatial discretization.}
\label{fig:spatial-convergence}
\end{figure}

For convergence in time, we set $T = 0.2$ and compute a reference solution with $N = 8000$ time steps and $h = 0.01$. We compare solutions using $h = 0.01$ and $N \in \{40, 80, 160, 320, 640, 1280\}$ time steps. The convergence rates are given in Figure \ref{fig:time-convergence}. Again, we observe a rate of around 1.

\begin{figure}[H]
\centering
\begin{tabular}{|c|c|c|c|c|}
    \hline
    $\Delta t$ & $\mathcal{E}_Q(Q_h, Q^\ast)$ & $Q_h$ rate & $\mathcal{E}_u(u_h, u^\ast)$ & $u_h$ rate \\
    \hline
    $5\times 10^{-3}$ & $5.015\times 10^{-4}$ & NaN & $3.331\times 10^{-5}$ & NaN\\
    \hline
    $2.5\times 10^{-3}$ & $1.703\times 10^{-4}$ & $1.558$ & $1.646\times 10^{-5}$ & $1.017$\\
    \hline
    $1.25\times 10^{-3}$ & $7.869\times 10^{-5}$ & $1.113$ & $8.162\times 10^{-6}$ & $1.012$\\
    \hline
    $6.25\times 10^{-4}$ & $1.461\times 10^{-5}$ & $2.430$ & $1.516\times 10^{-6}$ & $2.428$\\
    \hline
    $3.125\times 10^{-4}$ & $7.026\times 10^{-6}$ & $1.056$ & $7.293\times 10^{-7}$ & $1.056$\\
    \hline
    $1.5625\times 10^{-4}$ & $3.213\times 10^{-6}$ & $1.129$ & $3.336\times 10^{-7}$ & $1.128$\\
    \hline
\end{tabular}
\caption{Errors and convergence rates for time discretization.}
\label{fig:time-convergence}
\end{figure}

\section{Conclusions}
In this article, we presented a Q-tensor model for a liquid crystal director field under the influence of an electric field and adapted it, by using a truncation operator, to make it well-posed. We then designed an energy-stable fully-discrete finite element discretization for this model using convex splitting, and proved that the method is uniformly stable with respect to the discretization parameters within a reasonable parameter range. We showed that the implicit scheme is well-posed without any restrictions on the size of the time step by leveraging the Leray-Schauder fixed point theorem. Using the a priori estimates coming from the discrete energy balance and a special identity for the elliptic equation, we prove convergence of the approximations computed by the scheme to a weak solution of the system of PDEs~\eqref{eq:system_strong_formulation} for vanishing polarization. Finally, we present numerical experiments demonstrating the scheme's capabilities to capture realistic liquid crystal dynamics such as the Fr\'{e}edericksz transition. We also show numerically that the truncation operator is necessary as the $Q$-tensor entries grow unboundedly for certain choices of applied electric potential, which would make the elliptic equation ill-posed. To the best of our knowledge, this is the first numerical scheme for this model that has been shown to be stable and convergent.

\appendix
\section{Uniform energy stability of the scheme}\label{app:energystability}
Here we prove the uniform discrete energy stability of the scheme, Theorem~\ref{thm:fully_discrete_energy_stability}.

\begin{proof}
	Averaging~\eqref{eq:ufullydiscrete} at times $n$ and $n+1$ gives
	\begin{equation}
		\begin{split}
			\int_\Omega (\varepsilon_1\nabla u_h^{n+1/2} + \frac{\varepsilon_2}{2} (\T_R(Q^{n+1}_h)\nabla u_h^{n+1} + \T_R(Q_h^n)\nabla u_h^{n}) + \varepsilon_3\operatorname{div}(Q^{n+1/2}_h))\cdot\nabla\psi_h\, dx = 0.
		\end{split}
	\end{equation}
	Recall that $\tilde u_h^n = u_h^n - \tilde g_h^n$. Then $\tilde u^n_h = 0$ on $\partial \Omega\times[0,T]$. Take $\psi_h = D_t^+ \tilde u^n_h$ to get
	\begin{align*}
		0 &= \int_\Omega \left(\varepsilon_1 \nabla u^{n+1/2}_h\cdot D_t^+ \nabla u_h^n - \varepsilon_1\nabla u_h^{n+1/2}\cdot D_t^+\nabla\tilde g^n_h + \frac{\varepsilon_2}{2} (\T_R(Q_h^{n+1})\nabla u_h^{n+1} + \T_R(Q_h^n)\nabla u^{n}_h)\cdot D_t^+\nabla u_h^n\right.\\ 
		&\hspace{10ex}- \frac{\varepsilon_2}{2} D_t^+(\nabla \tilde g^n_h)\cdot (\T_R(Q^{n+1}_h)\nabla u_h^{n+1} + \T_R(Q_h^n)\nabla u_h^{n})\\ 
		&\hspace{15ex}\Big.+ \varepsilon_3\operatorname{div}(Q_h^{n+1/2}) \cdot D_t^+ \nabla u^n_h - \varepsilon_3\operatorname{div}(Q_h^{n+1/2})\cdot D_t^+\nabla\tilde g^n_h\Big)\, dx
	\end{align*}
	so that
	\begin{multline}
		\label{eq:discrete_div_integrated}
		\int_\Omega \left(\frac{\varepsilon_1}{2} D_t^+|\nabla u^n_h|^2 + \frac{\varepsilon_2}{2} (\T_R(Q_h^{n+1})\nabla u_h^{n+1} + \T_R(Q_h^n)\nabla u^{n}_h)\cdot D_t^+\nabla u_h^n + \varepsilon_3\operatorname{div}(Q_h^{n+1/2})\cdot D_t^+\nabla u_h^n\right)\, dx\\
		= \int_\Omega \left(\varepsilon_1\nabla u_h^{n+1/2}\cdot D_t^+\nabla \tilde g^n_h + \frac{\varepsilon_2}{2} (\T_R(Q_h^{n+1})\nabla u_h^{n+1} + \T_R(Q^n_h)\nabla u_h^{n})\cdot D_t^+\nabla\tilde g_h^n + \varepsilon_3\operatorname{div}(Q_h^{n+1/2})\cdot D_t^+\nabla \tilde g^n_h\right)\, dx.
	\end{multline}
	Now take $\Phi_h = D_t^+ Q^n_h$, which is zero on the boundary, as a test function in~\eqref{eq:Qfullydiscrete} and note that $(D_t^+ Q^n_h)_S = D_t^+ Q^n_h$ since $Q^n_h$ and $Q_h^{n+1}$ are symmetric. Symmetry also implies that $(\tilde\P^n_h\odot \nabla u_h^n(\nabla u_h^{n+1})^\top)_S : D_t^+ Q_h^n = \tilde\P^n_h\odot \nabla u_h^n(\nabla u_h^{n+1})^\top : D_t^+ Q^n_h$. With convexity of $\mathcal F_1, \mathcal F_2$ and using that $D_t^+ Q^n_h$ is trace-free, we obtain
	\begin{equation}
		\begin{split}
			&\int_\Omega |D_t^+ Q^n_h|^2\, dx + D_t^+ \int_\Omega \left(\frac{L}{2}|\nabla Q^n_h|^2 + \mathcal F_B(Q_h^{n})\right)\, dx\\
			&\le \frac{\varepsilon_2}{2} \int_\Omega \tilde\P_h^n\odot \nabla u_h^n(\nabla u_h^{n+1})^\top: D_t^+ Q_h^n\, dx + \varepsilon_3\int_\Omega \nabla u_h^{n+1/2} : \operatorname{div} D_t^+ Q_h^n\, dx.
		\end{split}
	\end{equation}
	Combining this with equation (\ref{eq:discrete_div_integrated}) gives
	\begin{align*}
		&\int_\Omega |D_t^+ Q^n_h|^2\, dx + D_t^+ \int_\Omega \left(\frac{L}{2}|\nabla Q^n_h|^2 + \mathcal F_B(Q_h^{n})\right)\, dx\\
		&\le \frac{\varepsilon_2}{2} \int_\Omega \tilde\P^n_h\odot \nabla u_h^n(\nabla u_h^{n+1})^\top: D_t^+ Q_h^n\, dx + \varepsilon_3\int_\Omega \nabla u_h^{n+1/2} : \operatorname{div} D_t^+ Q_h^n\, dx\\
		&+ \int_\Omega \left(\frac{\varepsilon_1}{2} D_t^+|\nabla u_h^n|^2 + \frac{\varepsilon_2}{2} (\T_R(Q_h^{n+1})\nabla u_h^{n+1} + \T_R(Q_h^n)\nabla u_h^{n})\cdot D_t^+\nabla u_h^n + \varepsilon_3\operatorname{div}(Q_h^{n+1/2})\cdot D_t^+\nabla u_h^n\right)\, dx\\
		&- \int_\Omega \left(\varepsilon_1\nabla u_h^{n+1/2}\cdot D_t^+\nabla \tilde g^n_h + \frac{\varepsilon_2}{2} (\T_R(Q_h^{n+1})\nabla u_h^{n+1} + \T_R(Q_h^n)\nabla u^{n}_h)\cdot D_t^+\nabla\tilde g^n_h + \varepsilon_3\operatorname{div}(Q_h^{n+1/2})\cdot D_t^+\nabla \tilde g_h^n\right)\, dx.
	\end{align*}
	And we see, using the facts that $\tilde\P_h^n$, $\T_R(Q_h^n)$, and $\T_R(Q_h^{n+1})$ are symmetric and that $A:bc^\top = A:cb^\top$ when $A$ is symmetric, that
	\begin{align*}
		&\tilde\P^n_h\odot \nabla u_h^n(\nabla u_h^{n+1})^\top : D_t^+ Q^n_h + (\T_R(Q^{n+1}_h)\nabla u_h^{n+1} + \T_R(Q_h^n)\nabla u_h^{n})\cdot D_t^+\nabla u_h^n\\
		&\hspace{10ex}= D_t^+ \T_R(Q_h^n): \nabla u_h^{n}(\nabla u_h^{n+1})^\top + (\T_R(Q_h^{n+1})\nabla u^{n+1}_h + \T_R(Q_h^n)\nabla u_h^{n})\cdot D_t^+\nabla u^n_h\\
		&\hspace{10ex}= D_t^+ (\T_R(Q_h^n):\nabla u_h^n(\nabla u_h^n)^\top).
	\end{align*}
	We also have
	\begin{equation*}
		\nabla u^{n+1/2}_h\cdot \operatorname{div} D_t^+ Q^n_h + \operatorname{div}(Q^{n+1/2}_h)\cdot D_t^+\nabla u^n_h
		= D_t^+ \left(\operatorname{div}(Q^n_h)\cdot\nabla u^n_h\right).
	\end{equation*}
	Thus, we have
	\begin{equation}
		\label{eq:pre_energy_balance_semidiscrete}
		\begin{split}
			&\int_\Omega |D_t^+ Q^n_h|^2\, dx + D_t^+ \int_\Omega \left(\frac{L}{2}|\nabla Q^n_h|^2 + \mathcal F_B(Q^{n}_h)\right)\, dx\\
			&\le \int_\Omega \frac{\varepsilon_2}{2} D_t^+(\T_R(Q_h^n):\nabla u_h^n(\nabla u_h^n)^\top)\, dx + \int_\Omega \varepsilon_3 D_t^+(\operatorname{div}(Q^n_h)\cdot\nabla u^n_h)\, dx + \int_\Omega \frac{\varepsilon_1}{2} D_t^+|\nabla u^n_h|^2\, dx\\
			&- \int_\Omega \left(\varepsilon_1\nabla u_h^{n+1/2}\cdot D_t^+\nabla \tilde g^n_h + \frac{\varepsilon_2}{2} (\T_R(Q_h^{n+1})\nabla u_h^{n+1} + \T_R(Q_h^n)\nabla u^{n}_h)\cdot D_t^+\nabla\tilde g^n_h + \varepsilon_3\operatorname{div}(Q_h^{n+1/2})\cdot D_t^+\nabla \tilde g^n_h\right)\, dx.
		\end{split}
	\end{equation}
	Now note that we have
	\[
	D_t^+ \int_\Omega (\varepsilon_1\nabla u^{n}_h + \varepsilon_2 \T_R(Q_h^n)\nabla u^{n}_h + \varepsilon_3\operatorname{div}(Q_h^{n}))\cdot\nabla\psi_h\, dx = 0,
	\]
	so because $u^{n}_h - \tilde g^{n}_h = 0$ on $\partial \Omega$, we take $\psi_h = u^{n}_h-\tilde g^{n}_h$. This yields
	\begin{equation}
		\begin{split}
			&D_t^+ \int_\Omega (\varepsilon_1 |\nabla u^{n}_h|^2 + \varepsilon_2(\nabla u^{n}_h)^\top \T_R(Q_h^n)\nabla u_h^{n} + \varepsilon_3\operatorname{div}(Q_h^{n})\cdot\nabla u_h^{n})\, dx\\ 
			&= D_t^+ \int_\Omega (\varepsilon_1 \nabla u_h^{n}\cdot\nabla\tilde g^{n}_h + \varepsilon_2 (\nabla \tilde g^{n}_h)^\top \T_R(Q_h^n)\nabla u_h^{n} + \varepsilon_3\operatorname{div}(Q_h^{n})\cdot\nabla\tilde g_h^{n})\, dx.
		\end{split}
	\end{equation}
	Multiplying this by $\Delta t$, adding to $\Delta t$ times (\ref{eq:pre_energy_balance_semidiscrete}), then summing over $n=0,1,\dots,N-1$ yields
	\begin{equation}
		\begin{split}
			&\Delta t\sum_{n=0}^{N-1} \int_\Omega |D_t^+ Q^n_h|^2\, dx + \int_\Omega \left(\frac{L}{2}|\nabla Q_h^N|^2 + \mathcal F_B(Q_h^{N}) + \frac{\varepsilon_1}{2}|\nabla u_h^N|^2 + \frac{\varepsilon_2}{2} (\nabla u^N_h)^\top \T_R(Q_h^N)\nabla u_h^N\right)\, dx\\
			&\le \int_\Omega \left(\frac{L}{2}|\nabla Q_h^0|^2 + \mathcal F_B(Q_h^{0}) + \frac{\varepsilon_1}{2}|\nabla u_h^0|^2 + \frac{\varepsilon_2}{2} (\nabla u^0_h)^\top \T_R(Q_h^0)\nabla u_h^0\right)\, dx\\
			&- \Delta t \sum_{n=0}^{N-1}\!\!\int_\Omega \!\left(\varepsilon_1\nabla u^{n+1/2}_h\cdot D_t^+\nabla \tilde g_h^n + \frac{\varepsilon_2}{2} (\T_R(Q^{n+1}_h)\nabla u^{n+1}_h + \T_R(Q_h^n)\nabla u_h^{n})\cdot D_t^+\nabla\tilde g^n_h + \varepsilon_3\operatorname{div}(Q_h^{n+1/2})\cdot D_t^+\nabla \tilde g_h^n\right)\, dx\\
			&- \int_\Omega (\varepsilon_1 \nabla u_h^{0}\cdot\nabla\tilde g_h^{0} + \varepsilon_2 (\nabla \tilde g_h^{0})^\top \T_R(Q_h^0)\nabla u_h^{0} + \varepsilon_3\operatorname{div}(Q_h^{0})\cdot\nabla\tilde g_h^{0})\, dx\\
			&+ \int_\Omega (\varepsilon_1 \nabla u_h^{N}\cdot\nabla\tilde g_h^{N} + \varepsilon_2 (\nabla \tilde g_h^{N})^\top \T_R(Q_h^N)\nabla u_h^{N} + \varepsilon_3\operatorname{div}(Q_h^{N})\cdot\nabla\tilde g_h^{N})\, dx\\
			&:= E_0 + A + B + D.
		\end{split}
	\end{equation}
	To bound $A$, we apply the Cauchy-Schwarz inequality then Young's inequality to obtain
	\begin{align*}
		|A| &\le \Delta t  \sum_{n=0}^{N-1}\left(\varepsilon_1\|\nabla u_h^{n+1/2}\|_{L^2} + \frac{|\varepsilon_2|}{2}\|\mathcal{T}_R(Q_h^{n+1})\nabla u^{n+1}_h\|_{L^2} + \frac{|\varepsilon_2|}{2}\|\mathcal{T}_R(Q_h^{n})\nabla u^{n}_h\|_{L^2} \right)\|D_t^+\nabla \tilde g_h^n\|_{L^2}\\
		&\hspace{3ex}+ \Delta t  \sum_{n=0}^{N-1}|\varepsilon_3|\|\operatorname{div}Q_h^{n+1/2}\|_{L^2}\|D_t^+\nabla \tilde g_h^n\|_{L^2}\\
		&\le \frac{\Delta t  (\varepsilon_1 + |\varepsilon_2| + |\varepsilon_3|)}{2} \sum_{n=0}^{N-1} \|D_t^+\nabla\tilde g^n_h\|_{L^2}^2 + \frac{\Delta t \varepsilon_1}{2}\sum_{n=0}^{N-1} \|\nabla u_h^{n+1/2}\|_{L^2}^2\\
		&\hspace{3ex}+ \frac{\Delta t |\varepsilon_2|}{4} \sum_{n=0}^{N-1} \|\mathcal{T}_R(Q^{n+1}_h)\nabla u^{n+1}_h\|_{L^2}^2 + \frac{\Delta t |\varepsilon_2|}{4} \sum_{n=0}^{N-1} \|\mathcal{T}_R(Q_h^{n})\nabla u^{n}_h\|_{L^2}^2 + \frac{\Delta t |\varepsilon_3|}{2}\sum_{n=0}^{N-1} \|\operatorname{div}Q^{n+1/2}_h\|_{L^2}^2.
	\end{align*}
	Then applying the triangle inequality and Young's inequality, we have
	\begin{equation}
			\frac{\Delta t \varepsilon_1}{2}\sum_{n=0}^{N-1} \|\nabla u_h^{n+1/2}\|_{L^2}^2 
			 \le \frac{\Delta t \varepsilon_1}{2}\sum_{n=0}^{N} \|\nabla u^n_h\|_{L^2}^2,\quad \text{and}\quad \frac{\Delta t |\varepsilon_3|}{2}\sum_{n=0}^{N-1} \|\operatorname{div}Q^{n+1/2}_h\|_{L^2}^2 
			 \le \frac{d\Delta t  |\varepsilon_3|}{2} \sum_{n=0}^N \|\nabla Q_h^n\|_{L^2}^2.
	\end{equation}
	We also have
	\begin{equation}
		\frac{\Delta t |\varepsilon_2|}{4} \sum_{n=0}^{N-1}\left( \|\mathcal{T}_R(Q_h^{n+1})\nabla u_h^{n+1}\|_{L^2}^2  + \|\mathcal{T}_R(Q_h^{n})\nabla u^{n}_h\|_{L^2}^2  \right)\le \frac{\Delta t |\varepsilon_2|R^2}{4}\sum_{n=0}^{N-1}\left( \|\nabla u_h^{n+1}\|_{L^2}^2+\|\nabla u_h^{n}\|_{L^2}^2\right).
	\end{equation}
	Combining these estimates for $A$, we obtain
	\begin{equation}
		\begin{split}
			|A| &\le \frac{\Delta t  (\varepsilon_1+|\varepsilon_2|+|\varepsilon_3|)}{2}\sum_{n=0}^{N-1}\|D_t^+\nabla\tilde g^n_h\|_{L^2}^2 + \frac{\Delta t (\varepsilon_1+R^2|\varepsilon_2|)}{2}\sum_{n=0}^N\|\nabla u^n_h\|_{L^2}^2 + \frac{d\Delta t |\varepsilon_3|}{2}\sum_{n=0}^N \|\nabla Q^n_h\|_{L^2}^2.
		\end{split}
	\end{equation}
	The term $B$ only depends on the initial conditions, so we leave it as it is. We continue to term $D$. By the Cauchy-Schwarz inequality and Young's inequality, we have
	\[
	\left|\int_\Omega \varepsilon_1 \nabla u_h^N\cdot\nabla\tilde g_h^N\, dx\right| \le \frac{\varepsilon_1}{2}\left(\delta_1\|\nabla u_h^N\|_{L^2}^2 + \frac{1}{\delta_1}\|\nabla \tilde g_h^N\|_{L^2}^2\right).
	\]
	Also,
	\begin{equation*}
		\left|\int_\Omega \varepsilon_2 (\nabla \tilde g^N_h)^\top \T_R(Q_h^N)\nabla u_h^N \, dx\right| 
		\le \frac{|\varepsilon_2|}{2}(\delta_2R^2\|\nabla u_h^N\|_{L^2}^2 + \frac{1}{\delta_2}\|\nabla \tilde g^N_h\|_{L^2}^2).
	\end{equation*}
	Now by the Cauchy-Schwarz inequality and Young's inequality, we have
	\begin{equation*}
		\left|\int_\Omega \varepsilon_3 \operatorname{div}(Q^N_h)\cdot\nabla\tilde g_h^N\, dx\right|
		\le \frac{\sqrt{d}|\varepsilon_3|}{2}\left(\delta_3\|\nabla Q^N_h\|_{L^2}^2 + \frac{1}{\delta_3}\|\nabla \tilde g^N_h\|_{L^2}^2\right).
	\end{equation*}
	Now combining all of these bounds and noting that $\lvert(\nabla u^N_h)^\top \T_R(Q^N_h)\nabla u_h^N\rvert \le R\lvert\nabla u_h^N\rvert_2^2$, we have
	\begin{equation}
		\label{eq:L2_bound_on_derivative_Qn}
		\begin{split}
			&\Delta t \sum_{n=0}^{N-1}\int_\Omega \lvert D_t^+ Q_h^n\rvert^2\, dx + \int_\Omega \mathcal{F}_B(Q_h^N)\, dx + \frac{1}{2}\left(L - d\Delta t |\varepsilon_3| - \sqrt{d}|\varepsilon_3|\delta_3\right)\int_\Omega \lvert \nabla Q_h^N\rvert^2\, dx\\
			&\hspace{10ex}+ \frac{1}{2}\left(\varepsilon_1 - |\varepsilon_2|R - \Delta t \varepsilon_1 - \Delta t |\varepsilon_2|R^2 - \varepsilon_1\delta_1 - |\varepsilon_2|\delta_2 R^2\right) \int_\Omega\lvert \nabla u^N_h\rvert^2\, dx\\
			&\le \left(\frac{\Delta t \varepsilon_1}{2} + \frac{\Delta t |\varepsilon_2|R^2}{2}\right)\sum_{n=0}^{N-1} \|\nabla u_h^n\|_{L^2}^2 + \frac{d\Delta t |\varepsilon_3|}{2}\sum_{n=0}^{N-1}\|\nabla Q_h^n\|_{L^2}^2\\
			&\hspace{10ex}+ \int_\Omega \left(\frac{L}{2}\lvert \nabla Q^0_h\rvert^2 + \mathcal{F}_B(Q^0_h) + \frac{\varepsilon_1}{2}\lvert\nabla u^0_h\rvert^2 + \frac{|\varepsilon_2|R}{2}\lvert \nabla u^0_h\rvert^2\right)\, dx\\
			&\hspace{10ex}+ \frac{\Delta t  (\varepsilon_1+|\varepsilon_2|+|\varepsilon_3|)}{2}\sum_{n=0}^{N-1}\|D_t^+ \nabla \tilde g^n_h\|_{L^2}^2\\
			&\hspace{10ex}+ \left\lvert \int_\Omega \left(\varepsilon_1\nabla u^0_h\cdot\nabla\tilde g^0_h + \varepsilon_2 (\nabla\tilde g_h^0)^\top \T_R(Q^0)\nabla u^0_h + \varepsilon_2 \operatorname{div}(Q^0_h)\cdot\nabla \tilde g^0_h\right)\, dx \right\rvert\\
			&\hspace{10ex}+ \frac{\varepsilon_1}{2\delta_1}\|\nabla\tilde g_h^N\|_{L^2}^2 + \frac{|\varepsilon_2|}{2\delta_2}\|\nabla\tilde g^N_h\|_{L^2}^2 + \frac{\sqrt{d} |\varepsilon_3|}{2\delta_3}\|\nabla \tilde g^N_h\|_{L^2}^2.
		\end{split}
	\end{equation}
		Now take $\delta_1 = \min\left\{\frac12, \frac{\varepsilon_1 - |\varepsilon_2|R}{6\varepsilon_1}\right\}$, $\delta_2 = \min\left\{\frac12, \frac{\varepsilon_1 - |\varepsilon_2|R}{6R^2|\varepsilon_2|}\right\}$, and 
	\[
	\delta_3 = \begin{cases} 
		L / (2\sqrt{d}|\varepsilon_3|) &\text{if } \varepsilon_3\ne 0,\\
		1 &\text{if } \varepsilon_3 = 0,
	\end{cases}
	\]
	and recall that we assume $\Delta t < \min\left\{\frac{L}{2d|\varepsilon_3|}, \frac{\varepsilon_1 - |\varepsilon_2|R}{3(\varepsilon_1 + |\varepsilon_2|R^2)}\right\}$. Note that with this choice, we have that the coefficient of $\| \nabla Q^N_h\|_{L^2}^2$ in the above expression is greater than 0.
	Similarly, the coefficient of $\|\nabla u^N_h\|_{L^2}^2$ in the above expression is greater than 0.
	Also note that we have $\delta_1^{-1} = \max\left\{2, \frac{6\varepsilon_1}{\varepsilon_1 - |\varepsilon_2|R}\right\}$ and $\delta_2^{-1} = \max\left\{2, \frac{6R^2|\varepsilon_2|}{\varepsilon_1 - |\varepsilon_2|R}\right\}$. Thus, we have an estimate of the form
	\[
	V(t^N) \le A(t^N) + \Delta t\sum_{n=0}^{N-1} B(t^n)V(t^n),
	\]
	where
	\[
	V(t^n) = \int_\Omega \mathcal{F}_B(Q^n_h)\, dx + \frac{1}{4}\left(L - 2d\Delta t |\varepsilon_3|\right)\|\nabla Q^n_h\|_{L_2}^2 + \frac{1}{6}\left(\varepsilon_1 - |\varepsilon_2|R\right)\|\nabla u_h^n\|_{L^2}^2,
	\]
	and
	\[
	B(t^n) = \max\left\{\frac{3(\varepsilon_1 + |\varepsilon_2|R^2)}{\varepsilon_1 - |\varepsilon_2|R}, \frac{2d|\varepsilon_3|}{L-2d\Delta t|\varepsilon_3|}\right\},
	\]
	and
	\begin{align*}
		A(t^N) &= \int_\Omega \left(\frac{L}{2}\lvert \nabla Q^0_h\rvert^2 + \mathcal{F}_B(Q_h^0) + \frac{\varepsilon_1}{2}\lvert\nabla u_h^0\rvert^2 + \frac{|\varepsilon_2|R}{2}\lvert \nabla u_h^0\rvert^2\right)\, dx\\
		&\hspace{10ex}+ \frac{\Delta t  (\varepsilon_1+|\varepsilon_2|+|\varepsilon_3|)}{2}\sum_{n=0}^{N-1}\|D_t^+ \nabla \tilde g^n_h\|_{L^2}^2\\
		&\hspace{10ex}+ \left\lvert \int_\Omega \left(\varepsilon_1\nabla u^0_h\cdot\nabla\tilde g_h^0 + \varepsilon_2 (\nabla\tilde g^0_h)^\top \T_R(Q_h^0)\nabla u^0_h + \varepsilon_2 \operatorname{div}(Q^0_h)\cdot\nabla \tilde g^0_h\right)\, dx \right\rvert\\
		&\hspace{10ex}+\left( \varepsilon_1 + \frac{3\varepsilon_1^2 + 3R^2|\varepsilon_2|^2}{\varepsilon_1 - |\varepsilon_2|R} + |\varepsilon_2| + \frac{d|\varepsilon_3|^2}{L} \right)\|\nabla\tilde g^N_h\|_{L^2}^2 . 
	\end{align*}
	Using the discrete Gr\"{o}nwall inequality, e.g.~\cite{Holte2009},
	\[
	V(t^N) \le A(t^N) + \Delta t\sum_{n=0}^{N-1} A(t^n)B(t^n)\operatorname{exp}\left(\Delta t \sum_{k=n+1}^{N-1} B(t^k)\right),
	\]
	we obtain
	\begin{align*}
		&\int_\Omega \mathcal{F}_B(Q^N_h)\, dx + \frac14\left(L - 2d\Delta t |\varepsilon_3|\right)\|\nabla Q_h^N\|_{L_2}^2 + \frac{1}{6}\left(\varepsilon_1 - |\varepsilon_2|R\right)\|\nabla u_h^N\|_{L^2}^2\\
		&\le A(t^N) + \Delta t\max\left\{\frac{3(\varepsilon_1 + |\varepsilon_2|R^2)}{\varepsilon_1 - |\varepsilon_2|R}, \frac{2d|\varepsilon_3|}{L-2d\Delta t|\varepsilon_3|}\right\}\\
		&\hspace{15ex}\times\sum_{n=0}^{N-1}A(t^n)\operatorname{exp}\left(\max\left\{\frac{3(\varepsilon_1 + |\varepsilon_2|R^2)}{\varepsilon_1 - |\varepsilon_2|R}, \frac{2d|\varepsilon_3|}{L-2d\Delta t|\varepsilon_3|}\right\}(N-1-n)\Delta t\right)\\
		&\le A(t^N) + C\exp(Ct^N)\Delta t\sum_{n=0}^{N-1} A(t^n).
	\end{align*}
	Thus we obtain the estimate
	\begin{align*}
		&\int_\Omega \mathcal{F}_B(Q^N_h)\, dx + \frac14\left(L - 2d\Delta t |\varepsilon_3|\right)\|\nabla Q_h^N\|_{L_2}^2 + \frac{1}{6}\left(\varepsilon_1 - |\varepsilon_2|R\right)\|\nabla u_h^N\|_{L^2}^2\\
		&\le C(1+t^N\operatorname{exp}(Ct^N))\left( \|\nabla Q_h^0\|_{L^2}^2 + \mathcal{F}_B(Q^0_h) + (1+R)\|\nabla u^0_h\|_{L^2}^2 + \|\nabla \tilde g^0\|_{L^2}^2\right)\\ 
		&+ C(1 + t^N\operatorname{exp}(Ct^N))\Delta t \sum_{n=0}^{N-1} \|D_t^+ \nabla\tilde g^n_h\|_{L^2}^2 + C\|\nabla\tilde g_h^N\|_{L^2}^2 + C\operatorname{exp}(Ct^N)\Delta t\sum_{n=0}^{N-1} \|\nabla \tilde g_h^n\|_{L^2}^2,
	\end{align*}
	which proves the result.
\end{proof}

\bibliographystyle{abbrv}
\bibliography{bibliography}
\end{document}